\documentclass[bj,numbers,preprint,noshowframe]{imsart}
\usepackage[utf8]{inputenc}
\usepackage[T1]{fontenc}
\usepackage[english]{babel}

\usepackage{amsfonts}
\usepackage{amssymb}
\usepackage{amsthm}
\usepackage{amsmath}
\usepackage{mathtools}
\usepackage{bbm}
\usepackage{shuffle}
\allowdisplaybreaks[1]

\usepackage{graphicx}
\usepackage{float}
\usepackage{caption}
\usepackage{subcaption}
\usepackage{makecell}
\usepackage{threeparttable}
\usepackage{multirow}

\usepackage{enumitem}
\usepackage{xcolor}
\usepackage{url}
\usepackage{xspace}
\usepackage{csquotes}

\makeatletter
\providecommand{\inits@size}[1]{}
\makeatother

\makeatletter
\def\imsart@plainnames{\def\fnms@size##1{##1}\def\snm@size##1{##1}}
\AtBeginDocument{%
  \setpkgattr{title}{size}{\LARGE}%
  \setpkgattr{title}{case}{}%
  \setpkgattr{title}{style}{\centering}%
  \setpkgattr{author}{size}{\large}%
  \setpkgattr{author}{style}{\centering}%
  \setpkgattr{author}{prefix}{}%
  \setpkgattr{address}{style}{\centering}%
  \imsart@plainnames
  \setpkgattr{abstract}{size}{\normalsize}%
  \setpkgattr{abstract}{width}{.85\textwidth}%
  \setpkgattr{abstract}{indent}{0pt plus 1fil}%
  \setpkgattr{abstractname}{skip}{\hfil\newline}%
  \def\footnoterule{\kern-3pt\hrule width 0.4\columnwidth height 0.4pt\kern 2.6pt}%
}
\makeatother

\makeatletter
\@ifundefined{input@path}{\def\input@path{}}{}
\edef\input@path{{./}{../}{../../}{../../../}\input@path}
\makeatother

\providecommand{\xishape}{\itshape}

\startlocaldefs
    \usepackage{math_symbols}
    \usepackage{theorem_env}
\endlocaldefs
\makeatletter
\def\journal@name{}
\makeatother

\begin{document}

\begin{frontmatter}
\title{Fast rates for the classification of Ornstein--Uhlenbeck paths \support{This research has been partially funded by Deutsche Forschungsgemeinschaft (DFG) SFB1294 and DFG CRC/TRR388.}}
\runtitle{Fast rates for the classification of OU paths}

\begin{aug}
\author[A]{\inits{S.}\fnms{Shanshan}~\snm{Meng}\ead[label=e1]{mengshan@hu-berlin.de}}
\author[B]{\inits{M.}\fnms{Markus}~\snm{Reiß}\ead[label=e2]{mreiss@math.hu-berlin.de}}
\address[A]{Humboldt-Universit\"at zu Berlin \printead[presep={,\ }]{e1}}
\address[B]{Humboldt-Universit\"at zu Berlin \printead[presep={,\ }]{e2}}

\end{aug}

\begin{abstract}
    In binary classification problems, plug-in classifiers can sometimes achieve faster excess risk rates than the underlying estimation rate. In this work, we study a supervised classification problem based on repeated independent observations of Ornstein--Uhlenbeck paths, with both time-homogeneous and time-inhomogeneous drifts. We show that, under a suitable separation condition on the class-dependent drifts, a plug-in classifier attains an excess risk of quadratic order in the corresponding parametric or nonparametric estimation error. The theoretical approach utilizes the Wiener chaos representation and spectral theory to characterize the log-likelihood ratio as a quadratic form of Gaussian random variables, enabling a precise analysis of margin behaviour and deviation results. This extends the fast-rate results from classification problems with linear discriminant models and Gaussian white noise models to dynamical diffusion systems with Gaussian structure.
\end{abstract}



\end{frontmatter}

\section{Introduction}\label{sec:intro}

Supervised classification is among the oldest and most thoroughly studied problems of statistical learning.
One constructs a decision rule that assigns a new observation to one of finitely many classes given $n$ labelled observations.
The quality of the classification is commonly evaluated by the excess of its misclassification probability over that of the optimal Bayes rule, see Devroye, Györfi and Lugosi~(1996)~\cite{devroye1996}, Tsybakov~(2009)~\cite{tsybakov2009}, Trabs, Jirak, Krenz and Reiß~(2021)~\cite{markus2023}.

For functional data the observation is not a vector but a trajectory.
Classification methods for such data have found applications in trajectory-based prediction; see, e.g.\, Chiou (2012) \cite{chiou2012} for an application in traffic flow forecasting.
The classification of sample paths of stochastic differential equations observed over a fixed time horizon has become an important special case within this framework; see e.g.\ Kowalek et al.~(2019)~\cite{kowalek2019} for identifying the underlying diffusion mode of particles in biology.
One research direction encodes each observed path through its rough-path signature and feeds these features into generic learning algorithms; see the primer of Chevyrev and Kormilitzin~(2026)~\cite{Chevyrev2026}.
Signature methods are often effective in practice, but statistical guarantees such as convergence rates for the resulting classifiers are not yet well understood in the absence of model-specific information.

A model-based alternative retains the model structure by building plug-in classifiers directly from estimators of the model coefficients.
For time-homogeneous SDEs with smooth, bounded space-dependent drifts, Denis et al. (2024) \cite{denis2024} proved that a nonparametric plug-in classifier attains the excess rate $n^{-s/(2s+1)}$ up to a subpolynomial factor, building on the ridge-type drift estimators of Denis, Dion-Blanc and Martinez (2021) \cite{denis2021}.
Zhao, Fan and Liu (2026) \cite{zhao2026} extend this framework to multidimensional diffusions with neural network estimators, and obtain excess risk bounds $n^{-s/(2s+d)+\varepsilon}$ for any $\varepsilon \in (0, 1/4]$ up to logarithmic and discretization terms.
In both cases the classification rate is inherited directly from the generic nonparametric estimation rate through the Lipschitz property of the excess risk, and therefore remains of the same order.
In classification theory, however, sometimes even fast excess risk rates of quadratic order in the estimation rates can be achievable when the posterior probability puts little mass near the decision threshold, see Gadat, Gerchinovitz and Marteau~(2020)~\cite{gadat2020}, who establish such fast rates in the Gaussian white noise model.
Whether the same holds for diffusion paths is the question addressed in this paper.

The prototypical model in which to pose this question is the Ornstein--Uhlenbeck (OU) process, one of the most basic parametric diffusion models, cf. Ciolek, Marushkevych and Podolskij~(2020)~\cite{ciolek2020}.
Here, we consider the binary classification model
\begin{equation} \label{eq:OUmodel}
  \diff X_t=Y\,b_1(t)\,X_t\diff t+(1-Y)\,b_0(t)\,X_t\diff t+\diff W_t,
  \qquad t\in[0,T],
\end{equation}
where $T>0$ is fixed, $Y\in\{0,1\}$ denotes the class label, $W$ is a standard scalar Brownian motion and $b_0,b_1$ are the deterministic time-dependent drift functions associated with the two classes.
Equivalently, conditional on $Y=i$, $i \in \{0,1\}$, the observed trajectory follows
\begin{equation}\label{eq:OUmodel-cond}
  \diff X_t=b_i(t)\,X_t\diff t+\diff W_t,
  \qquad t\in[0,T].
\end{equation}
When the drift coefficients are constant, the dynamics reduce to the classical time-homogeneous Ornstein--Uhlenbeck process.
Allowing $b_i$'s to vary with time, we obtain a time-inhomogeneous Gaussian Markov diffusion while preserving the linear structure that makes the OU process analytically tractable, which turns the classification problem into a nonparametric one.
From the statistical side, the likelihood of the OU model is explicit through Girsanov's theorem, see Liptser and Shiryaev (2001) \cite{LiptserShiryaev2001}, which has made it the standard class in parametric inference for diffusions, see Ibragimov and Has'minskii (1981) \cite{ibragimov1981} and Kutoyants (2004) \cite{kutoyants2004}; and in hypothesis testing between competing dynamics, see Gapeev and Küchler (2008) \cite{gapeev2007} and Lin'kov (1999) \cite{linkov1999}.

In the constant drift case, we show that the plug-in classifier attains the fast quadratic excess rate $n^{-1}/ \Delta$ up to logarithmic factors once the drift separation $\Delta = |b_1-b_0|$ satisfies $\Delta \gtrsim_{\log} n^{-1/2}$ (see Theorem~\ref{thm:mainconstant}).
The fast rate improves on the parametric estimation rate $n^{-1/2}$ significantly.
In the time-inhomogeneous case, we obtain the analogous bounds with the nonparametric estimation rate $n^{-s/(2s+1)}$ over Hölder balls of smoothness $s$ under Assumption~\ref{ass:drifts}.
If the drift separation $D = \Vert b_1-b_0 \Vert_{L^1}$ for $b_1 \ge b_0$ is bounded away from zero, the excess risk can achieve a fast rate of $n^{-2s/(2s+1)}/D$ up to logarithmic factors (see Theorem~\ref{thm:maininhomo}).
This is considerably faster than the nonparametric estimation rate.
The results hold for any estimator meeting the deviation bounds in Assumptions~\ref{ass:est} or~\ref{ass:estinhomo} respectively, cf. the estimators in the parametric theory \cite{ibragimov1981, kutoyants1994, kutoyants2004, liptser2001} and the nonparametric drift estimators in \cite{comte2020, comte2023, comte2024, denis2021, marie2023}.

The results of the fast rates rest on the behaviour of the model near the decision boundary.
Mammen and Tsybakov~(1999)~\cite{mammen1999} showed that the excess rate is governed by the mass that the posterior probability places near the decision boundary.
Tsybakov~(2004)~\cite{tsybakov2004} formulated this as a \textit{margin condition} on the posterior probability: if the probability of the posterior lying within $\varepsilon$ of the threshold is at most $C \varepsilon^{\alpha}$, rates faster than the parametric $n^{-1/2}$ become achievable.
Audibert and Tsybakov (2007) \cite{audibert2007} showed that simple plug-in rules attain these fast rates, and the margin condition has since become the organizing principle of classification theory, see Tsybakov (2009) \cite{tsybakov2009}.
A larger $\alpha$ corresponds to less mass near the decision boundary and faster excess rates.
The excess risk rates can be of order $v_n^{1+\alpha}$, if the estimation rate is $v_n$.
However, the condition is usually an \textit{assumption} placed on an abstract class of joint distributions.
In structured models the margin behaviour must be a consequence of the model.
The margin mass is a small-ball probability of the log-likelihood ratio near zero under the Bayes rule.
The margin condition therefore reduces to a density estimate for the log-likelihood variable near the decision boundary.
This is exactly what happens in the linear discriminant analysis (Trabs, Jirak, Krenz and Reiß~(2021)~\cite{markus2023}) and in the idealized Gaussian white-noise model (Gadat, Gerchinovitz and Marteau~(2020)~\cite{gadat2020}).
The uniform boundedness of the Gaussian density guarantees that the margin condition holds with $\alpha=1$ and yields quadratically faster classification rates under an explicit separation condition between the two classes.

In the OU model, the log-likelihood ratio is a quadratic functional of the Gaussian path and the situation becomes considerably more subtle.
The distribution is complicated with no convenient explicit density in general, see Dankel~(1991)~\cite{squareOU} and Fatalov~(2008)~\cite{fatalov2008}.
Ella-Mintsa~(2026)~\cite{mintsa2026} proposes a Malliavin calculus approach to establish the density bound, but the resulting argument is not applicable in all settings.
The integrability condition of $D Z / \Vert D Z \Vert_H^2$ for the Malliavin derivative of the log-likelihood variable $Z$ may already fail in the constant case where $b_1+b_0=0$ (see Remark~\ref{rem:malliavin}).
Moreover, the exponent $\alpha=1$ can be lost in the margin condition in the time-inhomogeneous setting (see Example~\ref{ex:countermargin}).
These observations indicate that establishing the margin condition requires a more careful analysis of the distributional structure of the underlying random variable, which ultimately motivates the use of more refined tools.

The starting point of our analysis is that the log-likelihood ratio belongs to the second Wiener chaos and admits a generalized $\chi^2$-representation with the spectral weights given by the eigenvalues of a Hilbert--Schmidt operator arising from the Malliavin analysis.
Rather than bounding the density directly, we study the spectrum of the associated operator, since the density estimates for Gaussian quadratic forms can be expressed in terms of the eigenvalues, see Ulyanov~(1991)~\cite{ulyanov1991}.
We relate the spectral problem to a Sturm--Liouville type boundary value problem, which makes the eigenvalue structure more tractable.
This connection allows us to derive two-sided eigenvalue bounds.
These estimates are then translated into a density bound for the log-likelihood ratio near the decision boundary, which yields the desired margin condition.
In particular, the margin condition is satisfied with exponent $\alpha = 1$ and an effective constant proportional to $1/D$ (Propositions~\ref{prop:margin} and \ref{prop:margininhomo}).
The same representation also controls the estimation term, and the resulting framework extends beyond the OU setting to more general Gaussian processes.
Altogether, our results show that plug-in classifiers can achieve quadratic rates of convergence in the OU classification model.
Our results obtain the excess risk of order $(\log n)^{3} n^{-2s/(2s+1)}/D$. In comparison, Denis et al. (2024) \cite{denis2024} obtain the generic rate $\exp (\sqrt{c \log n})n^{-s/(2s+1)}$ for bounded space-dependent drifts.
As discussed above, this improvement crucially relies on Assumption~\ref{ass:drifts}, without which the exponent $\alpha=1$ may be lost (see Example~\ref{ex:countermargin}).
Example~\ref{ex:counterunbounded} shows that a lower bound on $D$ alone does not yield a uniform margin constant or a fast rate of Theorem~\ref{thm:maininhomo}.
Some additional spectral-ratio condition is needed.

Section~\ref{sec:setting} introduces the OU classification model, the plug-in classifier and the decomposition of the excess risk into a margin term and an estimation term.
Section~\ref{sec:constantOU} states the results in the constant-drift case and concludes with a simulation study illustrating the findings.
Section~\ref{sec:inhomo} develops the time-inhomogeneous case with an analogous simulation study.
Section~\ref{sec-proof} contains the proof strategy and the key lemmas.

\section{The Ornstein--Uhlenbeck classification problem} \label{sec:setting}
This section introduces the precise setting of the Ornstein--Uhlenbeck classification problem.

\subsection{The Ornstein--Uhlenbeck classification model}\label{ssec:model}
We consider the binary classification model with feature space $C([0,T])$ endowed with the Borel $\sigma$-algebra induced by the uniform norm and $Y \in \{0,1\}$ with an unknown joint distribution $\prob_{X, Y}$.
Suppose that the feature $X=(X_t)_{t\in [0,T]}$ is a real-valued continuous trajectory feature with $X_0=0$, which is the solution of the equation
\begin{equation}\label{eq:OUmodel1}
  \diff X_t=Y\,b_1(t)\,X_t\diff t+(1-Y)\,b_0(t)\,X_t\diff t+\diff W_t,
  \qquad t\in[0,T],
\end{equation}
where $T>0$ is fixed, $b_1,b_0 \in L^2([0,T])$, $(W_t)_{t \in [0,T]}$ is a standard scalar Brownian motion, and $Y\sim\mathrm{Ber}(1/2)$ is independent of $W$.
We observe a sample $\Samp_n:=(X^{(i)},Y^{(i)})_{1\le i\le n}$ of $n$ independent copies of $(X,Y)$.
A trajectory $X^{(i)}$ labelled with $Y^{(i)}=1$ corresponds to an observation of the drift $b_1$, while a trajectory $X^{(i)}$ labelled with $Y^{(i)}=0$ is associated with the drift function $b_0$.

Conditionally on $Y=i$, the solution has the closed form
\begin{equation}\label{eq:OUsol}
  X_t=\int_0^t \exp\!\Bigl(\textstyle\int_s^t b_i(\tau)\diff\tau\Bigr)\diff W_s,
  \quad \text{with }
  \E_i[X_t]=0,\quad
  \E_i[X_t^2]=\int_0^t \exp\!\Bigl(2\textstyle\int_s^t b_i(\tau)\diff\tau\Bigr)\diff s,
\end{equation}
where $\E_i[\,\cdot\,]:=\E[\,\cdot\mid Y=i]$ denotes the expectation conditional on the label.
For $f\in L^2([0,T])$ define
\begin{equation}\label{eq:Xnorm}
  \Vert f\Vert_{X,i}^2:=\int_0^T f(t)^2\,\E_i[X_t^2]\diff t, \quad i \in \{0,1\}.
\end{equation}
This is the natural norm in which the estimation error of $b_i$ propagates to the classification function.

\begin{exa}[Constant-drift specialization]
For $b_i(t)\equiv b_i\in\R$ this reduces to $X_t=\int_0^t e^{b_i(t-s)}\diff W_s$ and
\begin{equation}\label{eq:varconst}
  \E_i[X_t^2]=
   \begin{cases}
     \dfrac{e^{2b_i t}-1}{2b_i}, & b_i\neq 0,\\[1.2ex]
     t, & b_i=0.
   \end{cases}
\end{equation}
\end{exa}

Let $\prob$ denote the Wiener measure on $C([0,T])$. For $\theta \in L^2([0,T])$ let $\prob_\theta$ denote the law of the OU process with drift $\theta$ (either a constant or a function).
Since
$\int_0^T \theta(s)^2 x_s^2 \diff s \le \sup |x_s|^2 \Vert\theta\Vert_{L^2}^2<\infty$ for every $x \in C([0,T])$, Theorem~7.7 in Liptser and Shiryaev~(2001)~\cite{LiptserShiryaev2001} yields $\prob \sim \prob_{\theta}$ with the Girsanov density
\begin{equation}\label{eq:girsanov}
  p_\theta(X):=\frac{\diff\prob_\theta}{\diff\prob}(X)
   =\exp\!\left(\int_0^T \theta(s)\,X_s\diff X_s
        -\frac12\int_0^T \theta(s)^2 X_s^2\diff s\right).
\end{equation}
For $b_0, b_1 \in L^2([0,T])$ define the \textit{log-likelihood ratio}
\begin{equation} \label{eq:llr}
  \ell(X; b_1, b_0) := \log \frac{p_{b_1}}{p_{b_0}}(X) = \int_0^T (b_1-b_0)(s) X_s \diff X_s - \frac{1}{2} \int_0^T (b_1^2-b_0^2)(s) X_s^2 \diff s.
\end{equation}
The \textit{posterior classification probability} $\prob_{X,Y}(Y=1 \mid X)$ for any drift pair $(b_1,b_0)$ is then
\begin{equation}\label{eq:eta-model}
  \eta(X;b_1,b_0)
    :=\frac{\tfrac{1}{2} p_{b_1} (X)}{\tfrac{1}{2} p_{b_1} (X) + \tfrac{1}{2} p_{b_0} (X)}
    = \bigl(1+e^{-\ell(X;b_1,b_0)}\bigr)^{-1}.
\end{equation}
A \textit{classifier} is a measurable map $\Phi:C([0,T])\to \{0,1\}$ (possibly depending on $\Samp_n$), and its performance is measured by the misclassification error
\begin{equation}
  \risk(\Phi):=\prob_{X,Y}\bigl(Y\neq\Phi(X)\bigr).
\end{equation}
\begin{defi}
  In the OU setting, for the estimators $(\hat{b}_1,\hat{b}_0)$ of the drift pair $(b_1, b_0)$, the \textit{plug-in classifier} is defined as
\begin{equation}\label{eq:plugin}
  \begin{aligned}
    \Phi(X;\hat{b}_1,\hat{b}_0)
     &:=\indic \left( \eta(X;\hat{b}_1,\hat{b}_0)\ge 1/2 \right)
     &= \indic \left( \ell(X;\hat{b}_1,\hat{b}_0) \ge 0 \right).
  \end{aligned}
\end{equation}
The \textit{Bayes classifier} is $\Phi^*(X):=\Phi(X; b_1,b_0)$.
\end{defi}
The accuracy of any data-dependent classifier $\widehat\Phi$ is quantified by its \textit{excess risk}
\begin{equation}\label{eq:excessrisk}
  \mathcal{E}(\widehat\Phi) := \risk(\widehat\Phi)- \inf_{\Phi}\risk(\Phi)
   =\E_X\!\left[\,\vert 2\eta(X;b_1,b_0)-1\vert\,
     \indic \bigl( \widehat\Phi(X)\neq\Phi^*(X) \bigr)\right],
\end{equation}
where $\risk$ and $\E_X$ are taken conditionally on $\Samp_n$, following the formulation commonly used in the literature \cite{devroye1996, audibert2007, markus2023}. The infimum $\inf_{\Phi}\risk(\Phi)$ is attained by the Bayes classifier $\Phi^*$ defined above.

\begin{rem}
  We focus on the model in the specific situation with equal class probabilities.
  However, for some $\pi \in (0,1)$ and $Y \sim$Ber($\pi$), the formulation of the posterior classification probability is analogous.
  The posterior, the decomposition and the estimation term carry over with the shift $\log \pi/(1-\pi)$. The margin proposition concern the small-ball probability of $\ell$ at $-\log \pi/(1-\pi)$ and the exponent $\alpha=1$ can fail there.
\end{rem}
From now on, $\prob_{b_i}$ denotes the conditional distribution given $Y=i$, and $\E_i$ the conditional expectation. When $\Samp_n$ is additionally included in the subscript, it denotes the corresponding (joint) distribution with $\Samp_n$. Moreover, $\prob_X$ denotes the marginal distribution of the input $X$.

\subsection{Excess risk decomposition and the margin condition}\label{ssec:decomp}
Write $\widehat{\Phi}(X)=\Phi(X; \hat{b}_1, \hat{b}_0)$ and $\widehat{\eta}(X)=\eta(X; \hat{b}_1, \hat{b}_0)$ for estimators $\hat{b}_1$ and $\hat{b}_0$. The derivation of the upper bound proceeds via the classical decomposition of the excess risk at level $\varepsilon > 0$,
\begin{align}
  \E_{\Samp_n}[\mathcal{E}(\widehat\Phi)]
  &=\E_{\Samp_n, X}\!\left[\vert 2\eta(X)-1\vert\,
       \indic \bigl( \widehat\Phi\neq\Phi^* \bigr)
       \Bigl(\indic \bigl( |\eta(X)-\frac12|\le\varepsilon \bigr)
            +\indic \bigl( |\eta(X)-\frac12|>\varepsilon \bigr)\Bigr)\right]
       \notag\\
  &\le 2\varepsilon\,
        \underbrace{\prob_X\bigl(\vert 2\eta(X)-1\vert\le 2\varepsilon\bigr)}
        _{=:(\mathrm{MA})}
     +\underbrace{\prob_{\Samp_n,X}\bigl(\vert\widehat{\eta}(X)-\eta(X)\vert>\varepsilon\bigr)}
        _{=:(\mathrm E)} .
  \label{eq:masterdecomp}
\end{align}
The margin term (MA) involves only the test path $X$, while the estimation term (E) involves both $\Samp_n$ and the independent test pair $(X,Y)$.
The first inequality for the term (MA) uses the bound $|2\eta-1|\le 2\varepsilon$ on $\{|\eta-\frac12|\le\varepsilon \}$.
It is controlled by the margin assumption and reflects the behaviour near the decision boundary.
The second term (E) uses the elementary inclusion $\{ \widehat\Phi\neq\Phi^* \} \cap \{|\eta-\frac12|>\varepsilon\}
 \subset \{|\widehat{\eta}-\eta|>\varepsilon\}$ together with
$|2\eta-1|\le 1$
and represents the regression estimation error of $\widehat{\eta}$.

As shown in earlier works on supervised classification (see e.g.\ \cite{mammen1999, audibert2007, gadat2020}), the margin probability (MA)
is a key component in the convergence rate of the excess risk. Mammen and Tsybakov \cite{mammen1999} show that, under suitable margin and density assumptions, the discriminant problem can attain fast rates of convergence, better than $1 / \sqrt{n}$. Audibert and Tsybakov \cite{audibert2007} extend the result to more general classification problems. The \textit{margin assumption} is stated as follows:
\begin{assu}\label{ass:margin}
  There exist constants $C_{\mathrm{M}} >0$ and $\alpha\ge 0$ such that
\begin{equation}\label{eq:margin}
    \prob_X\!\left(\Bigl|\eta(X)-\frac12\Bigr|\le\varepsilon\right)
    \le C_{\mathrm{M}}\,\varepsilon^{\alpha},
    \qquad \varepsilon>0 .
\end{equation}
\end{assu}
Here $\alpha$ and $C_{\mathrm{M}}$ may depend on the model, i.e.\ in our specific setting \eqref{eq:OUmodel} on the entries $b_1$, $b_0$ and $T$. However, they do not depend on $\varepsilon$ or the sample size $n$.
The exponent $\alpha$ measures how much mass concentrates near the decision boundary $\eta=1/2$.
$\alpha=0$ is the trivial bound, while larger $\alpha$ yields faster rates.
The case $\alpha=1$ is here the one of main interest, as the excess risk decays as the square of the estimation rate for the drifts up to logarithmic factors.
In the OU model the margin condition turns out to be governed by the drift separation for small $\varepsilon$; see Proposition~\ref{prop:margin} and Proposition~\ref{prop:margininhomo}.
The propositions also compute $C_{\mathrm{M}}$ explicitly as a function of the model and isolate its dependence on the drift separation.

Using the local Lipschitz property of the exponential function (Section~\ref{ssec:reduction}), both terms in \eqref{eq:masterdecomp} reduce to shifted small-ball and tail probabilities of a single quadratic functional of the OU path.
For deterministic coefficients $u \in L^2([0,T])$ and $v \in L^1([0,T])$,
define the quadratic functional
\begin{equation}\label{eq:Z}
  Z=Z(u,v)
   :=\int_0^T u(t)\,X_t\diff W_t + \frac12\int_0^T v(t)\,X_t^2\diff t ,
\end{equation}
where $X$ is the OU process \eqref{eq:OUsol} of drift $b_i$ driven by $W$ conditionally on $Y=i$.
In the margin probability, the log-likelihood ratio $\ell$ can be expressed by deterministic $u = b_1 - b_0$ and $v = \pm (b_1 - b_0)^2$ given $Y$.
In the estimation error they are determined by $b_1-b_0$ and $\hat b_i - b_i$ given $Y$ and $\Samp_n$.
The whole analysis rests on:
(i) a spectral representation of $Z$ as a shifted generalized $\chi^2$ variable (Section~\ref{ssec:chaos});
(ii) eigenvalue analysis for the associated operator (Section~\ref{ssec:factorization} -- \ref{ssec:cJJ-evs});
and (iii) density and tail bounds for generalized $\chi^2$ variables via the spectrum analysis (Section~\ref{ssec:density} -- \ref{ssec:estimation}).

\par\smallskip\noindent\textbf{Notations.}\enspace
Throughout the paper we write
\begin{equation}\label{eq:notation}
  a:=b_1-b_0,\qquad
  D:=\int_0^T a(t)\diff t,\qquad
  \omega_i(t):=\exp\!\Bigl(-\!\int_0^t b_i(\tau)\diff\tau\Bigr),
\end{equation}
for the drift difference, the drift separation, and the inverse state-transition function of class $i$.
Since $b_i\in L^2([0,T])$,
$\omega_i$ is continuous and strictly positive, and we set
\begin{equation}\label{eq:comega}
  c_\omega:=\min_{i\in\{0,1\}}\;\min_{t\in[0,T]}\omega_i(t)>0,\qquad
  C_\omega:=\max_{i\in\{0,1\}}\;\max_{t\in[0,T]}\omega_i(t)<\infty .
\end{equation}
In this notation $\E_i[X_t^2]=\omega_i(t)^{-2}\int_0^t\omega_i(s)^2\diff s$.
Write $C_{\mathrm{est},i} := \int_0^T \E_i[X_t^2]\diff t$ for the total variance of the OU process under class $i$.
When $a$ is absolutely continuous, we further define
\begin{equation}\label{eq:psi}
  \psi:=a'+a\,(b_0+b_1)=(b_1-b_0)'+(b_1^2-b_0^2).%
\end{equation}
The function $\psi$ governs the sign structure of the spectrum of the second-chaos operator associated with the log-likelihood ratio; see Section~\ref{ssec:factorization}.
\section{Results in the constant-drift case} \label{sec:constantOU}

Throughout this section $b_1,b_0\in\R$ are constants with $b_1 > b_0$, $\Delta: = b_1-b_0$,
so that $a = \Delta > 0$ and $D = T \Delta$ for fixed $T$.
The function $\psi = b_1^2 - b_0^2$ is constant.
We concentrate our attention on the drift pairs in the set
\begin{equation} \label{eq:classb}
	\mathcal{B} (L, \Delta_0) : = \left\{ (b_1,b_0) \in [-L,L]^2 \; : \; b_1 - b_0 \ge \Delta_0 \right\}, \quad L > 0 , \; \Delta_0 \in (0, 2L].
\end{equation}

In the constant-drift OU case, we have the following margin probability result for general $L > 0$, $\Delta > 0$.
\begin{prop} \label{prop:margin}
Let $(X,Y)$ follow \eqref{eq:OUmodel} with constant drifts $b_1 \neq b_0$, $\Delta:=\vert b_1- b_0 \vert > 0$ and $\vert b_0 \vert  \le L$, $\vert b_1 \vert  \le L$.
There exists a constant $C_{\mathrm{margin}}(LT)> 0$ such that
\begin{equation}\label{eq:marginbound}
  \prob_X\!\left(\Bigl|\eta(X)-\tfrac12\Bigr|\le\varepsilon\right)
    \le \frac{C_{\mathrm{margin}} (LT)\,\varepsilon}{T \Delta}, \quad \text{for all }\varepsilon \le \min \left\{ \frac{T \Delta}{64}, \frac{1}{16}\right\}.
\end{equation}
Note that $C_{\mathrm{margin}} (LT)$ is independent of $\Delta$ and $\varepsilon$, depending only on the product $LT$.
\end{prop}

\begin{cor}
	The model \eqref{eq:OUmodel} with constant drifts satisfies Assumption~\ref{ass:margin} with
	\begin{equation} \label{eq:assmarginconst}
		\alpha = 1, \quad C_{\mathrm{M}} = \max \left\{\frac{C_{\mathrm{margin}}(LT)}{T \Delta}, \; 16 \right\}.
	\end{equation}
	For $\varepsilon \le 1/16$ this is \eqref{eq:marginbound} (extended to all $\varepsilon\le 1/16$ using $C_{\mathrm{margin}} \ge 64$).
	For $\varepsilon > 1/16$ the trivial bound $\prob_X( \cdot) \le 1 \le 16\varepsilon$ applies.
\end{cor}

\begin{rem} \label{rem:CMconst}
	For a fixed pair $(b_1, b_0)$ with $\Delta >0$, the quantity $C_{\mathrm{M}}$ in \eqref{eq:assmarginconst} is a constant in the sense of Assumption~\ref{ass:margin}. The dependence on the model is fully explicit: $C_{\mathrm{margin}}$ does not depend on $\Delta$, while $C_{\mathrm{margin}}(LT) / T \Delta $ blows up as $\Delta \to 0$.
	This is not an artefact of the proof: no exponent $\alpha>0$ holds with a uniform constant over any class of models in which $\Delta$ accumulates at $0$.
	Uniformly over $\mathcal{B}(L, \Delta_0)$ the margin assumption holds with $\alpha = 1$ and $C_{\mathrm{M}} = \max \left\{ C_{\mathrm{margin}} (LT) /(T \Delta_0), 16\right\}$, which depends only on $LT$ and $\Delta_0$. It is finite for all $L > 0$ and $\Delta_0 >0$ by the constant-case proof; see Section~\ref{sec-proof}.
\end{rem}

\begin{rem} \label{rem:marginLT}
	The margin bound of Proposition~\ref{prop:margin} depends on $T$ in two opposite ways: through the factor $1/(T\Delta )$ and through the constant $C_{\mathrm{margin}}$, which grows super-exponentially in $T$, which is an artefact of the bound.
	This growth does not imply that the margin result gets worse with longer observation windows.
	The $T$-dependence of $C_{\mathrm{margin}}(LT)$ enters through the uniform bound $C_\omega/c_{\omega}$ of the weight functions to control the eigenvalue ratios and is not expected to reflect the actual growth of these ratios sharply in $T$.
	We keep the statement, since the dependence on $T$ does not affect the rate in $n$ in the excess risk.
\end{rem}

Recall the definition of $\Vert \cdot \Vert_{X,i}$ from \eqref{eq:Xnorm} and the definition
\begin{equation}\label{eq:Cest}
	C_{\mathrm{est},i} := \int_0^T \E_i [X_t^2] \diff t = \Vert \mathbf{1} \Vert_{X,i}^2
	=\begin{cases}
		\dfrac{1}{4 b_i^2} \left(e^{2b_iT} -1 - 2 b_i T\right), & b_i \neq 0,\\[1.2ex]
		\dfrac{T^2}{2}, & b_i = 0.
	\end{cases}
\end{equation}
We have the following assumption on the estimators.
\begin{assu}\label{ass:est}
Let $N_i$ be the number of sample paths with label $i$. There exist estimators
$\hat b_i$, built from the $N_i$ paths of class $i$, and constants $K_0, C_0>0$ such that
\begin{equation}\label{eq:estcond}
	\prob_{\Samp_n} \left( \sqrt{C_{\mathrm{est},j}}\left\vert \hat{b}_i-b_i\right\vert  \ge K_0 \sqrt{\frac{\log N_i}{N_i}}  \; \Big\vert \; N_i=n_i \right) \le \frac{C_0}{n_i}, \quad i = 0, 1, \quad n_i \ge 2.
\end{equation}
\end{assu}

\begin{rem}
	The canonical choice is the maximum likelihood estimator of the class-$i$ drift (see \cite{liptser2001, kutoyants2004}). Applying the concentration inequality of second-chaos variables gives a tail bound in the form of \eqref{eq:estcond}. Classical exponential bounds for regular i.i.d. experiments also give the same conclusion \cite{ibragimov1981}. Finite-sample large-deviation bounds for (quasi-)maximum likelihood estimators are given in \cite{spokoiny2012}.
\end{rem}

The estimation error is bounded in the following lemma, which is the constant-drift specialization of Lemma~\ref{lem:est-general}.

\begin{lem} \label{lem-est}
	Let $\hat{b}_i$ be estimators under Assumption~\ref{ass:est} and define
	\begin{equation*}
		\delta_{i} := 48 \left(2\sqrt{2} K_0 + 8\,K_0^2 + 8 K_0 \, \sqrt{C_{\mathrm{est},i}} \, \Delta \right).
	\end{equation*}
	Let
	\begin{equation} \label{eq:choiceeps}
	\varepsilon_i := \delta_{i} \bigl(\log n\bigr)^{\frac{3}{2}} n^{-\frac{1}{2}}.
	\end{equation}
	For $n\ge 27$ large enough such that $\varepsilon_i \le 3/4$, the plug-in posterior probability has the conditional estimation error
	\begin{align*}
	\prob_{\Samp_n,b_i} \left(\vert \widehat{\eta}(X)-\eta(X) \vert > \varepsilon_i  \right)
	\le \frac{9+8C_0}{n}.
\end{align*}
\end{lem}

The main result on the upper bound is stated in the following theorem.
\begin{thm}[Excess-risk upper bound, constant drift]\label{thm:mainconstant}
Let $L > 0$, $T>0$ and $\Delta_0 \in (0,2L]$.
Let $(\hat{b}_1, \hat{b}_0)$ be the estimators in Assumption~\ref{ass:est} with constants $K_0, C_0$ uniformly over $\mathcal{B}(L, \Delta_0)$.
We denote by
\begin{equation} \label{eq:paraestrate}
v_n : = (\log n / n )^{1/2}
\end{equation}
the standard convergence rate for parametric estimation error in Assumption~\ref{ass:est}.
There exists a constant $N_{L,T,K_0}$ depending only on $L$, $T$ and $K_0$ such that for any $n \ge N_{L,T, K_0}$, the plug-in classifier
$\widehat\Phi=\Phi(X;\hat b_1,\hat b_0)$ satisfies
\begin{equation}\label{eq:mainrate}
  \sup_{(b_1, b_0) \in \mathcal{B}(L, \Delta_0)} \E_{\Samp_n} [ \mathcal{E} (\widehat\Phi)] \;\lesssim\;
  v_n \, (\log n) \wedge \frac{v_n^2 (\log n)^2 }{\Delta_0} ,
\end{equation}
where the implied constants in the inequalities depend only on $K_0, C_0$ in Assumption~\ref{ass:est}, $L$ and $T$.
\end{thm}

\begin{rem}[Comparison with related models]
	Our conclusions are consistent with those obtained for the Gaussian white-noise model under a smoothness assumption by Gadat et al.\ \cite{gadat2020} and Trabs et al.\ \cite{markus2023}.
	The behaviour transitions continuously between a slow rate $\log^{3/2}n/\sqrt n$ (small separation) and a fast rate $1/(T \Delta\, n)$ up to logarithmic factors.
	The extra $\log^{3}n$ in the fast regime is heavier than the $\log^2 n$ of \cite{gadat2020}; in the special LDA-type setting one can reduce the logarithmic factor further.

	Denis et al.\ \cite{denis2024} show that, with known diffusion and a space-dependent drift satisfying a smoothness and boundedness condition, a plug-in classifier achieves the rate $n^{-\frac{s}{2s+1}}$ in a multi-class classification problem. The role of boundedness is explicit in their argument, and their result attains only the standard generic rate for classification problems. Our analysis is designed for the more general setting of an unbounded space-dependent drift and achieves the fast rate whenever the drift distance is bounded from below.
\end{rem}

\par\smallskip\noindent\textbf{Simulations.}\enspace
We illustrate the excess risk results in the time-homogeneous model with $T=5$ and $b_1=-1$, $b_0 = b_1-\Delta$.
For each $\Delta$ and each $n \in \{50,100,500,1000,5000,10000,50000\}$, the drifts are estimated by maximum likelihood from $n$ training paths per class, and the plug-in classifier $\widehat{\Phi}_n$ is defined as in~\eqref{eq:plugin}.
Paths are simulated by the Euler scheme with $2048$ time steps.
For each drift configuration, the estimation is repeated $R=12$ times with independent training data; within each repetition and class, the training sets are nested across $n$.
For each $\Delta$, we generate a test sample of $M=10^7$ paths drawn from the mixture with class probabilities $(1/2,1/2)$, independently of the training data.
The same test sample is shared across all values of $n$ and all repetitions.
The mean excess risk is computed by
\begin{equation*}
	\frac{1}{R} \sum_{r=1}^{R} \frac{1}{M} \sum_{k=1}^M |2\eta(X_k)-1|\,\indic \bigl(\widehat\Phi^{(r)}_n (X_k)\neq\Phi^*(X_k) \bigr).
\end{equation*}

\begin{figure}[t]
 \centering
 \includegraphics[width=0.8\linewidth]{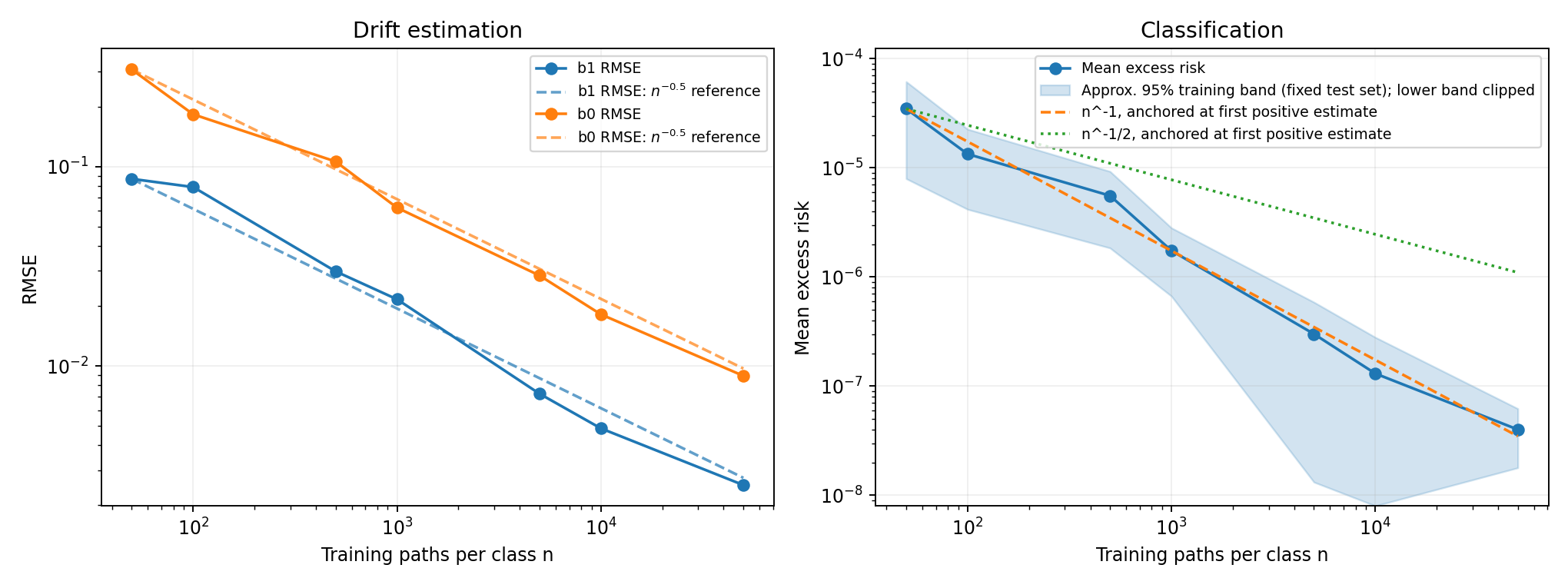}
 \footnotesize
 \caption{Constant-drift estimation and classification for $b_1=-1$, $b_0=-11$ ($\Delta=10$).
 Left: empirical RMSE for each drift coefficient. Right: mean excess risk.
 }
 \label{fig:constant-rmse-risk}
\end{figure}

Figure~\ref{fig:constant-rmse-risk} shows the root mean squared error of the parametric estimators (left) and the mean excess risk (right) against $n$ on a log--log scale.
The estimation error follows the parametric rate $n^{-1/2}$ of Assumption~\ref{ass:est} when $n$ gets large, while the excess risk is consistent with the rate $n^{-1}$, in agreement with the fast regime of Theorem~\ref{thm:mainconstant}.
At $n = 10000$ and $n=50000$, the estimated excess risks are approximately $10^{-7}$ and $4 \times 10^{-8}$, respectively.

\begin{figure}[t]
 \centering
 \includegraphics[width=0.8\linewidth]{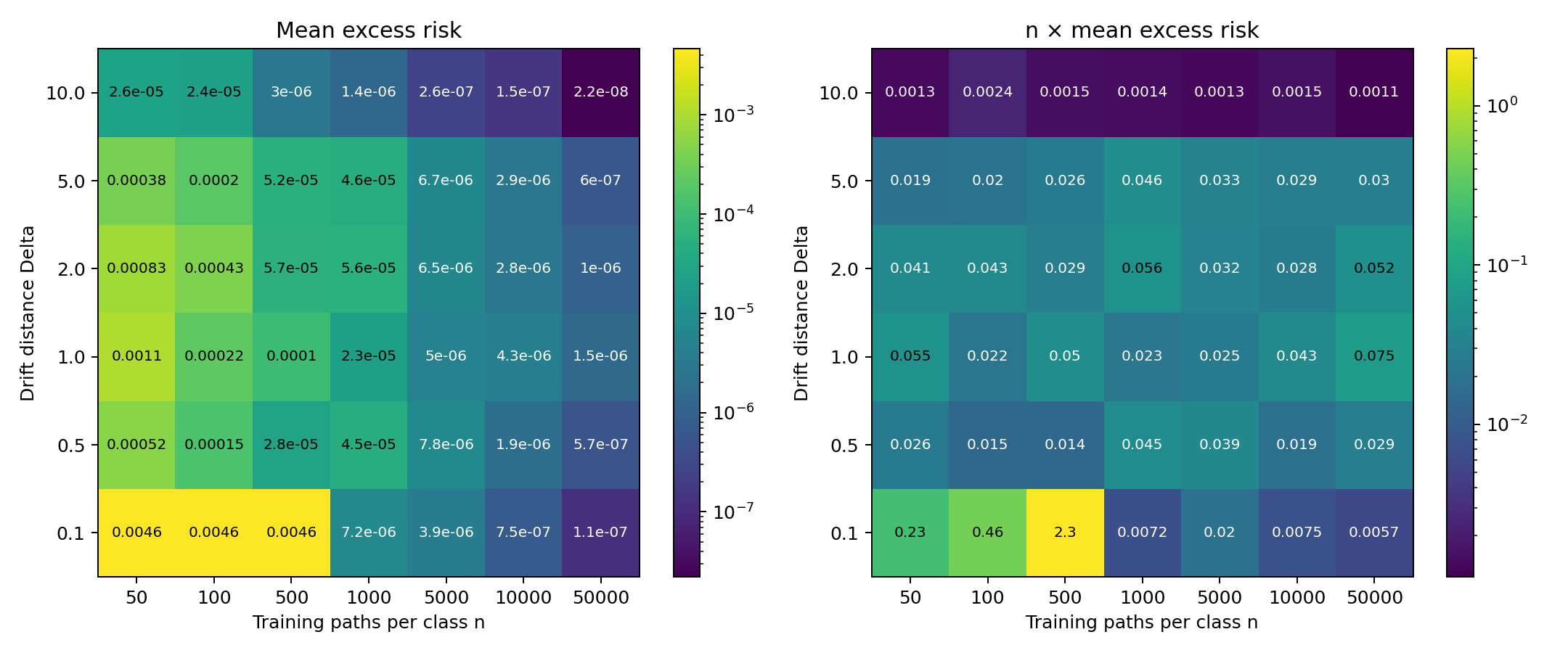}
 \footnotesize
 \caption{Constant-drift classification across training sample sizes and drift distances, with $b_1=-1$, $b_0=-1-\Delta$.
 Left: estimated mean excess risk $\mathcal E$.
 Right: the normalized quantity $n \mathcal E$.
 Color scales are logarithmic and cell labels give the numerical estimates.}
 \label{fig:constant-heatmap}
\end{figure}

Figure~\ref{fig:constant-heatmap} reports the mean excess risk $\mathcal{E}$ (left) and $n$ times mean excess risk (right) over a grid of $(\Delta, n)$.
For $\Delta \in [0.5, 10]$ the product $n \times \mathcal{E}$ is approximately constant in $n$, confirming the fast rate uniformly over this range.
The observed dependence on $\Delta$ is consistent with the factor $1/\Delta$ for $\Delta \ge 2$ along the $\Delta$-axis.
The transition from the slow rate to the fast regime is visible for $\Delta=0.1$.
Over the test configurations, the estimates are broadly compatible with an $n^{-1}$ decay.
The dependence on $\Delta$, on the other hand, differs from that of the upper bound~\eqref{eq:mainrate} in the regime $n^{-1/2} \ll \Delta \lesssim 1$, where the empirical excess risk is of order $\Delta n^{-1}$ rather than $(n\Delta)^{-1}$.
The same phenomenon occurs for linear discriminant analysis with two classes $\mathcal N (\mu_i, I_p)$ with $\Vert \mu_1-\mu_0 \Vert = \Delta$.
The estimation error decomposes into a shift induced by the estimated midpoint and a component orthogonal to the direction $\mu_1-\mu_0$, which contribute an excess risk of order $\Delta n^{-1}$ and $(p-1)(n\Delta)^{-1}$ respectively in Section 6.6 of \cite{anderson2003}.
In dimension $p=1$ only the former is present and yields an excess risk of order $\Delta n^{-1}$ for $\Delta \lesssim 1$.
In dimension $p \ge 2$ the orthogonal component dominates for $\Delta \lesssim 1$. This is the term behind the rate $1/\Delta$ in \cite{markus2023} when $p > 1$ and in \cite{gadat2020}.
The time-inhomogeneous case of Section~\ref{sec:inhomo} is of the multidimensional type.
We do not pursue this refinement here, since the dependence on $n$ is unaffected.

\section{Results in the time-inhomogeneous case} \label{sec:inhomo}

We now study the excess-risk convergence rate of the plug-in classifier for the time-inhomogeneous Ornstein--Uhlenbeck classification problem.
Recall the notation at the end of Section~\ref{sec:setting}.
We impose the following assumption on the drift function.
\begin{assu} \label{ass:drifts}
  Let $T > 0$. We assume
  \begin{enumerate}
    \item $b_0$, $b_1 \in L^2([0,T])$.
    \item $a = b_1- b_0$ is nonnegative with $a \not\equiv 0$ and Lipschitz continuous on $[0,T]$, and $\psi \in L^{\infty} ([0,T])$. Set $C_{\psi} := \Vert \psi \Vert_{L^{\infty}}$.
    \item One of the following holds almost everywhere on $[0,T]$:
    \begin{itemize}
        \item[$(-)$] $\psi(r) \ge 0$;
        \item[$(+)$] $\psi(r) \le -c_\psi$ for a constant $c_\psi \in (0, C_\psi]$. Define the \textit{spectral-ratio parameter}
        \begin{equation}\label{eq:rho}
		\rho:=\frac{C_\psi}{c_\psi}\left(\frac{C_\omega}{c_\omega}\right)^{4} < \infty.
	    \end{equation}
    \end{itemize}
  \end{enumerate}
\end{assu}
We assume that the two drift functions are ordered throughout the observation interval. The second assumption fixes the orientation of the two classes, which is symmetric after exchanging the labels.
It also gives $D = \int_0^T a(t) \diff t = \Vert a \Vert_{L^1} >0$.
The function $\psi$ determines the sign of the Volterra part of the spectrum of the second-chaos operator associated with the log-likelihood ratio; see
Section~\ref{ssec:factorization}. The multiplier appearing there is exactly
$-F'(r)=-\psi(r)\,e^{2\int_0^r b_i}$ for both labels $i\in\{0,1\}$.
The third assumption requires this multiplier to be of one sign.
In regime $(+)$ all eigenvalues of the operator are positive; see Section~\ref{ssec:density}.

\begin{rem} \label{rem:rhobound}
    The parameter $\rho$ also plays a role in estimating the ratio between the bounds of the sandwich function $-F'$ in the operator decomposition in Section~\ref{ssec:factorization} with $c_F = c_{\psi} /C_{\omega}^2$ and $C_F = C_{\psi} / c_{\omega}^2$, so that $\rho = C_{\omega}^2 C_F / (c_{\omega}^2 c_F)$.
    For $b_i \in L^{\infty}$ with $\Vert b_i \Vert_{L^{\infty}} \le L$, we have $\rho \le C_{\psi}e^{8LT} /c_{\psi}$ depending only on $L$, $T$ and the bounds of $\psi$.
    In regime $(-)$ no such condition is needed:
	there is at most one positive eigenvalue and the magnitude of other eigenvalues plays no role in the density upper bound; see Section~\ref{ssec:cJJ-evs}.
\end{rem}

To obtain nonparametric results analogous to the constant case, we introduce the following set of Hölder functions with smoothness $s$. For radius $R > 0$, $s >0$ and $k:=\lceil s\rceil-1$ define the Hölder ball
\begin{equation} \label{eq:hoelder}
    \begin{aligned}
    \Sigma(s,R):=\Bigl\{f:[0,T]\to\R \;:\; &
	f\in C^{k}([0,T]),\ \Vert f \Vert_{\infty} \le R, \\
    & \bigl|f^{(k)}(x)-f^{(k)}(y)\bigr|\le R \,|x-y|^{\,s-k}
	\ \forall x,y\in[0,T]\Bigr\}.
    \end{aligned}
\end{equation}
Note that $b_0$, $b_1\in\Sigma(s,R)$ with $s \ge 1$ implies the regularity requirements of Assumption~\ref{ass:drifts} parts 1--2.
The drifts are Lipschitz, so $a=b_1-b_0$ is Lipschitz with $a'\in L^\infty$ and $\psi=a'+a(b_0+b_1)\in L^\infty$.
For $s<1$ the Lipschitz continuity of $a$ in Assumption~\ref{ass:drifts} is a separate requirement on the difference of the drifts, which may be much smoother than the drifts themselves.

Similarly to the constant case, we consider drift pairs in the set
\begin{equation} \label{eq:classbinhomo}
    \mathcal{B}_{\mathrm{in}}(\rho_0, D_0) := \left\{(b_1,b_0) \text{ satisfying Assumption~\ref{ass:drifts}, } \rho \le \rho_0\text{ in regime }(+), \; D \ge D_0\right\}.
\end{equation}
The analysis reduces to a margin estimate and a tail estimate for the log-likelihood variable in \eqref{eq:Z}
with $(u,v) = (a, \pm a^2)$ for the margin term and $u = \hat b_i - b_i$, $v$ being a function of $u$ and $a$ for the estimation term; see \eqref{eq:M1}, \eqref{eq:M0}, \eqref{eq:E1} and \eqref{eq:E0} below.

\begin{prop} \label{prop:margininhomo}
Let $(X,Y)$ follow the model~\eqref{eq:OUmodel} with drift functions satisfying Assumption~\ref{ass:drifts} for fixed $T$.
There exists a constant $C_{\mathrm{margin}} (\rho)$ such that
\begin{equation}\label{eq:margininhomo}
  \prob_X\!\left(\Bigl|\eta(X)-\tfrac12\Bigr|\le\varepsilon\right)
    \le \frac{C_{\mathrm{margin}} (\rho)\,\varepsilon}{D}, \quad \text{for all }  \varepsilon \le \min \left\{\frac{D}{64}, \frac{1}{16}\right\}
\end{equation}
where $C_{\mathrm{margin}}(\rho)$ depends explicitly on $\rho$ in regime $(+)$ and is a universal constant in regime $(-)$.
\end{prop}
\begin{cor}
	The model \eqref{eq:OUmodel} under Assumption~\ref{ass:drifts} satisfies Assumption~\ref{ass:margin} with
	\begin{equation} \label{eq:assmargininhomo}
		\alpha = 1, \quad C_{\mathrm{M}} = \max \left\{\frac{C_{\mathrm{margin}}(\rho)}{D}, \; 16 \right\}
	\end{equation}
	for all $\varepsilon > 0$.
\end{cor}
Analogous to Remark~\ref{rem:CMconst}, $C_{\mathrm{M}}$ is a constant in the sense of Assumption~\ref{ass:margin}.
Uniformly over $\mathcal{B}_{\mathrm{in}}(\rho_0, D_0)$,
the margin assumption holds with $\alpha=1$ and the constant $C_{\mathrm{M}}=\max\{C_{\mathrm{margin}}(\rho_0)/D_0, 16\}$ by \eqref{eq:Cmargindef}.
The proof is given in Section~\ref{ssec:density}.

The separation condition $D > 0$ and the Lipschitz continuity of $a$ in Assumption~\ref{ass:drifts} cannot be dropped simultaneously to obtain a margin assumption with $\alpha \ge 1$.
The following example is built so that the Volterra part of the log-likelihood operator in Section~\ref{ssec:factorization} vanishes ($\psi = 0$ a.e.) while $a$ has a single jump.
The spectrum then reduces to two eigenvalues of opposite sign and the singularity of the density sits exactly at the origin.
\begin{exa}[Critical example to the Margin Assumption with $\alpha=1$] \label{ex:countermargin}
    Let $T=1$ and fix
    $    0 < c \le \frac{1}{2e}$ to keep $b_i < 0$.
    Define the right-continuous bounded variation function
    \begin{equation*}
        a(t) =
        \begin{cases}
            c \exp(2t), & 0 \le t < \frac{1}{2}, \\
            - c \exp(2(t-\frac{1}{2})), & \frac{1}{2} \le t \le 1.
        \end{cases}
    \end{equation*}
    Put
    \begin{equation*}
        b_1(t) = -1 + \frac{a(t)}{2}, \quad b_0(t) = -1 - \frac{a(t)}{2}.
    \end{equation*}
    Then we have $\psi = a' - 2a = 0$ on both pieces.
    The pair violates Assumption~\ref{ass:drifts} that $a$ changes sign, is not Lipschitz (jumps by $-c(1+e)$ at $t=1/2$), and $D= \int_0^1 a = 0$.
    By Lemma~\ref{lem:countermargin-reduct} the log-likelihood ratio in \eqref{eq:llr} is of rank two with opposite sign eigenvalues.
    Conditionally on $Y=i$
    \begin{equation} \label{eq:ellrank2}
        \ell(X; b_1, b_0) = \frac{c}{2} \bigl[(1+e)X_{\frac{1}{2}}^2 - e X_1^2\bigr] = \lambda_+ \xi_1^2 - \lambda_- \xi_2^2
    \end{equation}
    with $\lambda_{\pm}=\lambda_{\pm}(i) >0$ and independent standard normal variables $\xi_1$, $\xi_2$.
    By Lemma~\ref{lem:countermargin},
    there are $C_{c}$, $t_{c}>0$ such that
    \begin{equation*}
        \prob_{X} \left(\left| \eta(X) - \frac{1}{2} \right| \le t \right) \ge \frac{1}{2} \sum_i \prob_{b_i} \left(|\ell| \le 4 t \right) \ge C_{c} t \log \left(\frac{1}{t}\right), \quad 0 < t \le t_c,
    \end{equation*}
    with $C_c := \min_i c(\lambda_+(i),\lambda_-(i))$ and $t_c:=\min{\min_i t_0(i)/4,1/16}$, using $4t \log(1/(4t)) \ge 2t\log (1/t)$ for $t\le 1/16$.
    The margin assumption fails for any $\alpha \ge 1$ and Theorem~\ref{thm:maininhomo} does not apply.
    This example marks the boundary of the linear margin assumption.
\end{exa}

The next example shows that a lower bound on the separation $D$ alone does not suffice to guarantee a uniform margin constant.
The following family keeps $D=1$ while the two classes merge in distribution, which is in the proposition reflected by the spectral-ratio parameter $\rho$ diverging to infinity.
\begin{exa}[Example of $\rho$ unbounded] \label{ex:counterunbounded}
    Let $T=1$. For $\tau > 0$ set
    \begin{equation*}
    L:=\log\frac1\tau\ (\ge784),\qquad c:=\frac{1}{1+L},\qquad
    a(t):=\frac{c}{t\vee\tau},\qquad
    b_1=\frac{a-2}{2},\quad b_0=-\frac{a+2}{2}.
    \end{equation*}
    By Lemma~\ref{lem:counterrho}, every pair $(b_1,b_0)$ in this family satisfies Assumption~\ref{ass:drifts} in regime $(+)$ with $D=1$, while
    \begin{equation} \label{eq:rho-infty}
        \frac{\tau^{-2}+2\tau^{-1}}{3} \le \rho \le \frac{e^8}{\tau^2}, \quad \rho \longrightarrow \infty, \quad \text{ for } \tau \downarrow 0.
    \end{equation}
    With $\varepsilon_\tau:=7/ (4\sqrt L) \le 1/16$ for $\tau \le e^{-784}$,
    \begin{equation} \label{eq:rho-margin}
        \prob_X \Bigl(\big|\eta(X)-\frac12\big|\le\varepsilon_\tau\Bigr)\ \ge\ \frac{3}{8}
    \end{equation}
    Any constant $C$ for which $C \varepsilon_{\tau} / D$ bounds the margin probability must satisfy $C \ge 3 \sqrt{L}/14 \to \infty$.
    Hence, no positive margin exponent admits a uniform constant over $\{\tau > 0 \}$, although $D$ is fixed.
    In this family, $a$ concentrates near $t = 0$, where $X_t$ is still small.
    A uniform statement needs a parameter beyond $D$ over the family under consideration.
    We take $\rho$ in Proposition~\ref{prop:margininhomo}.
\end{exa}

When $b_0$, $b_1 \in \Sigma(s, R)$, standard nonparametric estimators built from $N$ i.i.d. paths of the class attain the classical rate $N^{-s/(2s+1)}$ in $L^2([0,T])$, cf. \cite{comte2023, denis2021}.
We do not fix a particular estimator and instead formulate the requirement on the estimators.
\begin{assu}\label{ass:estinhomo}
Let $N_i$ be the number of sample paths with label $i$ and let $s > 0$ be the smoothness index for the drift pairs. There exist estimators
$\hat b_i$, built from the $N_i$ paths of class $i$, and constants $K_s, C_s>0$ such that for $i, j \in \{0,1\}$ and all $n_i \in \N$,
\begin{equation}\label{eq:estcondinhomo}
	\prob_{\Samp_n} \left( \left\Vert \hat{b}_i-b_i\right\Vert_{X, j}  \ge K_s \left( \frac{\log N_i}{N_i}\right)^{\frac{s}{2s+1}} \; \Big\vert \; N_i=n_i \right) \le \frac{C_s}{n_i}.
\end{equation}
$C_s$ depends only on $s$ and not on $n_i$ or $N_i$.
\end{assu}
\begin{rem}
    Standard nonparametric projection estimators and their associated convergence rates are available for this model; see Comte and Genon-Catalot (2023, 2024) \cite{comte2023, comte2024}. For discrete or sparse observations see Mohammadi et al. (2024) \cite{Mohammadi2024}.

    For the canonical sieve maximum likelihood estimator of the degree $\lceil s \rceil -1$ on $\lceil (N_i/\log N_i)^{1/(2s+1)} \rceil$ intervals, one may take $C_s = c_1 + c_2 \lceil s \rceil$ with numerical constants $c_1, c_2$ independent of $s$ and $K_s$ bounded by a constant depending on $s$, $R$ and $T$; see Lemma~\ref{lem:sieveMLE}.
    This reflects the parametric requirement that the polynomial degree grows with $s$, cf. Comte and Genon-Catalot (2020)~\cite{comte2020}
    in the time-homogeneous case.
\end{rem}

The estimation error is bounded in the following lemma. Note that Assumption~\ref{ass:drifts} (3) is not needed here.
 \begin{lem} \label{lem-estinhomo}
	Let $\hat{b}_i$ be estimators satisfying Assumption~\ref{ass:estinhomo} for $i,j \in \{0,1\}$.
    Set
	\begin{equation*}
		\delta_{n, j} := 48 \left(2\sqrt{2} K_s + 8\, K_s^2 \bigl(\log n\bigr)^{\frac{s}{2s+1}} n^{-\frac{s}{2s+1}} + 8 \, K_s \Vert a \Vert_{X,j} \right).
	\end{equation*}
    Let
    \begin{equation} \label{eq:choiceepsinhomo}
	\varepsilon_j := \delta_{n,j} \bigl(\log n\bigr)^{\frac{3s+1}{2s+1}} n^{-\frac{s}{2s+1}}.
	\end{equation}
	For $n\ge 27$ large enough such that $0 < \varepsilon_j \le 3/4$,
	the plug-in posterior probability $\widehat{\eta}(X) = \eta(X; \hat{b}_1, \hat{b}_0)$ has the estimation error
	\begin{align*}
		\prob_{\Samp_n, b_j} \left(\vert \widehat{\eta}(X)-\eta(X) \vert > \varepsilon_j \right)
        \le \frac{9+8C_s}{n},
	\end{align*}
    where $C_s$ is the constant depending on $s$ in Assumption~\ref{ass:estinhomo}.
\end{lem}
The main result on the upper bound is stated in the following theorem.
\begin{thm}[Excess-risk upper bound, time-inhomogeneous drift]\label{thm:maininhomo}
    Let $s$, $R$, $T$, $\rho_0$, $D_0 > 0$.
    Let $(\hat{b}_1, \hat{b}_0)$ be estimators satisfying Assumption~\ref{ass:estinhomo} for any drift pair $(b_1, b_0) \in \mathcal{B}_{\mathrm{in}}(\rho_0, D_0) \cap \Sigma (s, R)^2$ as defined in \eqref{eq:hoelder} and \eqref{eq:classbinhomo} with constants $K_s$ and $C_s$ uniformly over the drift set.
    We denote by
    \begin{equation} \label{eq:nonparaestrate}
    v_n : = (\log n/ n)^{\frac{s}{2s+1}}
    \end{equation}
    the standard convergence rate for nonparametric estimation in Assumption~\ref{ass:estinhomo}.
    There exists a constant $N_{s,R,T,K_s}$ depending on $s$, $R$, $T$ and $K_s$ such that for any $n \ge N_{s,R,T, K_s}$, the plug-in classifier
    $\widehat\Phi=\Phi(X;\hat b_1,\hat b_0)$ satisfies
    \begin{equation}\label{eq:mainrateinhomo}
    \sup_{\substack{(b_1, b_0) \in \mathcal{B}_{\mathrm{in}}(\rho_0, D_0) \\ b_0, b_1 \in \Sigma(s,R)}} \E_{\Samp_n}[\mathcal{E} (\widehat\Phi)] \;\lesssim\;
     v_n \; (\log n) \; \wedge \frac{v_n^2 \; (\log n)^2 }{D_0},
    \end{equation}
    where the implied constants in the inequalities depend only on $s$, $R$, $\rho_0$, $T$, $K_s$ and $C_s$.
\end{thm}
As $s \to \infty$ the exponents tend to those of the constant case in Theorem~\ref{thm:mainconstant} with $D = T \Delta$.
The transition between the two regimes is governed by the comparison of the drift separation $D$ with the estimation accuracy of the drifts, exactly as in the constant case.

\par\smallskip\noindent\textbf{Simulations.}\enspace
We consider four drift differences on $[0,1]$, chosen to cover the regimes of Assumption~\ref{ass:drifts} and the degenerate case of Example~\ref{ex:countermargin}.
As shown in~\eqref{eq:Chisquare} in the proof, the log-likelihood ratio is of the form $\sum_{j\ge 0} \lambda_j N_j^2$ with independent standard normal variables $N_j$ up to a shift.
The four cases differ in the sign of the weights $\lambda_j$.
Let $b_1(t) = -1 + a(t)/2$ and $b_0(t) = -1 - a(t)/2$ with
\begin{itemize}
    \item $\psi \le -c < 0$, regime $(+)$: $a(t) = 1 + t/2$. All weights $\lambda_j$ are positive.
    \item $\psi \equiv 0$, regime $(-)$: $a(t) = e^{2(t-1)}$. Only one weight is nonzero.
    \item $\psi > 0$, regime $(-)$: $a(t) = e^{3(t-1)}$. One $\lambda_0$ is positive and infinitely many are negative.
    \item Example~\ref{ex:countermargin} with $c=1/(2e)$: $\int_0^1 a(t) \diff t= 0$, $a$ discontinuous at $1/2$. Only two weights are nonzero and are of opposite sign.
\end{itemize}
In the first three Cases $D= \Vert a \Vert_{L^1}$ and $b_1, b_0$ are analytic, so that Assumption~\ref{ass:drifts} holds.
The drifts are estimated by a local-constant Euler quasi-likelihood estimator.
It is obtained by treating the drift coefficient as constant within a symmetric time window with half-bandwidth $0.01$ and maximizing the corresponding local Euler likelihood from $n$ paths per class, simulated on a grid of $1000$ time steps.
The excess risk is computed as in Section~\ref{sec:constantOU} with $30000$ test paths and $20$ training repetitions.

\begin{figure}[t]
    \centering
    \includegraphics[width=\linewidth]{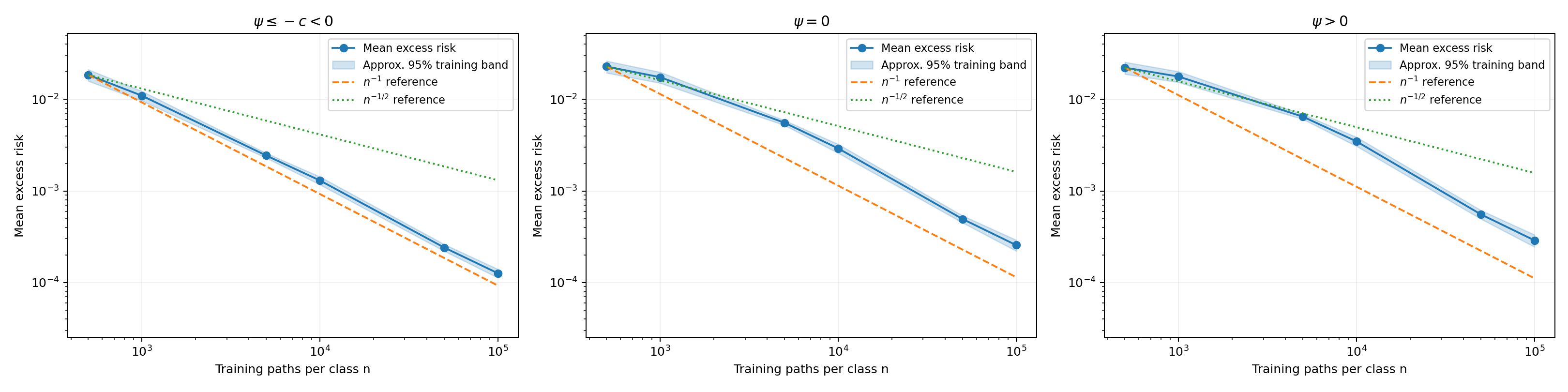}
    \caption{
    Excess risk for three time-inhomogeneous OU classification experiments:
    $\psi\le-c<0$ (left), $\psi=0$ (middle),
    $\psi>0$ (right).
    }
    \label{fig:time-inhomogeneous-comparison}
\end{figure}
In all three cases the excess risk decays at the fast rate $n^{-1}$ for large $n$.
This is faster than the error bound $v_n^2$ in Theorem~\ref{thm:maininhomo} for any finite $s$, because the drifts used here are analytic and the bias of the window estimator is negligible over the range of $n$ considered, so that the estimation error behaves parametrically.
The two cases with smaller separation ($D \approx 0.43$ in Case 2, $D \approx 0.32$ in Case 3) exhibit an initial phase of order $n^{-1/2}$ up to $n \approx 10^{3}$, followed by the $n^{-1}$ regime.
This is the transition between the two regimes of~\eqref{eq:mainrateinhomo}, which occurs when the estimation error falls below the separation.
For $D \approx 1.25$ in the first case, the fast regime is reached already at $n=500$.

\begin{figure}[t]
    \centering
    \includegraphics[width=0.8\linewidth]{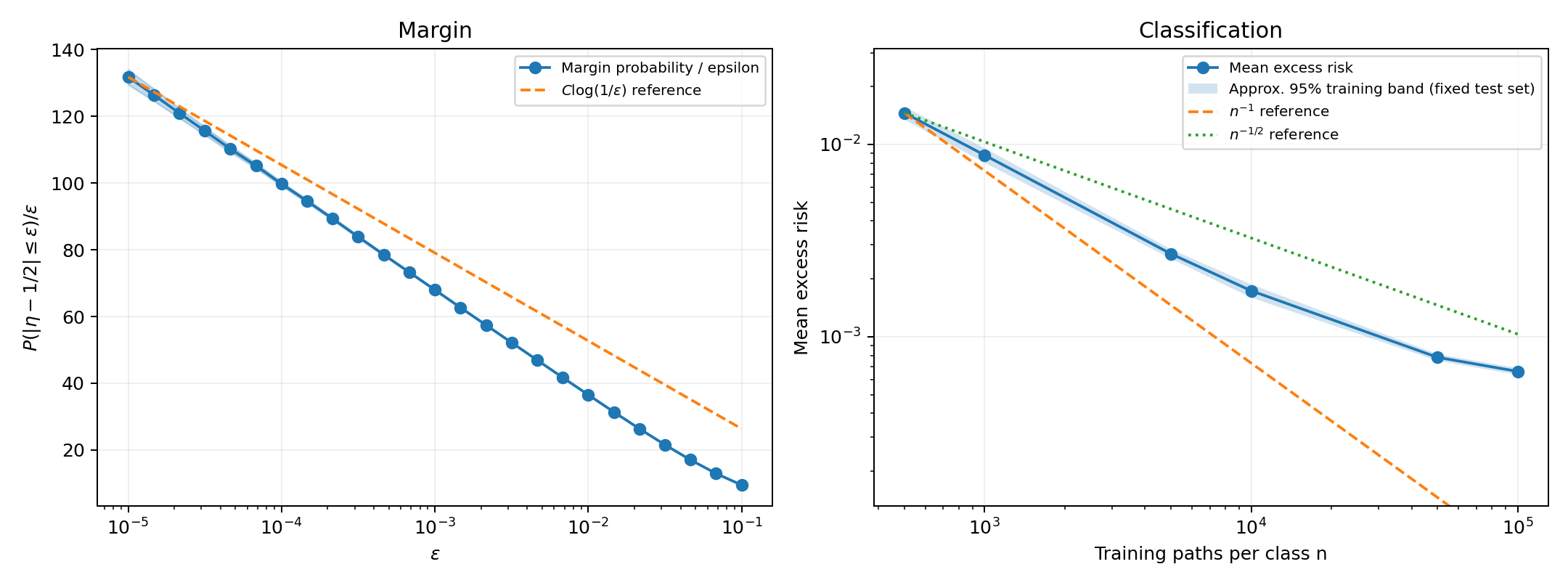}
    \small
    \caption{
    Margin behaviour and excess risk for Example~\ref{ex:countermargin} ($\int a =0$).
    Left: the empirical ratio $\prob(|\eta(X)-1/2|\le\varepsilon)/\varepsilon$.
    Right: mean excess risk.
    }
    \label{fig:countermargin}
\end{figure}
Example~\ref{ex:countermargin} lies outside Assumption~\ref{ass:drifts}.
It violates the nonnegativity, Lipschitz regularity and positive signed integral.
Figure~\ref{fig:countermargin} shows that the margin probability is at least of order $\varepsilon \log(1/\varepsilon)$, so the margin condition with $\alpha = 1$ fails.
The left panel is consistent with the growing ratio in the margin probability for the variable $Z=\lambda_+ \xi_1^2 -\lambda_- \xi_2^2$ in Lemma~\ref{lem:countermargin}.
Over the range $500 \le n \le 10^4$, the empirical excess risk follows a slope of about $-0.7$ and flattens beyond.
For $n \le 10^4$ the logarithmic factor in the margin varies strongly and degrades the rate from $n^{-1}$.
The flattening at larger $n$ is consistent with a fixed-window smoothing bias caused by the estimation term.
The window containing the discontinuity of $a$ at $1/2$ carries a bias, so the estimation error stops decreasing once its variance part falls below this bias.

\section{Proof Strategy and Key Lemmas} \label{sec-proof}
In this section we provide the proofs of all propositions, lemmas, and theorems stated in the previous sections.
Our approach relies on the Wiener chaos representation together with spectral theory.
We state the chaos and operator results in the general functional form, then specialize to the results with constant drifts where appropriate.

This section follows the order of the proof. Section~\ref{ssec:reduction} reduces the quantities of interest to random variables of the form $Z$; Section~\ref{ssec:chaos} represents $Z$ as a shifted generalized $\chi^2$ variable via the Wiener chaos expansion; Sections~\ref{ssec:factorization}--\ref{ssec:cJJ-evs} analyse the coefficients in this representation as eigenvalues of an operator; Section~\ref{ssec:density} derives the resulting distributional bounds and proves Propositions~\ref{prop:margin} and \ref{prop:margininhomo}; Section~\ref{ssec:estimation} proves the bound on the regression estimation error and Section~\ref{ssec:excess} proves the excess-risk bounds in Theorems~\ref{thm:mainconstant} and \ref{thm:maininhomo}.

\subsection{Reduction to distribution of the log-likelihood term} \label{ssec:reduction}
\par\smallskip\noindent\textbf{Margin term.}\enspace
As in earlier works on supervised classification with the Gaussian sequence model \cite{gadat2020}, we reduce the margin term via the following lemma.
\begin{lem} \label{lem:reducemargin}
	For $\varepsilon\le 1/16$, the margin probability can be bounded by the likelihood ratio via
	\begin{align*}
		& \prob_X \left( \left\vert \eta(X) - \frac{1}{2} \right\vert \le \varepsilon \right) \le \frac{1}{2} M_1 + \frac{1}{2} M_0,
	\end{align*}
	with
	\begin{align}
	M_1&:=\prob_{b_1}\!\left(\Bigl|\!\int_0^T\! a(t)X_t\diff W_t
		+\frac{1}{2}\!\int_0^T\! a^2(t)X_t^2\diff t\Bigr|\le 16\varepsilon
		\right),\label{eq:M1}\\
	M_0&:=\prob_{b_0}\!\left(\Bigl|\!\int_0^T\! a(t)X_t\diff W_t
		-\frac{1}{2}\!\int_0^T\! a^2 (t)X_t^2\diff t\Bigr|\le 16\varepsilon
		\right).\label{eq:M0}
	\end{align}
\end{lem}
\begin{proof}
We have
\begin{align*}
	& \prob_X\!\left(\Bigl|\eta(X)-\frac12\Bigr|\le\varepsilon\right) = \prob_X \left( \frac{\left\vert p_{b_1}(X)-p_{b_0}(X)\right\vert}{2(p_{b_1}(X) + p_{b_0} (X))} \le \varepsilon \right)
	\le \prob_X \left( \left\vert \frac{p_{b_1}(X)}{p_{b_0}(X)} - 1 \right\vert \le 8\varepsilon \right).
\end{align*}
The last inequality holds, since we have $\vert x -1 \vert \le 4\varepsilon \; / \; (1-2\varepsilon) \le 8\varepsilon$ from $\vert x -1 \vert \; / \; 2(x+1)\le \varepsilon$ for $\varepsilon \le 1/4$.

Conditionally on $Y=i$ we have $\diff X_t=b_i(t)X_t\diff t+\diff W_t$, and substituting this into the Girsanov density \eqref{eq:girsanov} gives
\begin{equation}\label{eq:logratio-margin}
	\log\frac{p_{b_1}}{p_{b_0}}(X)
	=\int_0^T a(t)X_t\diff W_t
	+\int_0^T a(t)\Bigl(b_i-\frac{b_1+b_0}{2}\Bigr)(t) X_t^2\diff t
	=\int_0^T a(t)X_t\diff W_t+\frac{\sigma_i}{2}\int_0^T a^2(t)X_t^2\diff t
\end{equation}
with $\sigma_1=+1$ and $\sigma_0=-1$.
Thus, conditioning on $Y$ and applying Lemma~\ref{lem:exp} for the exponential function with $8\varepsilon$ in place of $\varepsilon$, we obtain
\begin{equation*}
\prob_X\!\left(\Bigl|\eta(X)-\frac12\Bigr|\le\varepsilon\right)
\le\frac12 M_1+\frac12 M_0. \qedhere
\end{equation*}
\end{proof}

	\par\smallskip\noindent\textbf{Estimation term.}\enspace
	To control the estimation error, we apply the following lemma.
	\begin{lem} \label{lem:reduceest}
		The estimation error is bounded via
		\begin{align*}
			\prob_{\Samp_n,X} \left(\vert \widehat{\eta}(X)-\eta(X) \vert > \varepsilon  \right) \le \frac12 E_{11} + \frac12 E_{10} + \frac12 E_{01} + \frac12 E_{00},
		\end{align*}
		with
		\begin{align}
		E_{jj}&:=\prob_{\Samp_n, b_j}\!\left(\Bigl|\!\int_0^T\! (\hat b_j-b_j)(t)X_t\diff W_t
		-\frac12\!\int_0^T\! (\hat b_j-b_j)^2(t)X_t^2\diff t\Bigr|>\frac{\varepsilon}{6}\right),\label{eq:E1}\\
		E_{ij}&:=\prob_{\Samp_n, b_j}\!\left(\Bigl|\!\int_0^T\! (\hat b_i-b_i)(t) X_t\diff W_t
		+ \!\int_0^T\! \bigl(-\frac12(\hat b_i-b_i)^2+(b_j-b_i)(\hat b_i-b_i)\bigr) (t)X_t^2\diff t\Bigr|>\frac{\varepsilon}{6} \right), \; i \neq j. \label{eq:E0}
		\end{align}
	\end{lem}

	\begin{proof}
		We first use the likelihood ratio bound (see A.3 in \cite{gadat2020})
	\begin{align*}
		\big\vert \widehat{\eta}(X)-\eta(X) \big\vert \le 2 \left\vert \frac{p_{\hat{b}_1}(X)}{p_{b_1}(X)} -1\right\vert + \left\vert \frac{p_{\hat{b}_0}(X)}{p_{b_0}(X)} -1\right\vert,
	\end{align*}
	and the inclusion $\{2Z_1 + Z_2 > \varepsilon\} \subseteq \{2Z_1  > 2\varepsilon / 3 \} \cup \{ Z_2 > \varepsilon / 3\}$ for any random $Z_1, Z_2$. We obtain
	\begin{align*}
		\prob_{\Samp_n,X} \left(\vert \widehat{\eta}(X)-\eta(X) \vert > \varepsilon  \right) \le \prob_{\Samp_n,X} \left( \left\vert \frac{p_{\hat{b}_1}(X)}{p_{b_1}(X)} -1\right\vert > \frac{\varepsilon}{3}  \right) + \prob_{\Samp_n,X} \left( \left\vert \frac{p_{\hat{b}_0}(X)}{p_{b_0}(X)} -1\right\vert > \frac{\varepsilon}{3} \right).
	\end{align*}
Conditionally on $Y$,
\begin{align*}
	& \prob_{\Samp_n,X} \left( \left\vert \frac{p_{\hat{b}_1}(X)}{p_{b_1}(X)} -1\right\vert > \frac{\varepsilon}{3} \right) = \frac12 \prob_{\Samp_n,b_1} \left( \left\vert \frac{p_{\hat{b}_1}(X)}{p_{b_1}(X)} -1\right\vert > \frac{\varepsilon}{3} \right)
	+ \frac12 \prob_{\Samp_n,b_0} \left( \left\vert \frac{p_{\hat{b}_1}(X)}{p_{b_1}(X)} -1\right\vert > \frac{\varepsilon}{3} \right)\\
	& \prob_{\Samp_n,X} \left( \left\vert \frac{p_{\hat{b}_0}(X)}{p_{b_0}(X)} -1\right\vert > \frac{\varepsilon}{3} \right) = \frac12 \prob_{\Samp_n,b_1} \left( \left\vert \frac{p_{\hat{b}_0}(X)}{p_{b_0}(X)} -1\right\vert > \frac{\varepsilon}{3} \right)
	+ \frac12 \prob_{\Samp_n,b_0} \left( \left\vert \frac{p_{\hat{b}_0}(X)}{p_{b_0}(X)} -1\right\vert > \frac{\varepsilon}{3} \right).
\end{align*}
Using the log-likelihood function given $Y=j$
\begin{equation}\label{eq:logratio-general}
	\log\frac{p_{\hat{b}_i}}{p_{b_i}}(X)
	=\int_0^T \bigl(\hat{b}_i - b_i\bigr)(t)X_t\diff W_t
	+\int_0^T \Bigl(-\frac12 \bigl(\hat{b}_i - b_i\bigr)^2 + (b_j-b_i)\bigl(\hat{b}_i - b_i\bigr)
	\Bigr)(t)X_t^2\diff t,
\end{equation}
and applying Lemma~\ref{lem:exp} gives
\begin{align}
  \prob_{\Samp_n,b_j} \left( \left\vert \frac{p_{\hat{b}_i}(X)}{p_{b_i}(X)} -1\right\vert > \frac{\varepsilon}{3} \right) \le E_{ij}
\end{align}
for $0 < \varepsilon \le 3/4$ as upper bounds for the ratio probabilities.
\end{proof}
	Henceforth, we write $E_1$ to denote $E_{11}$ and $E_0$ to denote $E_{01}$, simplifying the expressions in the case $Y=1$. The terms under $Y=0$ follow analogously by symmetry.

The risk reduces to the four terms $M_1, M_0, E_1 (=E_{11}), E_0(=E_{01})$ in \eqref{eq:M1}, \eqref{eq:M0}, \eqref{eq:E1} and \eqref{eq:E0}. Inspecting their structure, we see that they are small-ball estimates or tail bounds for a stochastic quantity of the form in \eqref{eq:Z}
\begin{align*}
	Z =  \int_0^T u(t) X_t \diff W_t + \frac{1}{2}\int_0^T v(t) X_t^2 \diff t,
\end{align*}
given that $X_t$ is an Ornstein--Uhlenbeck process with drift function $b_i(t)$ and that $u \in L^2([0,T])$ and $v \in L^1([0,T])$ are deterministic given $Y$ and $\Samp_n$.

\subsection{Wiener chaos decomposition} \label{ssec:chaos}
Suppose $(X_t)_{t\in [0,T]}$ is an Ornstein--Uhlenbeck process with drift function $b_i (t) \in L^2([0,T])$, and $u \in L^2([0,T])$, $v  \in L^1([0,T])$ are deterministic coefficients that, together with $b_i$, are not simultaneously zero constant functions.
We briefly recall the elements of the Wiener chaos decomposition and the Malliavin derivative used below following the results of Nualart (2006) \cite{nualart2006} and Nourdin \& Peccati (2012) \cite{nourdin2012}. Every $Z \in L^2(\Omega, \sigma(W), \prob)$ admits the orthogonal expansion $Z = \sum_{q \ge 0}I_q(f_q)$ into multiple Wiener-It\^o integrals; see Proposition~\ref{prop-chaos}. The Malliavin derivative $D_t$ (which appears only with time subscripts) and its domain $\mathbb{D}^{1,2}$ are defined as in Nualart (2006) \cite{nualart2006}.

	The random variable $Z$ defined in \eqref{eq:Z} lies in $ L^2(\Omega)$ and has the Malliavin derivatives
	\begin{align}
		D_{t} Z= & \, u(t) X_t + \int_t^T u(r)\, D_t X_r \diff W_r + \int_t^T v(r)\, X_r \, D_t X_r \diff r \nonumber\\
		= & \, u(t) \int_0^t \exp \left(\int_\rho^t b_i(\tau) \diff \tau \right) \diff W_\rho + \int_t^T u(r) \exp \left( \int_t^r b_i (\tau) \diff \tau \right)   \diff W_r \nonumber\\
		& + \int_t^T v(r) \exp \left( \int_t^r b_i (\tau) \diff \tau \right)  \int_0^r \exp \left(\int_\rho^r b_i (\tau) \diff \tau \right) \diff W_\rho \diff r, \label{eq:DZ}\\
		D_{s} D_{t} Z
		= & \; u(t \vee s) \exp \left( \int_{t\wedge s}^{t \vee s} b_i (\tau) \diff \tau \right) + \int_{t \vee s}^T v(r) \exp \left( \int_s^r b_i (\tau) \diff \tau \right) \exp \left( \int_t^r b_i (\tau) \diff \tau \right) \diff r \label{eq:DDZ}
	\end{align}
	for $s,t \in [0,T]$. Here $D^k Z= 0$ for $k \ge 3$, since $D^2 Z$ is deterministic; thus $Z-\E[Z]$ lies in the second chaos.
	Following Proposition~\ref{prop-chaos} with $\E[D_t Z] = 0$,
	\begin{align*}
		Z = \E [Z] + \int_0^T \int_0^{t} \E [D_s D_t Z] \diff W_{s} \diff W_{t}.
	\end{align*}

	The symmetric kernel $\E [D_s D_t Z] \in L^2([0,T]^2)$ defines a compact, self-adjoint Hilbert--Schmidt operator on $L^2([0,T])$
	\begin{align} \label{eqHS}
		A g (t) = \int_0^T \E [D_s D_t Z] g(s) \diff s.
	\end{align}
	We write $\setl \lambda_j: \; j \ge 0 \setr$ for the eigenvalues of $A$ (counted with multiplicity, $\vert \lambda_0 \vert \ge \vert \lambda_1 \vert \ge \vert \lambda_2 \vert \ge \cdots$) and $\setl e_j: \; j \ge 0 \setr$ for the corresponding orthonormal eigenfunctions in $L^2([0,T])$.
	$\E [D_s D_t Z]$ admits the expansion
	\begin{align*}
		\E [D_s D_t Z] = \sum_{j = 0}^{\infty} \lambda_j \; e_j (t) \cdot e_j(s) \quad\text{in }L^2([0,T]^2),
		\qquad
		\Vert\lambda\Vert_2^2=\sum_{j=0}^\infty\lambda_j^2=\Vert \E[D_s D_t Z]\Vert^2_{L^2([0,T]^2)}<\infty.
	\end{align*}
	The double stochastic integral of $\E [D_s D_t Z]$ is defined by
	\begin{align*}
		\int_0^T\int_0^t \E [D_s D_t Z] \diff W_s \diff W_t
		= \sum_{j = 0}^{\infty} \lambda_j \int_0^T\int_0^t e_j (t) \cdot e_j(s) \diff W_s \diff W_t
	\end{align*}
	Applying It\^o's lemma to $\int_0^t e_j(s) \diff W_s$ at time $t=T$, we obtain
	\begin{align*}
		\left( \int_0^T e_j(s) \diff W_s \right)^2 = 2 \int_0^T \int_0^t e_j(s) \diff W_s e_j(t) \diff W_t + 1.
	\end{align*}
	$N_j := \int_0^T e_j(s) \diff W_s$, $j \ge 0$ are independent standard normal variables in $\R$. Thus,
	\begin{align} \label{eq:Chisquare}
		Z &
		= \E [Z] + \frac{1}{2} \sum_{j = 0}^{\infty} \lambda_j \left( N_j^2 - 1\right),
		\qquad \var(Z)=\frac12\sum_{j=0}^\infty\lambda_j^2 ,
	\end{align}
	with the series converging in $L^2$.
	$Z$ is thus a quadratic form of Gaussian random variables in $L^2$ with shifted expectation. The density function of $Z$ exists if at least one $\lambda_j \neq 0$. However, it is difficult to obtain a closed expression for the density function if $\lambda_j \neq 0$ for more than two $j$'s. The eigenvalues are indexed starting from 0, since $\lambda_0$ will be used to denote the leading eigenvalue in the next sections.

	\begin{rem}
		For any Gaussian $X$, all higher-order ($k \ge 3$) Malliavin derivatives of $Z$ vanish. Therefore, $Z$ can always be represented through the second-order Wiener chaos and admits an expansion in terms of generalized chi-squared variables as in \eqref{eq:Chisquare}. This observation implies that the subsequent analysis extends to more general Gaussian processes,
		not only to the Ornstein--Uhlenbeck case.
	\end{rem}
	\begin{rem} \label{rem:malliavin}
	The classical Malliavin density formula in Nualart (2006) \cite{nualart2006} requires $Z \in \mathbb{D}^{1,2}$ with $\Vert DZ\Vert_H > 0$ a.s., together with the integrability condition
	\begin{equation*}
	\frac{DZ}{\Vert DZ\Vert_H^2} \in \mathrm{Dom}(\delta) \subset L^2(\Omega; H),
	\end{equation*}
	so that the divergence $\delta\!\bigl(DZ/\Vert DZ\Vert_H^2\bigr)$ is well defined and the density admits the Nualart representation (and the associated upper bound).
	This condition fails already in constant case with $b_1+b_0 = 0$.
	We first carry out the computation for $X=W$.
	\begin{equation*}
		Z = \int_0^T W_s \diff W_s = \frac12\bigl(W_T^2 - T\bigr).
	\end{equation*}
	Then
	\begin{equation*}
		D_t Z = W_t + \int_0^T D_t W_s \diff W_s = W_t + \int_0^T \indic (t \le s) \diff W_s = W_T \indic (0 \le t \le T).
	\end{equation*}
	Thus,
	\begin{equation*}
		\Vert DZ\Vert_H^2 = \int_0^T W_T^2 \diff t = T\,W_T^2,\qquad
	\frac{D_t Z}{\Vert DZ\Vert_H^2} = \frac{1}{T\,W_T}
	\end{equation*}
	is not integrable. Hence the classical density upper bound is unavailable, even though $Z$ itself possesses a smooth density on $(-T/2, \infty)$.
	The same computation applies to the log-likelihood ratio of two constant drifts with $b_1+b_0=0$, since $\ell = \frac{\Delta}{2}(X_T^2 -T)$ and $X_T$ is a nondegenerate centred Gaussian under $\prob_{b_i}$.
	The same holds for our random variable in \eqref{eq:Z} whenever the operator has rank at most 2.
	\end{rem}

\subsection{Factorization of the operator $A$} \label{ssec:factorization}
We first analyse the structure of the eigenvalue spectrum of $A$. Recall the notations in Section~\ref{sec:setting}.

\par\smallskip\noindent\textbf{Notations.}\enspace Throughout the eigenvalue analysis we work conditionally on a fixed label $Y=i\in\{0,1\}$. We write
\begin{itemize}
\item $\omega(t) := \omega_i (t) = \exp\!\bigl(-\int_0^t b_i(\tau)\,\diff\tau\bigr)$, with bounds $c_\omega \le \omega(t)\le C_\omega$;
\item for the margin terms: $u = a = b_1 - b_0$ and $v = \sigma_i a^2$ with $\sigma_1 = +1$, $\sigma_0 = -1$;
\item the kernel multiplicative function
\begin{equation*}
	F(r) : = u(r) \omega(r)^{-2} + \int_r^T v(\rho) \omega(\rho)^{-2} \diff \rho.
\end{equation*}
When $u$ is absolutely continuous and $v\in L^1$, $F$ is absolutely continuous with
\begin{equation}\label{eq:Fprime}
	F'(r)=\bigl(u'(r)+2b_i(r)u(r)-v(r)\bigr)\,\omega(r)^{-2},
	\qquad\text{for a.e.\ } r\in[0,T].
\end{equation}
\end{itemize}
For the margin specialization $(u,v)=(a,\sigma_ia^2)$ we obtain from \eqref{eq:Fprime}
\begin{equation}\label{eq:Fprime-margin}
	F'(r)=\bigl(a'+a\,(2b_i-\sigma_i a)\bigr)(r)\,\omega(r)^{-2}
	=\bigl(a'+ a\,(b_0+b_1)\bigr)(r)\,\omega(r)^{-2}
	=\psi(r)\,\omega(r)^{-2}.
\end{equation}
The label $i$ and the sign $\sigma_i$ cancel.
The Volterra multiplier $-F'=-\psi\,\omega_i^{-2}$ is driven by the same function $\psi$ from \eqref{eq:psi} for both classes, and its sign pattern is that of $-\psi$.
This identifies the assumption-level function $\psi$ of Assumption~\ref{ass:drifts} with the object $F'\omega^2$.

\begin{lem}\label{lem:Afactorization}
	Define the Volterra operator on $L^2([0,T])$ by
	\begin{equation*}
		Jg(t) = \int_0^t g(s) \diff s
	\end{equation*}
	and the multiplication operator for function $h$ as
	\begin{equation*}
		M_h g (t)= h(t) g(t).
	\end{equation*}
	$A$ can be decomposed as
	\begin{equation} \label{eq:Adecomp}
	A = M_{\omega} \left(R+J^* M_{-F'} J \right) M_{\omega},
	\end{equation}
	where $Rg = F(T) \int_0^T g(s) \diff s \in \R$ is a rank-one operator when $a(T) > 0$.
\end{lem}
\begin{proof}
	Using the notations above, we have the reformulation
\begin{equation*}
    \begin{aligned}
    Ag(t) & = \int_0^T \left[ u(s\vee t) \exp \left(\int_{s \wedge t}^{s \vee t} b_i(\tau) \diff \tau \right) + \int_{s\vee t}^T v(\rho) \exp \left(\int_s^\rho b_i(\tau) \diff \tau + \int_t^\rho b_i(\tau) \diff \tau\right) \diff \rho\right] g(s) \diff s \\
    & = \omega(t) \int_0^T F(s \vee t)\, \omega (s)\, g(s) \diff s \\
	& = \omega (t) \int_0^T \left(F(T) - \int_{s \vee t}^T F'(r) \diff r \right) \omega (s) g(s) \diff s ,
    \end{aligned}
\end{equation*}
where $F$ is as declared above.
If $a(T) \neq 0$, then $R$ is a rank-one operator; otherwise, $R = 0$.
With these conventions, we obtain the stated decomposition in \eqref{eq:Adecomp}.
\end{proof}

\begin{rem}
$R$ is an at most rank-one operator which describes the boundary condition.
It is self-adjoint, since for any $h \in L^2([0,T])$,
\begin{equation*}
    \langle Rg, h \rangle = \int_0^T F(T) \int_0^T g(s) \diff s \; h(t) \diff t = \int_0^T g(s) \int_0^T F(T)  h(t) \diff t \diff s = \langle g, Rh \rangle.
\end{equation*}
Also, $R$ is positive semidefinite if $F(T) \ge 0$; if $F(T) > 0$, its only non-zero eigenvalue is $T F(T)$, with the constant eigenfunction $g \equiv 1 / \sqrt{T}$.
This is guaranteed by $F(T) = a(T)\, \omega(T)^{-2} \ge 0$ in Assumption~\ref{ass:drifts}.

In the sequel, we analyse the eigenvalues of $A$ for $a(T) >0$. If $F(T) = 0$, then $R=0$ and only the Volterra part in \eqref{eq:Adecomp} survives. The bounds below remain valid in this degenerate case: In regime~$(+)$ the bounds for $j \ge 1$ and the lower bound in Corollary~\ref{cor:posregime} below are unchanged, and in regime~$(-)$ the operator $K=-V$ has no positive eigenvalue as in Lemma~\ref{lem:negspec}.
\end{rem}

\begin{rem} \label{rem:itocorrection}
	It\^o's formula gives
	\begin{align*}
		\diff \left(\frac{X_t^2}{2}\right) = X_t \diff X_t + \frac{1}{2} \diff t = b_i(t) X_t^2 \diff t + X_t \diff W_t + \frac{1}{2} \diff t.
	\end{align*}
	Multiplying by $a(t)$, rearranging the terms, integrating on $[0,T]$ give
	\begin{align*}
		\int_0^T a(t) X_t \diff W_t = \int_0^T a(t) \diff \left(\frac{X_t^2}{2}\right) - \int_0^T a(t) b_i(t) X_t^2 \diff t - \frac{1}{2} \int_0^T a(t) \diff t.
	\end{align*}
	Integrating by parts in the pathwise integral with weak derivative $a'$ and adding the Lebesgue term yield
	\begin{align} \label{eq:Zito}
		Z = \underbrace{\frac{a(T)}{2} X_T^2 - \frac{1}{2} \int_0^T \bigl(a' + 2a b_i -\sigma_i a^2 \bigr) (t) X_t^2 \diff t}_{=: Q} - \frac{1}{2} \int_0^T a(t) \diff t.
	\end{align}
	$Q$ is a purely quadratic functional of the centred Gaussian path $X$ with no linear or constant parts. We get $Q(X) \overset{d}{=} \frac{1}{2} \sum_j \lambda_j N_j^2$. The second term in $Q$ has the same prefactor integrand $a' + a(2b_i - \sigma_i a)$ as in $F'$. The deterministic integral $-\frac{1}{2}\int_0^T a(t) \diff t$ corresponds exactly to the It\^o correction and is the shift $D$ discussed in Lemma~\ref{lem:shift} in Section~\ref{ssec:density}.
\end{rem}

Let $\lambda_k(B)$ denote the ordered eigenvalues of a Hilbert--Schmidt operator $B$, $ | \lambda_1(B) | \ge | \lambda_2(B) | \ge \cdots$.
For the indefinite operator $K := R + J^* M_{-F'} J$, we compare the spectrum of $K$ with that of $A = M_{\omega} K M_{\omega}$.
The comparison follows by applying Lemma~\ref{lem:congruence} in Appendix~\ref{sec:appendix} with $H = L^2([0,T])$, $S = M_\omega$ and $K = R + J^* M_{-F'} J$. The function $\omega$ is continuous and strictly positive on $[0,T]$, so $0 < c_\omega \le \omega \le C_\omega < \infty$ and
\begin{equation*}
	c_\omega \Vert g \Vert\le \Vert M_\omega g \Vert\le C_\omega \Vert g \Vert, \qquad g \in L^2([0,T]),
\end{equation*}
with $M_\omega$ self-adjoint and $M_\omega^{-1} = M_{1/\omega}$. Moreover $R + J^* M_{-F'} J$ is compact and self-adjoint, since $R$ is self-adjoint of rank one and $J$ is Hilbert--Schmidt while $M_{-F'}$ is bounded and self-adjoint under Assumption~\ref{ass:drifts}. Lemma~\ref{lem:congruence} (iii) yields
\begin{equation}\label{eq:omega-sandwich}
	c_{\omega}^2\, |\lambda_k(K) | \le | \lambda_k(A) | \le C_{\omega}^2\, |\lambda_k(K) |,\qquad k\ge0,
\end{equation}
in magnitude ordering, and the congruence by $M_\omega$ also preserves the numbers of positive and negative eigenvalues.

\begin{lem} \label{lem:JMJ}
If $0<c_F\le-F'\le C_F$ on $[0,T]$, then for every $k\ge1$,
\begin{equation*}
  c_F\,\lambda_k(J^*J)\le\lambda_k(J^*M_{-F'}J)\le C_F\,\lambda_k(J^*J),
\end{equation*}
and thus,
\begin{equation}\label{eq:RJMJ}
  \lambda_k(R + c_F J^*J)\le\lambda_k(R+J^*M_{-F'}J)
   \le\lambda_k(R + C_F J^*J).
\end{equation}
\end{lem}
\begin{proof}
For $\Vert g\Vert=1$, $\langle J^*M_{-F'}Jg,g\rangle=\langle M_{-F'}Jg,Jg\rangle
\in[c_F\langle J^*Jg,g\rangle, C_F\langle J^*Jg,g\rangle]$.
The case $k=1$ follows by taking suprema of the Rayleigh quotient; for $k>1$ apply the Courant--Fischer min--max principle \cite{ReedSimon1978, Bhatia1997} to both sides.
The bound \eqref{eq:RJMJ} follows from the operator order $c_FJ^*J\preceq J^*M_{-F'}J\preceq C_FJ^*J$ and Weyl's monotonicity principle in Weyl (1912) \cite{weyl1912}.
\end{proof}

\subsection{Eigenvalue analysis of the operator $A$} \label{ssec:cJJ-evs}
In the previous sections, we represented the log-likelihood ratio variable as a shifted weighted chi-squared distribution, with the weights determined by the eigenvalues of the operator $A$ in \eqref{eqHS}. Our ultimate goal is to derive density bounds for this random variable via the spectral representation.
To this end, we exploit the special structure of $R+c J^* J$ to characterize the range and properties of its eigenvalues. These spectral estimates are then transferred to the operator $R+J^*M_{-F'}J$ via the congruence relation in \eqref{eq:omega-sandwich} and the bounds in Lemma~\ref{lem:JMJ}, covering the two different regimes of the sign of $-F'$ in Assumption~\ref{ass:drifts}.

In a subsequent section, we specialize to the constant drift case, where the additional structure allows us to obtain more explicit two-sided bounds on the eigenvalues.

\subsubsection{Eigenvalue analysis for the operator $R+J^*M_{-F'}J$}
By \eqref{eq:RJMJ} the leading eigenvalue of $R+J^*M_{-F'}J$ is controlled by that of $R+cJ^*J$ with $c\in\{c_F,C_F\}$. We solve the latter exactly.
\begin{lem} \label{lem:posODE}
For $c>0$ the following are equivalent:
\begin{enumerate}
	\item $h=(R+cJ^*J)g$ for some
$g\in C([0,T])$;
	\item $h\in C^2([0,T])$ solves
\begin{equation}\label{eq:ODE}
  h''(t)=-c\,g(t),\qquad h'(0)=0,\qquad F(T)\,h'(T)=-c\,h(T).
\end{equation}
\end{enumerate}
\end{lem}
\begin{proof}
$(1)\Rightarrow(2)$: from
\begin{equation*}
	(R+cJ^*J)g(t)=F(T)\int_0^T g+c\int_0^T(T-s\vee t)g
\end{equation*}
we get
\begin{equation*}
	h'(t)=-c\int_0^t g.
\end{equation*}
Thus,
\begin{equation*}
	h''=-cg, \quad h'(0)=0, \quad F(T) h'(T)=-c h(T).
\end{equation*}
$(2)\Rightarrow(1)$: setting
$g=-h''/c$ and integrating by parts give $h=(R+cJ^*J)g$.
\end{proof}

The eigenvalues of $J^*J$ are classical: the eigenfunctions of $J^*Jg(t)=\int_0^T(T-s\vee t)g(s) \diff s$ solves $-\phi''=\mu\phi$, $\phi'(0)=0$, $\phi(T)=0$, whence
\begin{equation}\label{eq:JJ-evs}
  \lambda_k(J^*J)=\frac{T^2}{\bigl((k+\frac12)\pi\bigr)^2}\asymp T^2k^{-2},
  \qquad k\ge0.
\end{equation}

\begin{lem}[Eigenvalue bounds for $R+cJ^*J$, $c > 0$]\label{lem:posspec}
For $c > 0$, we have with the notation of \eqref{eq:transc} and $F(T) > 0$ below
\begin{enumerate}
  \item[\textup{(a)}] For $j\ge1$,
    $\dfrac{cT^2}{((j+\frac12)\pi)^2}<\lambda_j (R + c J^* J)<\dfrac{cT^2}{(j\pi)^2}$,
    so $\lambda_j (R + c J^* J) \asymp cT^2j^{-2}$.
  \item[\textup{(b)}] For $j=0$, since $\tan x\ge x$ on $(0,\frac\pi2)$,
    $x_0\tan x_0\ge x_0^2$ gives $x_0\le\sqrt{cT/F(T)}$, hence
    \begin{equation}\label{eq:lam0}
      \lambda_0 (R + c J^* J) =\frac{cT^2}{x_0^2}\ge T\,F(T),
    \end{equation}
    and also $\lambda_0>4cT^2/\pi^2$.
\end{enumerate}
The strict inequalities become equalities at $F(T) =0$.
\end{lem}

\begin{proof}
Assume first $F(T)>0$, so that $\gamma=cT/F(T) > 0$.
The operator $R+cJ^*J$ is compact and self-adjoint, and it is positive definite.
In particular $0$ is not an eigenvalue and every eigenvalue is positive.
Let $(R+cJ^*J)g=\lambda g$ with $g\neq0$, $\lambda>0$.
Then $g=\lambda^{-1}(R+cJ^*J)g$ is continuous, because $Rg$ is a constant function and $J^*Jg(t)=\int_0^T(T-s\vee t)\,g(s)\diff s$ is continuous in $t$.
Hence Lemma~\ref{lem:posODE} applies to $h:=\lambda g$ and $g$ solves,
\begin{equation} \label{eq:evpcJJ}
	g''=\mu g,\qquad g'(0)=0,\qquad F(T)\,g'(T)=-c\,g(T),
	\qquad \mu=-\frac{c}{\lambda}<0 .
\end{equation}
Write $\mu=-\alpha^2$ with $\alpha>0$. The equation with $g'(0)=0$ forces $g(t)=A\cos(\alpha t)$, $A\neq0$, and the right boundary condition becomes $F(T)\,\alpha\sin(\alpha T)=c\cos(\alpha T)$.
Here we have $\cos(\alpha T)\neq0$, since otherwise
$\sin(\alpha T)=\pm1$ and the left-hand side would equal $\pm F(T)\alpha\neq0$ while the right-hand side vanishes.
Dividing by $F(T)\,\alpha\cos(\alpha T)$ and setting
$x:=\alpha T$ yields
\begin{equation} \label{eq:transc}
x \tan x= \frac{cT}{F(T)} =: \gamma >0 \quad \text{ with } \quad
\lambda= \frac{c}{\alpha^2} = \frac{cT^2}{x^2}.
\end{equation}
Conversely, for every root $x>0$ of \eqref{eq:transc} the function $g(t)=\cos(xt/T)$ solves the boundary value problem,
so $cT^2/x^2$ is an eigenvalue by Lemma~\ref{lem:posODE}.
Each eigenvalue is simple, because for fixed $\mu$ the solutions of $g''=\mu g$ with $g'(0)=0$ form a one-dimensional space.

Wherever $\tan x\le0$ is defined, we have $x\tan x\le0<\gamma$, so all positive
roots of \eqref{eq:transc} lie in the branches $(j\pi,(j+\tfrac12)\pi)$, $j\ge0$,
on which $\tan x>0$.
On each such branch the map $x\mapsto x\tan x$ is continuous
and strictly increasing, since its derivative is $\tan x+x/\cos^2x>0$.
It increases from $0$ at $x=j\pi$ to $+\infty$ as $x\uparrow(j+\tfrac12)\pi$.
Hence it takes the value $\gamma$ at exactly one interior point $x_j\in(j\pi,(j+\tfrac12)\pi)$, and both
inclusions are strict.
Since $j\mapsto x_j$ is strictly increasing, the eigenvalues $cT^2/x_j^2$ are strictly decreasing in $j$.
Therefore the $j$-th largest eigenvalue is
\begin{equation*}
	\lambda_j(R+cJ^*J)=\frac{cT^2}{x_j^{2}},\qquad j\ge0 .
\end{equation*}

We can compute the bounds for the eigenvalues: (a) For $j\ge1$, the strict inclusion $x_j\in(j\pi,(j+\tfrac12)\pi)$ gives
\begin{equation*}
	\frac{cT^2}{((j+\tfrac12)\pi)^2}<\lambda_j(R+cJ^*J)<\frac{cT^2}{(j\pi)^2},
\end{equation*}
and in particular $\lambda_j(R+cJ^*J)\asymp cT^2j^{-2}$.
(b) For $j=0$ we have $x_0\in(0,\tfrac\pi2)$, so $\lambda_0=cT^2/x_0^2>4cT^2/\pi^2$. Moreover $\tan x\ge x$ on $(0,\tfrac\pi2)$
gives $\gamma=x_0\tan x_0\ge x_0^2$, i.e.\ $x_0\le\sqrt\gamma=\sqrt{cT/F(T)}$,
and therefore
\begin{equation*}
	\lambda_0(R+cJ^*J)=\frac{cT^2}{x_0^{2}}\ge\frac{cT^2}{\gamma}=T\,F(T).
\end{equation*}

Finally, if $F(T)=0$, then $R=0$ and $\lambda_j(cJ^*J)=cT^2/((j+\tfrac12)\pi)^2$ directly by \eqref{eq:JJ-evs}.
All bounds of (a) and (b) hold in this case as boundary cases (with $T\,F(T)=0$).
\end{proof}

Combining Lemma~\ref{lem:posspec} with Lemma~\ref{lem:JMJ} and the congruence sandwich \eqref{eq:omega-sandwich} gives the complete description of the spectrum in the positive regime.

\begin{cor}[Spectrum in regime $(+)$]\label{cor:posregime}
	Assume $0<c_F\le-F'\le C_F$ a.e.\ and $F(T) \ge 0$. Then $K=R+J^*M_{-F'}J\succ0$, all eigenvalues of $A$ are positive, and
	\begin{equation}\label{eq:posregime-evs}
		\lambda_0(A)\ge c_\omega^2\,\max\Bigl\{T\,F(T),\ \frac{4c_FT^2}{\pi^2}\Bigr\},
		\qquad
		\frac{c_\omega^2\, c_F\, T^2}{((j+\frac12)\pi)^2}\le\lambda_j(A)
		\le\frac{C_\omega^2\, C_F\, T^2}{(j\pi)^2},\quad j\ge1 .
	\end{equation}
	Moreover $\lambda_0(A)\le C_\omega^2\bigl(TF(T)+4C_FT^2/\pi^2\bigr)$ by Weyl's monotonicity inequality for sums of positive semidefinite operators \cite{weyl1912}.
\end{cor}

The boundary value problem of Lemma~\ref{lem:posODE} and the transcendental equation \eqref{eq:transc} were derived under $c>0$.
When $-F'<0$, the Volterra multiplier in the operator is negative.
The constant in the simplification $R+cJ^*J$ is negative.
Also, $\gamma=cT/F(T)<0$, so the above bounds are unavailable.
We give a self-contained replacement that uses only the rank-one structure of $R$ and the spectral sandwich of Lemma~\ref{lem:JMJ}.

Suppose $-F'\le0$ on $[0,T]$ and write $G:=F'\ge0$, so that
\begin{equation}\label{eq:Bdef}
  K:=R+J^*M_{-F'}J=R-V,\qquad V:=J^*M_G J\succeq0,\quad G=F'\ge0 .
\end{equation}
Here $R$ is rank one and positive (recall $F(T)>0$), and $V$ is compact, positive semidefinite, with eigenvalues controlled by Lemma~\ref{lem:JMJ}: if $0<c_G\le G\le C_G$ then
\begin{equation}\label{eq:Kev}
  c_G\,\lambda_k(J^*J)\le\mu_k(V)\le C_G\,\lambda_k(J^*J),\qquad
  \mu_k(V)\asymp k^{-2}\ (k\ge1),
\end{equation}
writing $\mu_0(V)\ge\mu_1(V)\ge\cdots \ge 0$ for the eigenvalues of $V$.

\begin{lem}[Eigenvalue bounds in regime $(-)$]\label{lem:negspec}
Let $K=R-V$ as in \eqref{eq:Bdef}, with $R=T F(T) \,\varphi\varphi^*$ rank-one positive with $TF(T)>0$, $\varphi:=\mathbf 1/\sqrt T$, $\Vert\varphi\Vert=1$ and $V\succeq0$ compact.
Let $\mu_0\ge\mu_1\ge\cdots\ge0$ be the eigenvalues of $V$, and let $\nu_1(K) \ge \nu_2(K) \ge \cdots$ denote the magnitudes of the negative eigenvalues of $K$, counted with multiplicity (with $\nu_j(K) := 0$ beyond their number).
\begin{enumerate}
  \item[\textup{(i)}] $\lambda_1(K)\le0$; i.e.\ $K$ has at most one positive
    eigenvalue.
  \item[\textup{(ii)}] The negative eigenvalues of $K$ interlace those of $V$:
    \begin{equation}\label{eq:interlace}
      \mu_{k+1}\;\le\; \nu_{k+1}(K) \;\le\;\mu_k,
      \qquad k\ge0,
    \end{equation}
	whether $K$ has one or no positive eigenvalue. In particular $\nu_j (K)\to 0$ and inherits the decay rate of $\mu_j$.
\end{enumerate}
\end{lem}

The proof is a Courant--Fischer argument using only the rank-one structure of $R$ and the positivity of $V$.
On $\varphi^\perp$ the quadratic form of $K$ is nonpositive, which caps the number of positive eigenvalues at one, while $-K=V-R$ is a rank-one downward perturbation of $V$, whose min--max values interlace the $\mu_k$.
The details are deferred to Appendix~\ref{sec:appendix}.

\subsubsection{The eigenvalue problem in the constant case} \label{ssec:eigenconst}
For the OU process with constant drift $b_i \in \R$ and constant $u=a$, $v=\sigma_ia^2$, we have
\begin{align*}
	F(T) = a e^{2 b_i T}, \quad F'(r) = a(b_0+b_1)\, e^{2b_i r}, \quad \omega_i (r)= e^{-b_i r}.
\end{align*}
Assume without loss of generality $a=\Delta>0$ and set
\begin{equation*}
	\kappa:=-(b_0+b_1),
\end{equation*}
so that $-F'(r)=\Delta\kappa\, e^{2b_ir}$. $\kappa$ does not depend on the label $i$.
For $b_i \neq 0$, $A$ has the explicit formulation
\begin{equation*}
	Ag(t) =\int_0^T\!\Bigl[\bigl(a-\frac{v}{2b_i}\bigr)e^{b_i|t-s|}
        +\frac{v}{2b_i}e^{b_i(2T-s-t)}\Bigr]g(s)\diff s,
\end{equation*}
using $\int_{s\vee t}^T e^{b_i(2r-s-t)}\diff r =\frac{1}{2b_i}\bigl(e^{b_i(2T-s-t)}-e^{b_i|t-s|}\bigr)$;
otherwise
\begin{equation*}
	Ag(t) =\int_0^T\! \bigl(a + v\,(T - s \vee t)\bigr)g(s)\diff s.
\end{equation*}

In the constant case we have
\begin{equation} \label{eq:EZconstant}
 \E [Z] = \frac{v}{2} \int_0^T \E_i[X_s^2] \diff s =
	\left\{ \begin{matrix}
		\frac{v}{8b_i^2} \left(e^{2b_iT} - 1 - 2b_i T \right), &  b_i \neq 0, \\
		\frac{vT^2}{4}, & b_i = 0.
	\end{matrix} \right.
\end{equation}
The bounds of $\omega$ are $c_{\omega,i}=1\wedge e^{-b_iT}$ and $C_{\omega,i}=1\vee e^{-b_iT}$, so that $e^{2b_ir}=\omega(r)^{-2}\in[C_{\omega,i}^{-2},c_{\omega,i}^{-2}]$.
We have the following spectrum result for the random variable $Z$.
\begin{prop}\label{prop:const-spectrum}
With an abuse of notation, write $c_\omega=c_{\omega,i}$, $C_\omega=C_{\omega,i}$ within this proposition.
\begin{enumerate}
  \item[\textup{(i)}]
  If $\kappa > 0$, then $0<c_F\le-F'\le C_F$ with $c_F = \Delta \kappa\, C_{\omega}^{-2}$, $C_F = \Delta \kappa\, c_{\omega}^{-2}$. Hence, by Corollary~\ref{cor:posregime}, all eigenvalues of $A$ are positive,
  \begin{equation*}
	c_{\omega}^2\, T \Delta\, e^{2b_i T}\le\lambda_0(A) \le C_{\omega}^2 T\Delta \left(  e^{2b_i T} + \frac{4 \kappa T}{\pi^2 c_{\omega}^2}\right),
  \end{equation*}
  and for $j \ge 1$
\begin{equation*}
	\frac{c_\omega^2\, \Delta \kappa\, T^2}{C_{\omega}^2 \bigl((j+\frac{1}{2})\pi \bigr)^2} \le\lambda_j(A)
    \le \frac{C_\omega^2\, \Delta \kappa\, T^2}{c_{\omega}^2 \bigl(j\pi \bigr)^2} .
\end{equation*}
Hence $\lambda_j(A)\asymp\Delta\kappa\,j^{-2}$ for $j\ge1$.

\item[\textup{(ii)}]
If $\kappa = 0$, then $K=R$ and
  \begin{equation*}
	c_{\omega}^2\, T \Delta\, e^{2b_i T}\le\lambda_0(A) \le C_\omega^2\,T\Delta\, e^{2b_iT}, \quad \lambda_j (A) = 0 , \; j\ge 1.
  \end{equation*}
\item[\textup{(iii)}]
If $\kappa < 0$, then $-F'\le0$ with $0<c_G\le F'\le C_G$ for $c_G = \Delta |\kappa|\, C_{\omega}^{-2}$, $C_G = \Delta |\kappa|\, c_{\omega}^{-2}$. The operator $A$ has at most one positive eigenvalue, and the magnitudes $\nu_j(A)$ of its negative eigenvalues satisfy
\begin{equation*}
	 \frac{c_\omega^2\, \Delta |\kappa|\, T^2}{C_{\omega}^2 \bigl((j+\frac{1}{2})\pi \bigr)^2}\le\nu_j(A) \le \frac{C_\omega^2\, \Delta |\kappa|\, T^2}{c_{\omega}^2 \bigl((j-\frac{1}{2})\pi \bigr)^2} , \qquad j\ge1 .
\end{equation*}
\end{enumerate}
In all cases the nonzero eigenvalues obey $|\lambda_j(A)|\asymp j^{-2}$ as $j\to\infty$ (when infinitely many are present), with constant sign for $j\ge1$.
\end{prop}

\begin{proof}
	The congruence sandwich argument \eqref{eq:omega-sandwich} reduces every claim about $A$ to the spectrum of $K=R + J^* M_{-F'} J$, including the sign counts.
	Applying Corollary~\ref{cor:posregime} (case (i)), the identity $K=R$ (case (ii)), and Lemma~\ref{lem:negspec} with \eqref{eq:Kev} (case (iii)), with
	\begin{equation*}
		F(T) = \Delta e^{2b_i T}, \quad -F'(r) = \Delta \kappa\, e^{2b_i r},
	\end{equation*}
	gives the bounds. In particular, in case (i) Weyl's inequality yields
	\begin{align*}
		\lambda_0(K) \le T F(T) + C_F \lambda_0(J^*J) = T \Delta e^{2b_i T} + \frac{4 \Delta \kappa T^2}{\pi^2 c_{\omega}^2},
	\end{align*}
	and multiplying by $C_\omega^2$ gives the stated upper bound.
	In case (iii), Lemma~\ref{lem:negspec}(ii) gives $\mu_{j}\le\nu_{j}(K)\le\mu_{j-1}$ for $j\ge1$, and $\mu_k\in[c_GT^2((k+\frac12)\pi)^{-2}, C_GT^2((k+\frac12)\pi)^{-2}]$ by \eqref{eq:Kev} and \eqref{eq:JJ-evs}.
	The sandwich \eqref{eq:omega-sandwich} transfers this to $A$, since the congruence preserves the sign pattern, so the $j$-th negative eigenvalue of $A$ is comparable to the $j$-th negative eigenvalue of $K$ by Lemma~\ref{lem:congruence}(ii).
\end{proof}

\begin{rem}
	The sign of $\kappa=-(b_0+b_1)$ separates the two regimes of Assumption~\ref{ass:drifts} (3) in the constant case: regime $(+)$ is $b_0+b_1<0$ and regime $(-)$ is $b_0+b_1 \ge 0$. In particular the two regimes together cover all constant pairs $b_1\neq b_0$.
\end{rem}

\subsection{Distributional characterization} \label{ssec:density}
Section~\ref{ssec:chaos} states that the random variable in \eqref{eq:Z}
\begin{align*}
	Z = \int_0^T u(t) X_t \diff W_t + \frac{1}{2} \int_0^T v(t) X_t^2 \diff t
\end{align*}
has the representation of a generalized weighted chi-squared distribution in \eqref{eq:Chisquare} with eigenvalues of the operator $A$.
In this section, we connect the density of the shifted generalized chi-squared variable on the "small-ball" event $\{\vert Z \vert \le \varepsilon \}$ to the eigenvalue bounds obtained in Section~\ref{ssec:cJJ-evs}.
These spectral estimates provide the key ingredient for controlling the density near the relevant threshold, which is crucial for the margin condition.
The bound depends on locating the small-ball event, which is discussed in the following lemma. $f_X$ denotes the density function of a variable $X$ in $\R$.

\begin{lem} \label{lem:shift}
For every drift pair $b_0,b_1$ and every label $i$ under Assumption~\ref{ass:drifts},
\begin{equation}\label{eq:shiftid}
  \sum_{j\ge0}\lambda_j- 2 \E [Z] \, = \, \Tr A-2\,\E[Z]=\int_0^T a(t)\,\diff t=D .
\end{equation}
Consequently, writing $\bar Z:=\sum_{j\ge0}\lambda_jN_j^2$,
\begin{equation}\label{eq:Zbar}
  \bar Z=2Z+D,\qquad \{|Z|\le\delta\}=%
  \{D- 2\delta \le \bar Z \le D + 2\delta\}.
\end{equation}
\end{lem}

\begin{proof}
Applying the result of Brislawn (1991) \cite{brislawn1991}, we obtain
\begin{equation}
	\mathrm{Tr} A = \int_0^T \E[D_tD_tZ]\diff t,
\end{equation}
since $A$ is compact and trace class due to $F'\in L^{\infty}$.
From the Malliavin kernel of Section~\ref{ssec:chaos} at $s=t$,
\begin{equation*}
  \E[D_tD_tZ]=u(t)+\int_t^T v(r)\,\exp \left(2\int_t^r b_i (\tau) \diff \tau \right) \diff r .
\end{equation*}
Hence, by Fubini and $\E_i[X_r^2]=\int_0^r \exp \left(2\int_{\rho}^r b_i (\tau) \diff \tau \right)\diff \rho$,
\begin{equation*}
  \Tr A=\int_0^T u (t) \diff t
  +\int_0^T v(r)\Bigl(\int_0^r \exp \Bigl(2\int_t^r b_i (\tau) \diff \tau \Bigr)\diff t\Bigr)\diff r
  = D +\int_0^T v(r)\,\E_i[X_r^2]\,\diff r .
\end{equation*}
On the other hand $\E[Z]=\frac12\int_0^T v(r)\E_i[X_r^2]\diff r$ because the It\^o integral is centred.
Subtracting $2\E[Z]$ removes the second summand, leaving $\int_0^T u (t) \diff t$. Equation~\eqref{eq:Zbar} follows from $\bar Z=2(Z-\E Z)+\sum_j\lambda_j$ and $\sum_j\lambda_j=\Tr A$.
\end{proof}

The variable of interest is a rescaled generalized chi-squared variable which is shifted by $- \int_0^T u(t) \diff t/2$. This is the It\^o correction in Remark~\ref{rem:itocorrection}.
It is also the functional generalization of the constant-case shift by $- T\Delta/2$.
We first bound the density when the spectrum has an arbitrary sign pattern, requiring control of the positive part only.

Let $\bar Z=\sum_{j\ge0}\lambda_jN_j^2$ with $\sum_j|\lambda_j|<\infty$ and $\lambda_0:=\max_j\lambda_j>0$.
Order the \emph{positive} eigenvalues other than $\lambda_0$ as $\theta_1\ge\theta_2\ge\cdots>0$ (a finite or infinite list, possibly empty).
Let $n_+ \in \N_0 \cup \{\infty\}$ denote the number of the $\theta_k$'s.
Set
\begin{equation}\label{eq:PS}
  S_+:=\sum_{k\ge1}\theta_kN_k^2\ge0,\qquad
  S_-:=-\!\!\sum_{j:\lambda_j<0}\lambda_jN_j^2\ge0,\qquad
  P:=\lambda_0N_0^2+S_+.
\end{equation}
We have
\begin{align*}
	\bar Z=P-S_-
\end{align*}
with $P$, $S_-\ge0$ independent.

\begin{defi}[Block Ulyanov factor]
	For an integer $m \ge 2$ define
	\begin{equation}\label{eq:Pim}
		\Pi_m:=
		\begin{cases}
			0, & n_+=0,\\
			\Bigl(\displaystyle\prod_{j=2}^{m\wedge n_+}\frac{\theta_1}{\theta_j}\Bigr)^{1/2}
			\displaystyle\prod_{k> m}\Bigl(1-\frac{\theta_k}{\theta_1}\Bigr)^{-1/2},
			&  n_+ \ge 1,
		\end{cases}
	\end{equation}
	with the conventions that empty products equal $1$ and that $\Pi_m:=\infty$ if $\theta_{m+1}=\theta_1$. For $m=2$ this is the classical Ulyanov factor $\Pi$ in \cite{ulyanov1995, ulyanov1991}.
\end{defi}

Recall the elementary scale bound for a $\theta\chi^2_k$ ($\theta>0$) variable: for $k\ge1$ let
\begin{equation*}
	g_k(y;\theta):=\frac{y^{k/2-1}}{(2\theta)^{k/2}\Gamma(k/2)}e^{-y/(2\theta)}
\end{equation*}
denote the density.
Substituting $\tau=y/(2\theta)$, we obtain
\begin{equation}\label{eq:betam}
	\sup_{\theta>0}g_k(y;\theta)=\frac{\beta_k}{y}\quad\text{for every }y>0,
	\qquad
	\beta_k:=\frac{\sup_{\tau>0}\tau^{k/2}e^{-\tau}}{\Gamma(k/2)}=\frac{(k/2)^{k/2}e^{-k/2}}{\Gamma(k/2)} ,
\end{equation}
with $\beta_1=\frac{e^{-1/2}}{\sqrt{2\pi}}\le0.25$ and $\beta_2=e^{-1}\le0.37$. Set $\bar\beta_m:=\max_{1\le k\le m}\beta_k$.

\begin{lem} \label{lem:blockulyanov}
	Let $m\ge2$ and $\Pi_m<\infty$, and suppose $n_+ \ge 1$. Then for all $y>0$,
	\begin{equation}\label{eq:blockulyanov}
		f_{S_+}(y)\le\Pi_m\; g_{m\wedge n_+}(y;\theta_1)
		\le\frac{\bar\beta_m\,\Pi_m}{y}.
	\end{equation}
\end{lem}
\begin{proof}
	Write $m':= m \wedge n_+$ and $S_+ = \sum_{j=1}^{m'} \theta_j N_j^2 + \sum_{k > m'} \theta_k N_k^2$. Denote $S':= \sum_{k > m'} \theta_k N_k^2$.
	Replacing $\theta_j$ by $\theta_1$ for $j =2, \cdots, m'$ gives the multiplicative cost $\sqrt{\theta_1 / \theta_j}$ to the density via repeated application of Lemma~\ref{lem-chidensity}. Hence
	\begin{equation*}
		f_{S_+}(y)\le\prod_{j=2}^{m'} \Bigl(\frac{\theta_1}{\theta_j}\Bigr)^{1/2} f_{\theta_1\chi^2_{m'}+\sum_{k > m'} \theta_k N_k^2}(y).
	\end{equation*}
	If $S'=0$, then $f_{\theta_1\chi^2_{m'}}(y)=g_{m'}(y;\theta_1)$ and \eqref{eq:blockulyanov} follows.
	If $S'\neq0$, then $n_+>m\ge2$, so $m'=m\ge2$ and $y\mapsto y^{m/2-1}$ is nondecreasing. Hence, for $0\le w\le y$,
	\begin{equation*}
		g_m(y-w;\theta_1)=\frac{(y-w)^{m/2-1}}{(2\theta_1)^{m/2}\Gamma(m/2)}\,e^{-\frac{y-w}{2\theta_1}}
		\le\frac{y^{m/2-1}}{(2\theta_1)^{m/2}\Gamma(m/2)}\,e^{-\frac{y}{2\theta_1}}\,e^{\frac{w}{2\theta_1}}
		=g_m(y;\theta_1)\,e^{\frac{w}{2\theta_1}}.
	\end{equation*}
	By the convolution formula and the moment generating function of $S'$ at $\frac1{2\theta_1}$ (finite because $\theta_k/\theta_1\le\theta_{m+1}/\theta_1<1$ for $k>m$ and $\sum_k\theta_k<\infty$),
	\begin{equation*}
		f_{\theta_1\chi^2_m+S'}(y)=\int_0^yg_m(y-w;\theta_1)f_{S'}(w)\diff w
		\le g_m(y;\theta_1)\,\E\bigl[e^{S'/(2\theta_1)}\bigr]
		=g_m(y;\theta_1)\prod_{k>m}\Bigl(1-\frac{\theta_k}{\theta_1}\Bigr)^{-1/2} .
	\end{equation*}
	Multiplying the two products gives the first inequality of \eqref{eq:blockulyanov}; the second follows from the definition of $\beta$ in \eqref{eq:betam} and $\beta_{m'}\le\bar\beta_m$.
\end{proof}

We obtain the main tool for the density bound.
In contrast to a direct convolution analysis, no relation between the leading eigenvalue $\lambda_0$ and the evaluation point is required.

\begin{thm} \label{thm:density}
Let $\bar Z=\sum_{j\ge0}\lambda_jN_j^2$ with $\sum_j|\lambda_j|<\infty$ and $\lambda_0=\max_j\lambda_j>0$, and let $m\ge2$ be an integer with $\Pi_m<\infty$ as in \eqref{eq:Pim}. Then
\begin{equation}\label{eq:densitybound}
  f_{\bar Z}(x)\le\frac{2\,(\beta_1+\bar\beta_m\,\Pi_m)}{x}\qquad\text{for all }x>0 .
\end{equation}
If $\bar Z$ has no positive eigenvalue ($\lambda_0 = 0, \lambda_j \le 0$), then $\prob(\bar Z>0)=0$ and no density bound is needed on $(0,\infty)$.
\end{thm}

\begin{proof}
Fix $x>0$. Since $\lambda_0>0$, the variable $P$ in \eqref{eq:PS} has a density.
If $S_-\neq0$ it has a density as well and, by independence,
\begin{equation*}
  f_{\bar Z}(x)=\int_0^\infty f_P(x+y)\,f_{S_-}(y)\,\diff y\le \sup_{w\ge x}f_P(w),
\end{equation*}
since $f_{S_-}$ is supported on $[0,\infty)$ and integrates to $1$.
If $S_-=0$ then $f_{\bar Z}=f_P$ and the same bound holds trivially. It remains to bound $f_P(w)$ for $w\ge x$.

If $S_+ = 0$, using that $f_{\lambda_0N_0^2}$ is decreasing on $(0,\infty)$ and \eqref{eq:betam} with $k=1$,
\begin{align*}
f_P(w)
=f_{\lambda_0N_0^2}(w)\le f_{\lambda_0N_0^2}(x)\le\frac{\beta_1}{x} \qquad\text{for } w\ge x,
\end{align*}
which implies \eqref{eq:densitybound}.
Otherwise $S_+$ has a density. By independence,
\begin{equation*}
  f_P(w)=\int_0^w f_{\lambda_0N_0^2}(w-y)\,f_{S_+}(y)\,\diff y
  =\int_0^{x/2} f_{\lambda_0N_0^2}(w-y)\,f_{S_+}(y)\,\diff y +\int_{x/2}^{w} f_{\lambda_0N_0^2}(w-y)\,f_{S_+}(y)\,\diff y.
\end{equation*}
On $[0,x/2]$ we have $w-y\ge w-x/2\ge x/2$, so by monotonicity of the $\chi^2_1$ density and \eqref{eq:betam},
\begin{equation*}
  f_{\lambda_0N_0^2}(w-y)\le f_{\lambda_0N_0^2}(\frac{x}2)
  \le\frac{2\beta_1}{x}.
\end{equation*}
Thus, the integral over this interval is at most $2\beta_1/x$, since $f_{S_+}$ integrates to at most 1 there.
On $[x/2,w]$ we use Lemma~\ref{lem:blockulyanov}: $f_{S_+}(y)\le\bar\beta_m\Pi_m/y\le 2 \bar\beta_m\Pi_m/x$, and
$\int_{x/2}^w f_{\lambda_0N_0^2}(w-y)\diff y\le 1$.
Adding the two summands gives
\begin{equation*}
	f_P(w)\le \frac{2(\beta_1+\bar\beta_m\Pi_m)}{x}
\end{equation*}
for every $w\ge x$. The inequality holds for any $m$ such that $\Pi_m < \infty$.
\end{proof}

\begin{rem} \label{rem-lambdas}
	The assumption $\Pi_m < \infty$ for some $m$ is the essential hypothesis in the positive regime. It is introduced to ensure that we can effectively control the density of an infinite-dimensional sum of weighted chi-squared variables when estimating the upper bound via a finite-dimensional approximation.
	This condition is generally satisfied by the eigenvalues of typical Sturm--Liouville problems.
	Lemma~\ref{lem:ulyanov} verifies the finiteness of $\Pi_m$ for every integer $m \ge 2$ with $(m+1)^2>\frac{9}{4}\rho$ in the positive regime.

	The unbounded nature of the $\chi^2(1)$ distribution requires careful control over the range of $\varepsilon$ near the left endpoint of the support.
	In the positive regime, the overall bias causes the support of the variable $\bar{Z}$ to shift to $[- D, \infty)$.
	This shift leads to a possibly infinite density of $Z$ at the endpoint $- D/2$. The range assumption $x \ge D/2$ is introduced to guarantee that the evaluation window of $\bar{Z}$ stays away from this endpoint.
\end{rem}

\begin{proof}[Proof of Margin Assumption (Proposition~\ref{prop:margin} and Proposition~\ref{prop:margininhomo})]
	We first prove Proposition~\ref{prop:margininhomo}.
	Fix a label $i$ and let $Z$ be the margin variable of \eqref{eq:M1}--\eqref{eq:M0} conditionally on $Y=i$.
	For all $\varepsilon \le D/64$,
	$\{|Z|\le16\varepsilon\}=\{|\bar Z-D|\le32\varepsilon\}$ with $\bar Z=\sum_j\lambda_j N_j^2$ and $D>0$ by Lemma~\ref{lem:shift}.
	According to the sign structure of the eigenvalues, we have two regimes.

	\emph{Regime $(+)$.} By Corollary~\ref{cor:posregime} all eigenvalues are positive. By Lemma~\ref{lem:ulyanov}, there exists $m \in \N$ such that $\Pi_m < \infty$.
	Theorem~\ref{thm:density} gives $f_{\bar Z}(x)\le c_d/D$ for $x \ge D/2$ with
	\begin{equation*}
		c_d = \max\{4 (\beta_1 + \bar{\beta}_m \Pi_m),1\}.
	\end{equation*}

	\emph{Regime $(-)$.} Here $-F' \le 0$ and Lemma~\ref{lem:negspec} shows
	$A$ has at most one positive eigenvalue.
	If $A$ has no positive eigenvalue, then $\bar Z\le0$ almost surely while $D-32\varepsilon\ge D/2>0$, so $M_i=0$ holds trivially.
	Otherwise $S_+=0$, $\Pi_m=0$, and Theorem~\ref{thm:density} gives $f_{\bar Z}(x)\le c_d/D$ for $x \ge D/2$ with $c_d = \max\{4\beta_1,1\}=1$ (independent of the Ulyanov factor).

	In both regimes, if $32\varepsilon\le D/2$, the interval $[D-32\varepsilon,D+32\varepsilon]$ lies in $[D/2,\infty)$ and has length $64\varepsilon$, so
	\begin{equation*}
		M_i \le 64 \varepsilon \sup_{x\ge D/2} f_{\bar Z}(x)
		\le 64 \varepsilon \cdot \frac{c_d}{D}
		=\frac{64 c_d \,\varepsilon}{D},
	\end{equation*}
	where $c_d$ depends only on $\rho$.
	Thus, letting $C_{\mathrm{margin}} : = 64c_d$ under Assumption~\ref{ass:drifts},
	\begin{equation*}
		\prob_X\Bigl(\bigl|\eta(X)-\frac12\bigr|\le\varepsilon\Bigr)
		\le\frac12M_1+\frac12M_0
		\le\frac{64 c_d \,\varepsilon}{D}
		= \frac{C_{\mathrm{margin}}\,\varepsilon}{D} .
	\end{equation*}
	In regime $(+)$, Theorem~\ref{thm:density} applies with every $m\ge 2$ such that $(m+1)^2 > \frac94 \rho$.
	Lemma~\ref{lem:ulyanov} bounds $\Pi_m \le B_m(\rho)$ with $B_m(\rho)$ the right-hand side of \eqref{eq:Pimbound}.
	Hence
	\begin{equation} \label{eq:Cmargindef}
		64c_d \le C_{\mathrm{margin}} (\rho) := 256 \inf_{m\ge 2: (m+1)^2 > \frac94 \rho} \bigl(\beta_1 + \bar\beta_m B_m(\rho) \bigr).
	\end{equation}
	For fixed $m$ the function $\rho \mapsto B_m(\rho)$ is increasing on $[1, \frac{4}{9}(m+1)^2)$.
	Consequently every $m$ admissible for $\rho_0$ is admissible for $\rho\le\rho_0$ with a smaller value, and $C_{\mathrm{margin}}(\rho)\le C_{\mathrm{margin}}(\rho_0)$ whenever $\rho\le\rho_0$.
	In regime $(-)$, $C_{\mathrm{margin}} = 64c_d= 64$ is a uniform constant.

	If instead $\varepsilon > D/64$, then $C_{\mathrm{margin}}\,\varepsilon/D\ge64\cdot\varepsilon/D>1$, so the bound in \eqref{eq:margininhomo} holds trivially. This proves Proposition~\ref{prop:margininhomo}.

	Proposition~\ref{prop:margin} is the constant-drift special case: Assumption~\ref{ass:drifts} holds with $a = \Delta = b_1 - b_0 > 0$, and Assumption~\ref{ass:drifts} holds in regime $(+)$ if $b_1 + b_0 < 0$ and in regime $(-)$ if $b_1 + b_0 \ge 0$ with $D = T \Delta > 0$.
	In regime $(+)$, Lemma~\ref{lem:ulyanov} applies with $\rho = \rho_i = \bigl(C_{\omega, i} / c_{\omega, i} \bigr)^4 = e^{4 | b_i | T} \le e^{4 LT}$.
	The resulting constant $C_{\mathrm{margin}}$ depends only on $LT$ for all pairs in $\mathcal{B} (L, \Delta_0)$.
\end{proof}

\subsection{Control of the regression estimation error}\label{ssec:estimation}
We now prove the upper bound of the estimation error using the concentration inequality for our variable $Z$.
We first collect the required norm identities for the spectral expansion in the estimation error.

\begin{lem} \label{lem:chaosnorms}
	Conditionally on $Y=i$, let $\lambda = (\lambda_j)_{j \ge 1}$ denote the eigenvalues of the operator $A$ of the corresponding variable.
	\begin{enumerate}
		\item[\textup{(i)}] For $Z_1=\int_0^T u(t) X_t\diff W_t$ with $u \in L^2([0,T])$:
		$\E[Z_1]=0$ and
		\begin{equation*}
			\Vert\lambda\Vert_2^2=2\var(Z_1)=2\int_0^T u(t)^2\,\E_i[X_t^2]\diff t=2\Vert u \Vert_{X,i}^2,
			\qquad \Vert\lambda\Vert_\infty\le\Vert\lambda\Vert_2=\sqrt2\,\Vert u\Vert_{X,i} .
		\end{equation*}
		\item[\textup{(ii)}] For $Z_2=\int_0^T h(t)X_t^2\diff t$ with $h \in L^1([0,T])$: the operator $A$ (built with $u=0$, $v=2h$ in $Z$) is trace class with
		\begin{equation*}
			\E[Z_2]=\int_0^Th(t)\,\E_i[X_t^2]\diff t,
			\qquad
			\Vert\lambda\Vert_1=\sum_j|\lambda_j|\le2\int_0^T|h(t)|\,\E_i[X_t^2]\diff t,
		\end{equation*}
		and $\Vert\lambda\Vert_\infty\le\Vert\lambda\Vert_2\le\Vert\lambda\Vert_1$.
	\end{enumerate}
\end{lem}

\begin{proof}
	(i) The It\^o integral is centred, and by the It\^o isometry $\var(Z_1)=\int_0^T u^2\E_i[X_t^2]$. On the other hand $\var(Z_1)=\frac12\Vert\lambda\Vert_2^2$ by the chi-squared expansion of $Z$ in \eqref{eq:Chisquare}.

	(ii) $\E[Z_2]=\int h\E_i[X_t^2]$ holds by Fubini.
	Split $h=h_+-h_-$ with $h_+ = \max \{h, 0\}$ and $h_- = \max\{-h, 0\}$. Let $A_\pm$ be the operators with kernels built from $v=2h_\pm$ and $u=0$, i.e.\ with kernel $\omega(s)\omega(t)\int_{s\vee t}^T2h_\pm(r)\omega(r)^{-2}\diff r$.
	Setting $S_\pm:=M_{\sqrt{2h_\pm}\,\omega^{-1}}JM_\omega$, which is Hilbert--Schmidt with kernel $\sqrt{2h_\pm(t)}\,\omega(t)^{-1}\indic (s\le t)\,\omega(s)$,
	and $A_\pm =S_\pm^{*}S_\pm\succeq0.$
	So $A_\pm$ is positive semidefinite and trace class with
	\begin{equation*}
		\Tr A_\pm=\bigl\Vert S_\pm\bigr\Vert_{\HS}^2
		=\int_0^T2h_\pm(r)\,\omega(r)^{-2}\int_0^r\omega(s)^2\diff s\diff r
		=2\int_0^Th_\pm(r)\,\E_i[X_r^2]\diff r<\infty,
	\end{equation*}
	where the last equality absorbs the factor $\omega(r)^{-2}$ via the identity
	\begin{equation*}
		\E_i[X_r^2]=\int_0^r e^{2\int_s^rb_i(\tau)\diff\tau}\diff s
	=\omega(r)^{-2}\int_0^r\omega(s)^2\diff s.
	\end{equation*}
	Hence $A=A_+-A_-$ is trace class with trace norm $\Vert A\Vert_1\le\Tr A_++\Tr A_-=2\int|h|\E_i[X^2]$, and $\sum_j|\lambda_j|\le\Vert A\Vert_1$.
\end{proof}

Note that the derivation of these bounds does not use the differentiability of the coefficients or the sign of the eigenvalues: parts (i) and (ii) hold for arbitrary $u \in L^2$, $h\in L^1$, in particular for the estimation errors $\hat b_i-b_i$ given $\Samp_n$.

The general estimation lemma is the following.
\begin{lem}[General estimation error bound]\label{lem:est-general}
	Let $j\in\{0,1\}$ and let $(\phi^{(j)}_N)_{N\in\N}$ be a sequence with $\phi^{(j)}_N\le\delta_n$ for all $N\ge n/4$, for some $\delta_n>0$.
	Assume that the estimators $\hat b_0,\hat b_1$ are built from the subsamples with the same label and satisfy
	\begin{equation}\label{eq:est-general-cond}
		\prob_{\Samp_n} \Bigl(\bigl\Vert\hat b_i-b_i\bigr\Vert_{X,j}\ge\phi^{(j)}_{n_i}\;\Big|\;N_i=n_i\Bigr)\le\frac{C_s}{n_i},\qquad i\in\{0,1\},\ n_i\in\N .
	\end{equation}
	Set $\delta_{\mathrm{cross},j}:=\delta_n^2+2\delta_n\Vert a\Vert_{X,j}$ with $a=b_1-b_0$. Then for $n\ge27$ and every $0<\varepsilon\le3/4$ with $\varepsilon\ge12\, \delta_{\mathrm{cross},j}$,
	\begin{equation}\label{eq:est-general}
		\begin{aligned}
		\prob_{\Samp_n,b_j}\bigl(|\widehat{\eta}(X)-\eta(X)|>\varepsilon \bigr)
		&\le\frac{2+8C_s}{n}
		+4\exp\Bigl(-\frac{\varepsilon^2}{288\,\delta_n^2+24 \sqrt{2}\,\delta_n\,\varepsilon}\Bigr)\\
		&\quad+\exp\Bigl(-\frac{\varepsilon^2}{576\,\delta_n^4+48\,\delta_n^2\,\varepsilon}\Bigr)
		+2\exp\Bigl(-\frac{\varepsilon^2}{576\,\delta_{\mathrm{cross},j}^2+48\,\delta_{\mathrm{cross},j}\,\varepsilon}\Bigr) .
		\end{aligned}
	\end{equation}
	If moreover
	\begin{equation}\label{eq:est-threshold}
		\varepsilon\ge\max\bigl\{48 \sqrt{2}\,\delta_n\log n,\;96\,\delta_n^2\log n,\;96\,\delta_{\mathrm{cross},j}\log n\bigr\},
	\end{equation}
	then each exponential term in \eqref{eq:est-general} is at most $1/n$ and
	\begin{equation*}
		\prob_{\Samp_n,X}\bigl(|\widehat{\eta}(X)-\eta(X)|>\varepsilon\;\big|\;Y=j\bigr)\le\frac{9+8C_s}{n} .
	\end{equation*}
\end{lem}

\begin{proof}
Let $N_i$ denote the size of the subset in $ \Samp_n $ with label $i$. Then $N_1$ is Binomial$(n,1/2)$ distributed and $N_0=n-N_1$.
By Hoeffding's inequality, for $n\ge27$
\begin{equation*}
	\prob(N_i<n/4)\le e^{-n/8} \le 1/n.
\end{equation*}
Indeed $n/8\ge\log n$ for $n\ge27$.
By \eqref{eq:est-general-cond} and conditioning on the value of $N_i$,
\begin{equation*}
	\prob_{\Samp_n}\Bigl(\bigl\Vert\hat b_i-b_i\bigr\Vert_{X,j}\ge\phi^{(j)}_{N_i},\;N_i\ge \frac n4\Bigr)
	=\!\!\sum_{n_i\ge n/4}\!\prob_{\Samp_n}\Bigl(\bigl\Vert\hat b_i-b_i\bigr\Vert_{X,j}\ge\phi^{(j)}_{n_i}\Big|N_i=n_i\Bigr)\prob(N_i=n_i)
	\le \frac{4C_s}{n}.
\end{equation*}
Define the event
\begin{equation*}
	\mathcal A:=\Bigl\{N_0\ge\frac{n}{4},\;N_1\ge\frac{n}{4}\Bigr\}\cap\Bigl\{\bigl\Vert\hat b_i-b_i\bigr\Vert_{X,j}\le\phi^{(j)}_{N_i}\ \text{for }i=0,1\Bigr\}.
\end{equation*}
We have
$\prob(\mathcal A^c)\le 2/n+8C_s/n$ and $\Vert\hat b_i-b_i\Vert_{X,j}\le\delta_n$ for both $i$ on $\mathcal A$. Note that $\mathcal A$ is $\sigma(\Samp_n)$-measurable and independent of the test pair $(X,Y)$.

By Lemma~\ref{lem:exp} as in Section~\ref{ssec:reduction}, for $0<\varepsilon\le3/4$,
\begin{equation*}
	\prob_{\Samp_n, b_j}\bigl(|\widehat{\eta}-\eta|>\varepsilon \bigr)
	\le\prob_{\Samp_n, b_j}(\mathcal A^c)+\sum_{i=0,1}\prob_{\Samp_n, b_j}\Bigl(\bigl|\Lambda_i\bigr|>\frac\varepsilon6,\;\mathcal A \Bigr),
\end{equation*}
where, by \eqref{eq:logratio-general} with $u_i:=\hat b_i-b_i$,
\begin{equation*}
	\Lambda_i:=\log\frac{p_{\hat b_i}}{p_{b_i}}(X)
	=\int_0^T u_i(t)X_t\diff W_t+\int_0^Th_i(t)X_t^2\diff t,
	\qquad h_i:=u_i\bigl((b_j-b_i)-\frac{u_i}2\bigr) .
\end{equation*}
For $i=j$ we have $h_j=-u_j^2/2 \le 0$; for $i\neq j$, $h_i=- u_i^2/2\pm a\,u_i$, where $a = b_1 - b_0$. Conditionally on $Y=j$, we split the probability
\begin{equation*}
	\prob_{\Samp_n, b_j}\Bigl(\bigl|\Lambda_i\bigr|>\frac\varepsilon6,\;\mathcal A \Bigr) \le \prob_{\Samp_n, b_j} \left(\left\vert \int_0^T u_i X_t \diff W_t \right\vert > \frac{\varepsilon}{12} , \mathcal{A}\right) + \prob_{\Samp_n, b_j} \left(  \left\vert \int_0^T h_i(t) X_t^2 \diff t \right\vert > \frac{\varepsilon}{12} , \mathcal{A} \right).
\end{equation*}

Conditionally on $\Samp_n$, the coefficients $u_i$, $h_i$ are deterministic, and the test path $X$ (given $Y=j$) is an OU path independent of $\Samp_n$. Each tail probability is therefore computed for deterministic coefficients and then integrated over $\Samp_n$ restricted to $\mathcal A$. Since each bound is a monotone function of $\Vert u_i\Vert_{X,j}$, it holds uniformly on $\mathcal A$ with $\Vert u_i\Vert_{X,j} \le \delta_n$.

The stochastic term satisfies $\int u_iX\diff W=\frac12\sum_j\lambda_j(N_j^2-1)$ with $\Vert\lambda\Vert_2^2=2\Vert u_i\Vert_{X,j}^2\le2\delta_n^2$ and $\Vert\lambda\Vert_\infty\le\sqrt2\delta_n$ by \eqref{eq:Chisquare}.
Lemma~\ref{lem:LM} at $t=\varepsilon/6$ gives
\begin{equation*}
	\begin{aligned}
	\prob_{b_j}\Bigl(\Bigl|\int_0^T u_i X_t\diff W_t\Bigr|>\frac\varepsilon{12}\Bigr)
	& = \prob_{b_j}\Bigl(\Bigl|\sum\lambda_j(N_j^2-1)\Bigr|>\frac\varepsilon6\Bigr) \\
	& \le 2\exp\Bigl(-\frac{\varepsilon^2/36}{8\delta_n^2+\frac{4\sqrt2}{6}\delta_n\varepsilon}\Bigr) \\
	& \le 2\exp\Bigl(-\frac{\varepsilon^2}{288\delta_n^2+24\sqrt{2}\,\delta_n\varepsilon}\Bigr).
	\end{aligned}
\end{equation*}
This term appears twice, once for each of $i=0$ and $i=1$. Thus, there are in total $4$ exponentials of this type in the final bound.

For the quadratic term with $i = j$, we have $\int_0^T h_jX^2=-\frac12 \int_0^Tu_j^2 X^2\le0$, so the tail is one-sided.
Write
\begin{equation*}
Q:=\int u_j^2X_t^2\diff t=\E Q+\frac12\sum\lambda_k(N_k^2-1).
\end{equation*}
By Lemma~\ref{lem:chaosnorms}(ii) (with $h=u_j^2 \ge 0$),
\begin{equation*}
	\E_j Q=\Vert u_j\Vert_{X,j}^2\le\delta_n^2, \quad \Vert\lambda\Vert_\infty\le\Vert\lambda\Vert_2\le\Vert\lambda\Vert_1\le2\delta_n^2.
\end{equation*}
Since $\varepsilon \ge 12\delta_{\mathrm{cross},j}\ge12\delta_n^2\ge 12\,\E_j Q$, the event $\{\frac12Q>\varepsilon/12\}$ implies $\sum\lambda_k(N_k^2-1)=2(Q-\E_j Q)>\varepsilon/3-2\delta_n^2\ge \varepsilon/6$. The tail bound in Lemma~\ref{lem:LM} gives
\begin{equation*}
	\prob_{b_j}\Bigl(\frac12\int u_j^2X_t^2\diff t>\frac\varepsilon{12}\Bigr)
	\le\exp\Bigl(-\frac{\varepsilon^2/36}{16\delta_n^4+\frac{8}{6}\delta_n^2\varepsilon}\Bigr)
	\le\exp\Bigl(-\frac{\varepsilon^2}{576\delta_n^4+48\delta_n^2\varepsilon}\Bigr) .
\end{equation*}

For the cross term where $i \neq j$, $h_i=- u_i^2/2 \pm a u_i$. By the Cauchy--Schwarz inequality in weighted space $L^2(\E_j[X_t^2]\diff t)$,
\begin{equation*}
	\int_0^T|h_i|\,\E_j[X_t^2]\diff t\le\frac12\Vert u_i\Vert_{X,j}^2+\Vert a\Vert_{X,j}\Vert u_i\Vert_{X,j}
	\le\frac12\delta_n^2+\delta_n\Vert a\Vert_{X,j}=\frac12 \delta_{\mathrm{cross},j} .
\end{equation*}
Write $Z_2:=\int h_iX^2=\E_j Z_2+\frac12\sum\lambda_k(N_k^2-1)$.
We have $|\E_j Z_2|\le\frac12\delta_{\mathrm{cross},j}$ and $\Vert\lambda\Vert_\infty\le\Vert\lambda\Vert_2\le\Vert\lambda\Vert_1\le \delta_{\mathrm{cross},j}$ by Lemma~\ref{lem:chaosnorms}. Analog to the case $i=j$, the event $\{|Z_2|>\varepsilon/12\}$ implies $|\sum\lambda_k(N_k^2-1)|>2(\varepsilon/12-\delta_{\mathrm{cross},j}/2)\ge\varepsilon/6-\varepsilon/12=\varepsilon/12$. Lemma~\ref{lem:LM} gives
\begin{equation*}
	\prob_{b_j}\Bigl(\Bigl|\int_0^Th_iX_t^2\diff t\Bigr|>\frac\varepsilon{12}\Bigr)
	\le2\exp\Bigl(-\frac{\varepsilon^2/144}{4\delta_{\mathrm{cross},j}^2+\frac{4}{12}\delta_{\mathrm{cross},j}\varepsilon}\Bigr)
	=2\exp\Bigl(-\frac{\varepsilon^2}{576\delta_{\mathrm{cross},j}^2+48\,\delta_{\mathrm{cross},j}\varepsilon}\Bigr) .
\end{equation*}

Adding $\prob(\mathcal A^c)\le(2+8C_s)/n$ with all the exponential terms above gives \eqref{eq:est-general}.

For the threshold claim, note first that \eqref{eq:est-threshold} implies $\varepsilon \ge 12\delta_{\mathrm{cross},j}$, and that $\log n\ge1$ for $n \ge 27$.
From $\varepsilon\ge48\sqrt2\,\delta_n\log n$ we get, using $\log^2 n \ge \log n$,
\begin{equation*}
	288\,\delta_n^2+24\sqrt2\,\delta_n\varepsilon
	\le\Bigl(\frac{288}{4608}+\frac{24\sqrt2}{48\sqrt2}\Bigr)\frac{\varepsilon^2}{\log n}
	=\frac{9}{16}\,\frac{\varepsilon^2}{\log n},
\end{equation*}
so the exponent of the stochastic term is at least $\frac{16}{9}\log n\ge\log n$. From $\varepsilon\ge96\,\delta_n^2\log n$,
\begin{equation*}
	576\,\delta_n^4+48\,\delta_n^2\varepsilon
	\le\Bigl(\frac{576}{9216}+\frac{48}{96}\Bigr)\frac{\varepsilon^2}{\log n}
	=\frac{9}{16}\,\frac{\varepsilon^2}{\log n},
\end{equation*}
and the same computation applies to the cross term with $\delta_{\mathrm{cross},j}$ in place of $\delta_n^2$.
Hence each of the $7$ exponential terms is at most $e^{-\log n}=1/n$, whence the bound $(9+8C_s)/n$.
\end{proof}

\begin{proof}[Proof of Lemma~\ref{lem-est}]
	Under Assumption~\ref{ass:est}, on $\{\sqrt{C_{\mathrm{est},j}}|\hat b_i-b_i|\le K_0 \sqrt{\log N_i/N_i}\}$ we have
	\begin{equation*}
	\Vert\hat b_i-b_i\Vert_{X,j} \le K_0 \sqrt{\frac{\log N_i}{N_i}}.
	\end{equation*}
	So \eqref{eq:est-general-cond} holds with $\phi^{(j)}_N=K_0 \sqrt{\log N/N}$ for both $j \in \{0,1\}$. The map $N\mapsto\sqrt{\log N/N}$ is nonincreasing for $N\ge3$, and for $N\ge n/4$ (with $n\ge27$)
	\begin{equation*}
		\phi^{(j)}_N\le K_0 \sqrt{\frac{\log (n/4)}{n/4}}=2 K_0 \sqrt{\frac{\log n-\log4}{n}} \le 2 K_0 \sqrt{\frac{\log n}{n}}=:\delta_{n} .
	\end{equation*}
	Moreover $\Vert a\Vert_{X,j}=\Delta\Vert\mathbf1\Vert_{X,j}=\Delta\sqrt{C_{\mathrm{est},j}}$, so
	$\delta_{\mathrm{cross},i}=\delta_n^2+2\delta_n\sqrt{C_{\mathrm{est},i}}\,\Delta$.
	We apply Lemma~\ref{lem:est-general} with conditioning label $j=i$.
	The choice of $\varepsilon_i$ in \eqref{eq:choiceeps} satisfies the threshold \eqref{eq:est-threshold}.
	\begin{itemize}
		\item the first summand of $\delta_{i}$ contributes $96\sqrt{2}K_0\,(\log n)^{3/2}n^{-1/2}=48\sqrt2\,\delta_{n}\log n$;
		\item the second summand dominates $96\,\delta_{n}^2 \log n=384 K_0^2 (\log n)^{2}n^{-1}$ since $\log n\le n$;
		\item the second and third summands together dominate
		\begin{equation*}
		96\,\delta_{\mathrm{cross},i}\log n=384\,K_0 \bigl(K_0(\log n)^{2}n^{-1}+\sqrt{C_{\mathrm{est},i}} \Delta\,(\log n)^{3/2}n^{-1/2}\bigr).
		\end{equation*}
	\end{itemize}
	Lemma~\ref{lem:est-general} yields the claim.
\end{proof}

\begin{proof}[Proof of Lemma~\ref{lem-estinhomo}]
	Under Assumption~\ref{ass:estinhomo}, \eqref{eq:est-general-cond} holds with $\phi^{(j)}_N=K_s (\log N / N)^{s/(2s+1)}$ for both $j$. The map $N\mapsto\phi_N$ is nonincreasing for $N \ge 3$. For $n \ge 27$ and $N\ge n/4$,
	\begin{equation*}
		\phi^{(j)}_N
		\le K_s (\log n)^{\frac{s}{2s+1}}\Bigl(\frac n4\Bigr)^{-\frac{s}{2s+1}}
		=K_s 4^{\frac{s}{2s+1}}(\log n)^{\frac{s}{2s+1}} n^{-\frac{s}{2s+1}}
		\le 2K_s (\log n)^{\frac{s}{2s+1}} n^{-\frac{s}{2s+1}}=: \delta_n ,
	\end{equation*}
	using $4^{s/(2s+1)}\le 2$. With $\delta_{\mathrm{cross},j} = \delta_{n}^2+2\delta_n \Vert a \Vert_{X,j}$, the choice~\eqref{eq:choiceepsinhomo} satisfies the threshold~\eqref{eq:est-threshold}.
	Lemma~\ref{lem:est-general} yields the claim.
\end{proof}

\subsection{Proof of the excess risk rate} \label{ssec:excess}
\begin{proof}[Proof of Theorem~\ref{thm:mainconstant} and Theorem~\ref{thm:maininhomo}]
	We prove Theorem~\ref{thm:maininhomo} here; Theorem~\ref{thm:mainconstant} follows by the same argument with $\phi_N = K_0 \sqrt{\log N/N}$, $\delta_{n,i}$ and $\delta_{\mathrm{cross},i}$, as in the estimation error, $D = T \Delta$ and the choice
	\begin{equation*}
	\varepsilon_i := \delta_{i} \bigl(\log n\bigr)^{\frac{3}{2}} n^{-\frac{1}{2}}.
	\end{equation*}

	Conditionally on $Y=j$, the decomposition \eqref{eq:masterdecomp} reads, for any $0<\varepsilon_j\le1/16$,
	\begin{equation}\label{eq:cond-decomp}
		\E_{\Samp_n,j} \left[\vert 2\eta(X) -1 \vert \indic \bigl( \widehat{\Phi}(X) \neq \Phi^*(X) \bigr) \right]
		\le 2 \varepsilon_j \,\prob_{b_j}\Bigl(\bigl|\eta(X)-\frac12\bigr|\le\varepsilon_j \Bigr) + \prob_{\Samp_n, b_j}\Bigl(|\widehat{\eta}-\eta|>\varepsilon_j \Bigr) .
	\end{equation}
	By the reduction of Section~\ref{ssec:reduction} and Theorem~\ref{thm:density}, the conditional margin probability obeys the same bound as \eqref{eq:margininhomo}:
	\begin{equation*}
		\prob_{b_j}\Bigl(\bigl|\eta(X)-\frac12\bigr|\le\varepsilon_j\Bigr)\le 1 \wedge \frac{C_{\mathrm{margin}}\,\varepsilon_j}{D} .
	\end{equation*}
	Choose $\varepsilon_j$ as in \eqref{eq:choiceepsinhomo}, that is $\varepsilon_j=\delta_{n,j}(\log n)v_n$.
	By Lemma~\ref{lem-estinhomo}, the estimation term is at most $(9+8C_s)/n$.
	The factor $\delta_{n,j}$ is bounded uniformly over the class: since $\log n\le n$, the middle term in its definition is at most $8K_s^2$.
	Indeed, for $(b_1, b_0)$ in the given set, we have $\Vert b_j \Vert_{\infty} \le R$ and thus
	\begin{equation*}
		\E_j[X_t^2]=\int_0^t\exp\Bigl(2\int_s^t b_j(\tau)\diff\tau\Bigr)\diff s
		\le\int_0^t e^{2R(t-s)}\diff s=\frac{e^{2Rt}-1}{2R}.
	\end{equation*}
	Hence, by \eqref{eq:Xnorm} and $\int_0^T\frac{e^{2Rt}-1}{2R}\diff t =\frac{e^{2RT}-1-2RT}{4R^2}$,
	$\Vert a\Vert_\infty\le\Vert b_1\Vert_\infty+\Vert b_0\Vert_\infty\le2R$ yields
	\begin{equation}\label{eq:aXbound}
		\Vert a\Vert_{X,j}^2
		\;\le\;e^{2RT}-1-2RT \le 2\,(RT)^2\,e^{2RT},
	\end{equation}
	whence
	\begin{equation*}
	\delta_{n,j}\le 48\bigl(2\sqrt2\,K_s+8K_s^2+8\sqrt2\,K_s RT e^{RT}\bigr)=:\bar\delta_{s,R,T}.
	\end{equation*}
	In particular $\varepsilon_j\le\bar\delta_{s,R,T}(\log n)v_n\le 1/16$ for $n\ge N_{s,R,T,K_s}$.

	\begin{itemize}
		\item When $D_0 \lesssim \bigl(\log n \bigr)^{\frac{3s+1}{2s+1}}\; n^{ - \frac{s}{2s+1}} = (\log n)v_n$, we have
		\begin{equation*}
			\frac{C_{\mathrm{margin}} (\rho)\,\varepsilon_j}{D} \le \frac{C_{\mathrm{margin}} (\rho_0)\,\varepsilon_j}{D_0}
		\end{equation*}
		by \eqref{eq:Cmargindef} and we shall use the trivial margin bound $1$:
		\begin{align*}
			\eqref{eq:cond-decomp} \le  2\varepsilon_j + \frac{9+8C_s}{n} \lesssim \bigl(\log n \bigr)^{\frac{3s+1}{2s+1}}n^{ - \frac{s}{2s+1}} = (\log n) v_n,
		\end{align*}
		since $\varepsilon_j \le \bar\delta_{s,R,T}(\log n)v_n$ for $b_1,\ b_0 \in \Sigma(s,R)$.
		Eventually, the constant only depends on $s, R, \rho_0$, $T$, $K_s$ and $C_s$.
		\item When $D \gtrsim \bigl(\log n \bigr)^{\frac{3s+1}{2s+1}}\; n^{ - \frac{s}{2s+1}}$ (with the implied constant chosen so that $D\ge64\,(\varepsilon_0\vee\varepsilon_1)$), the margin bound applies and
		\begin{align*}
			\eqref{eq:cond-decomp} \le \frac{2\,C_{\mathrm{margin}}\,\varepsilon_j^2}{D}+ \frac{9+8C_s}{n}
			\lesssim \frac{\bigl(\log n \bigr)^{2} \, v_n^2}{D},
		\end{align*}
		where we use $\varepsilon_j^2\lesssim(\log n)^{2} v_n^{2}$ and that $n^{-1}\le n^{-2s/(2s+1)}$ is dominated by the proposed rate. The hidden constant in $\lesssim$ depends only on $s$, $R$, $\rho_0$, $T$, $K_s$ and $C_s$, and is uniform over the class $\mathcal{B}_{\mathrm{in}}(\rho_0, D_0) \cap \Sigma(s,R)^2$.

	\end{itemize}
	Averaging over $j\in\{0,1\}$ with weights $\frac12$ gives the theorem.

\end{proof}

\section{Appendix} \label{sec:appendix}

\begin{lem}[Example~\ref{ex:countermargin}] \label{lem:countermargin-reduct}
	In the setting of Example~\ref{ex:countermargin}, the log-likelihood ratio satisfies~\eqref{eq:ellrank2} with $\lambda_+, \lambda_- >0$ and independent standard normal variables $\xi_1$, $\xi_2$.
\end{lem}

\begin{proof}
	From~\eqref{eq:llr} and $b_1^2-b_0^2=(b_1-b_0)(b_1+b_0)=-2a$,
	\begin{equation*}
		Z = \ell(X; b_1, b_0)=\int_0^1 a(t) X_t\,\diff X_t+\int_0^1 a(t)X_t^2\,\diff t.
	\end{equation*}
	Since $a$ is of bounded variation and is right-continuous, It\^o's formula applied to $a(t)X_t^2$ with $X_0=0$ gives
	\begin{equation*}
		a(1)X_1^2=\int_0^1 X_t^2\,\diff a(t)+2\int_0^1 a(t)X_t\,\diff X_t+\int_0^1 a(t) \,\diff t ,
	\end{equation*}
	where $\diff a(t)= 2a(t)\,\diff t-c(1+e)\,\delta_{1/2}$ as a measure.
	Solving for the stochastic integral and using $\int_0^1 a =0$, we have
	\begin{equation*}
		\int_0^1 a(t)X_t\,\diff X_t
		=\frac12\Bigl(a(1)X_1^2-2\int_0^1 a(t) X_t^2\,\diff t+c(1+e)X_{1/2}^2\Bigr).
	\end{equation*}
	Inserting this into $Z$ cancels the Lebesgue integral, and $a(1)=-ce$ yields $Z=\frac c2\big[(1+e)X_{1/2}^2-eX_1^2\big]$.
	Conditionally on the class, $z:= (X_{1/2}, X_1)^T$ is a centred Gaussian vector with covariance $\Sigma$, and $Z=z^T M z$ with $M:=\frac{c}{2} \operatorname{diag}(1+e,-e)$.
	Write $z = \Sigma^{1/2} \xi$ with $\xi \sim N(0,I_2)$ and diagonalise the symmetric matrix $\Sigma^{1/2}M\Sigma^{1/2}$.
	The determinant $\det M\det\Sigma=-\tfrac{c^2}{4}e(1+e)\det\Sigma$ is negative.
	This gives one positive and one negative eigenvalue and hence~\eqref{eq:ellrank2}.
\end{proof}

\begin{lem} \label{lem:countermargin}
	Let $\xi_1$, $\xi_2$ be independent standard normal random variables, $\lambda_+$, $\lambda_- >0$ and $Z = \lambda_+ \xi_1^2 - \lambda_- \xi_2^2$.
	Then there exist constants $c > 0$ and $t_0 \in (0,1)$ depending only on $\lambda_+, \lambda_-$, such that
	\begin{equation*}
		\prob \left( \left| Z \right| \le t \right) \ge c \, t \log \left(1/t \right), \qquad \text{for all } 0 < t \le t_0.
	\end{equation*}
\end{lem}
\begin{proof}
	Set $X = \xi_1^2$, $U=\lambda_- \xi_2^2/ \lambda_+$ and $s=t/ \lambda_+$.
	We have
	\begin{align*}
		\prob(|Z|\le t)= \prob(|X-U| \le s),
	\end{align*}
	with $X$ and $U$ independent.
    The density of $X$ is $f(x)=(2\pi x)^{-1/2}e^{-x/2}$, hence
	\begin{equation*}
		f(x)\ \ge\ c_1\,x^{-1/2}\quad\text{for }0<x\le 2,\qquad c_1:=(2\pi)^{-1/2}e^{-1}.
	\end{equation*}
	The density of $U$ is $g(u)=\frac{\lambda_+}{\lambda_-}f\!\left(\frac{\lambda_+ u}{\lambda_-}\right)$, hence
	\begin{equation*}
		g(u)\ \ge\ c_2\,u^{-1/2}\quad\text{for }0<u\le 1, \qquad c_2:=\left( \frac{\lambda_+}{\lambda_-}\right)^{1/2} (2\pi)^{-1/2} e^{-\lambda_+/(2\lambda_-)}.
	\end{equation*}
	Note that $c_2>0$ depends only on $\lambda_+, \lambda_-$.

	Assume $t\le \lambda_+$, i.e.\ $s\le 1$. For $s\le u\le 1$ we have $[u-s,u+s]\subset[0,2]$, so
	\begin{equation*}
		P(|X-u|\le s)=\int_{u-s}^{u+s}f(x)\,dx
	\ \ge\ c_1\int_{u-s}^{u+s}x^{-1/2}\,dx
	=2c_1\bigl(\sqrt{u+s}-\sqrt{u-s}\bigr).
	\end{equation*}
	Using $\sqrt{u+s}-\sqrt{u-s}=\dfrac{2s}{\sqrt{u+s}+\sqrt{u-s}}$ and $\sqrt{u+s}+\sqrt{u-s}\le 2\sqrt{2u}$, we obtain
	\begin{equation*}
		P(|X-u|\le s)\ \ge\ \frac{\sqrt2\,c_1\,s}{\sqrt u}
	\end{equation*}
	with $s\le u\le 1$.

	By independence,
	\begin{equation*}
		P(|X-U|\le s)=\int_0^\infty P(|X-u|\le s)\,g(u)\,du
	\ \ge\ \int_s^1\frac{\sqrt2\,c_1 s}{\sqrt u}\cdot c_2u^{-1/2}\,du
	=\sqrt2\,c_1c_2\,s\log\frac{1}{s} .
	\end{equation*}
	Since $s=t/\lambda_+$, we have $\log(1/s)=\log(1/t)+\log \lambda_+$.
	Choose $t_0<\min(1,\lambda_+)$ small enough that $\log(1/s) \ge \tfrac12\log(1/t)$ for all $t\le t_0$ (e.g.\ $t_0\le \lambda_+^2$ if $\lambda_+<1$; if $\lambda_+\ge1$ this is automatically satisfied).
	Then
	\begin{equation*}
		P(|Z|\le t)\ \ge\ \frac{\sqrt2\,c_1c_2}{2\lambda_+}\;t\log\frac{1}{t},\qquad 0<t\le t_0.
	\end{equation*}
\end{proof}

\begin{lem} \label{lem:counterrho}
	In the setting of Example~\ref{ex:counterunbounded}, that is
	\begin{equation*}
    L:=\log\frac1\tau\ (\ge784),\qquad c:=\frac{1}{1+L},\qquad
    a(t):=\frac{c}{t\vee\tau},\qquad
    b_1=\frac{a-2}{2},\quad b_0=-\frac{a+2}{2}.
    \end{equation*}
	For every $0 < \tau \le e^{-784}$ such that $\varepsilon_{\tau} :=7/ 4\sqrt L \le 1/16$, we have
	\begin{enumerate}[label=(\roman*)]
		\item Assumption~\ref{ass:drifts} holds in regime $(+)$ with $\psi = a' - 2a \le -3c$ a.e., $c_{\psi} = 3c$,
		$C_\psi=c\tau^{-2}+2c\tau^{-1}$, $c_\omega\ge e^{-1/2}$,
      	$C_\omega\le e^{3/2}$, and $D=\int_0^1a=1$;
		in particular, $\frac{\tau^{-2}+2\tau^{-1}}{3} \le \rho\le e^8\tau^{-2}$.
\item Conditionally on $Y=0$, with $Z:=\ell(X; b_1, b_0)$ and $\bar Z:=2Z+D$ (see Lemma~\ref{lem:shift}),
      \begin{equation*}
          \E_0[\bar Z] \in \big[1-\tfrac1L,\,1\big],\qquad
          \operatorname{Var_0}(\bar Z)\le\frac{14}{L}.
      \end{equation*}
\item $\prob_{b_0}\bigl(|Z|\le7/\sqrt L\bigr)\ge3/4$, and
      consequently~\eqref{eq:rho-margin} holds.
	\end{enumerate}
\end{lem}

\begin{proof}
\emph{(i)} We have $b_1-b_0=a$ and $b_0+b_1=-2$, so $\psi=a'-2a$.
On $[0,\tau)$, $a\equiv c/\tau$ and $\psi=-2c/\tau\le-3c$ because $\tau\le2/3$.
On $[\tau,1]$, $a=c/t$ and $\psi=-c/t^2-2c/t\le-3c$, with $|\psi|$ maximal at $t=\tau$.
Thus, $a$ is Lipschitz, nonnegative, $\psi\in L^\infty$ and regime $(+)$ holds with the stated $c_\psi,C_\psi$.
Further $D=\int_0^\tau\frac c\tau\,\diff t+\int_\tau^1\frac ct\,\diff t=c(1+L)=1$.

For the weight functions, $\int_0^ta\in[0,1]$ gives $\int_0^tb_1=\tfrac12\int_0^ta-t\in[-1,\tfrac12]$ and $\int_0^tb_0=-\tfrac12\int_0^ta-t\in[-\tfrac32,0]$,
hence $\omega_1\in[e^{-1/2},e]$, $\omega_0\in[1,e^{3/2}]$.
Finally, we have $\rho=\frac{C_\psi}{c_\psi}\bigl(\frac{C_\omega}{c_\omega}\bigr)^4 \le\frac{\tau^{-2}+2\tau^{-1}}{3}\,e^8\le e^8\tau^{-2}$.
Also, we have $\rho \ge \frac{C_\psi}{c_\psi}$ since $\bigl(\frac{C_\omega}{c_\omega}\bigr)^4 \ge 1$.

\emph{(ii)} Under $Y=0$, $\diff X_t=b_0X_t\,\diff t+\diff W_t$ and $b_1^2-b_0^2=(b_1-b_0)(b_1+b_0)=-2a$, so by the log-likelihood ratio in~\eqref{eq:llr}
\begin{equation*}
	\begin{aligned}
	Z & =\int_0^1a(t) X_t\,\diff X_t+\int_0^1a(t)X_t^2\,\diff t
    =\int_0^1a(b_0+1)(t)X_t^2\,\diff t+\int_0^1a(t)X_t\,\diff W_t \\
	& =-\frac12\int_0^1a(t)^2 X_t^2\,\diff t+\int_0^1a(t)X_t\,\diff W_t ,
	\end{aligned}
\end{equation*}
using $b_0+1=-a/2$.
Hence $\bar Z =1-\int_0^1a(t)^2X_t^2\,\diff t+2\int_0^1a(t)X_t\,\diff W_t$ by Lemma~\ref{lem:shift} and $\E[\bar Z_0]=1-\Vert a\Vert_{X,0}^2$.
Since $b_0\le-1$,
$X_t=\int_0^te^{\int_s^tb_0}\diff W_s$ implies $\E_0[X_t^2]\le t$.
Thus,
\begin{equation*}
    \Vert a\Vert_{X,0}^2\le\int_0^1a(t)^2t\,\diff t
    =\int_0^\tau\frac{c^2}{\tau^2}t\,\diff t+\int_\tau^1\frac{c^2}{t}\,\diff t
    =c^2\bigl(\tfrac12+L\bigr)\le c^2(1+L)=c\le\frac1L .
\end{equation*}
For the variance, $\operatorname{Var}_0(\bar Z)\le2\operatorname{Var}_0\bigl(\int a^2X^2\bigr)+8\,\E_0\bigl(\int aX\,\diff W\bigr)^2$.
The second term equals $8\Vert a\Vert_{X,0}^2\le8/L$. For the first, by Cauchy--Schwarz and $\E[X_t^4]=3(\E[X_t^2])^2$,
\begin{equation*}
    \E\Bigl(\int_0^1a^2X_t^2\,\diff t\Bigr)^2
    =\iint a(s)^2a(t)^2\,\E[X_s^2X_t^2]\,\diff s\,\diff t
    \le3\Bigl(\int_0^1a^2\,\E[X_t^2]\,\diff t\Bigr)^2=3\Vert a\Vert_{X,0}^4\le\frac3{L^2},
\end{equation*}
so $\operatorname{Var}_0(\bar Z)\le6/L^2+8/L\le14/L$ for $L\ge1$.

\emph{(iii)} Since $\{|Z|\le7/\sqrt L\} = \{|\bar Z-1|\le14/\sqrt L \}$, and $|1-\E\bar Z|\le1/L\le1/\sqrt L$,
Chebyshev's inequality gives
\begin{equation*}
    \prob_{b_0}\Bigl(|\bar Z-1|>\frac{14}{\sqrt L}\Bigr)
    \le\prob_{b_0}\Bigl(|\bar Z-\E_0\bar Z|>\frac{13}{\sqrt L}\Bigr)
    \le\frac{14/L}{169/L}<\frac14.
\end{equation*}
As $\eta=(1+e^{-Z})^{-1}$ satisfies $|\eta-\tfrac12|\le\tfrac14|Z|$,
$\{|Z|\le7/\sqrt L\}\subset\{|\eta-\tfrac12|\le7/(4\sqrt L)\}$.
With $Y \sim $ Ber(1/2),
\begin{equation*}
	\prob_X(|\eta-\frac12|\le\varepsilon_\tau)\ge\frac12\prob_{b_0}(|Z|\le7/\sqrt L)\ge\frac38.
\end{equation*}
The restriction $\tau\le e^{-784}$ is exactly $\varepsilon_\tau\le1/16$.
\end{proof}

\begin{lem} \label{lem:exp}
	Let $Z$ be a continuous random variable. For any $0 < \varepsilon \le 1/2$,
	\begin{align*}
		\prob_Z \left( \left\vert \exp(Z) -1 \right\vert \le \varepsilon \right) \le \prob_Z \left( \left\vert Z \right\vert \le 2 \varepsilon  \right).
	\end{align*}
	For any $0 < \varepsilon \le 1/4$,
	\begin{align*}
		\prob_Z \left( \left\vert \exp(Z) -1 \right\vert > \varepsilon \right) \le \prob_Z \left( \left\vert Z \right\vert > \frac{\varepsilon}{2}  \right).
	\end{align*}
\end{lem}

\begin{proof}
	Using the inequalities $\ln (1+\varepsilon) \le \varepsilon$ and $\ln (1- \varepsilon) \ge -2 \varepsilon$ when $\varepsilon \le 1/2$, the first inequality holds.
	The second inequality holds, since $\ln (1+\varepsilon) \ge \varepsilon/2$ and $\ln (1- \varepsilon) \le - \varepsilon$ for all $0 < \varepsilon \le 1/4$.
\end{proof}

\begin{prop}[Chaos expansion \cite{nourdin2012}] \label{prop-chaos}
	Every $F \in L^2(\Omega)$ can be expanded as
	\begin{align*}
		F = \E[F] + \sum_{p=1}^{\infty} I_p(g_p)
	\end{align*}
	for some unique collection of symmetric kernels $g_p \in L^2([0,T]^p)$, $p \ge 1$, where
	\begin{align*}
		I_p (g_p) =
		p! \int_0^T \int_0^{t_1} \cdots \int_0^{t_{p-1}} g_p(t_1, \cdots, t_p) \diff W_{t_p} \cdots \diff W_{t_1}.
	\end{align*}
	Moreover, if $F \in \text{Dom}(D^n)$ for some $n \ge 1$, then $g_p = \frac{1}{p!} \E [ D^p F ]$ for all $p \le n$.
\end{prop}

\begin{lem}[Spectral comparison] \label{lem:congruence}
Let $H$ be a separable Hilbert space, $K$ a compact self-adjoint operator on $H$ and $S$ a bounded self-adjoint operator such that
\begin{equation*}
	c \Vert x \Vert\le \Vert S x \Vert\le C \Vert x \Vert, \qquad x \in H,
\end{equation*}
for constants $0 < c \le C < \infty$. Set $A := S K S$ and denote by $\lambda^+_1(\cdot) \ge \lambda^+_2(\cdot) \ge \cdots > 0$ the positive eigenvalues and by $\lambda^-_1(\cdot) \ge \lambda^-_2(\cdot) \ge \cdots > 0$ the absolute values of the negative eigenvalues, each counted with multiplicity.
Then
\begin{enumerate}
	\item[\textup{(i)}] $A$ and $K$ have the same (possibly infinite) number of positive eigenvalues and the same number of negative eigenvalues;
	\item[\textup{(ii)}] $c^2 \lambda^{\pm}_k (K) \le \lambda^{\pm}_k(A) \le C^2 \lambda^{\pm}_k(K)$ for every admissible $k$.
	\item[\textup{(iii)}] for the eigenvalues ordered by magnitude, $\vert \lambda_0(\cdot) \vert \ge \vert \lambda_1(\cdot) \vert \ge \cdots$ counted with multiplicity (setting $\lambda_k(\cdot) := 0$ if fewer than $k+1$ nonzero eigenvalues exist),
	\begin{equation*}
		c^2 \vert \lambda_k(K) \vert \le \vert \lambda_k (A) \vert \le C^2 \vert \lambda_k(K)\vert, \qquad k \ge 0.
	\end{equation*}
\end{enumerate}
\end{lem}
\begin{proof}
Note first that $A=SKS$ is compact and self-adjoint, and that $S$ is bijective.
Its inverse is again self-adjoint. Substituting $x=S^{-1}z$ into the bounds for $S$ gives $C^{-1}\Vert z\Vert\le\Vert S^{-1}z\Vert\le c^{-1}\Vert z\Vert$.

Suppose $K$ has at least $k$ positive eigenvalues. Let $V_k$ be the span of orthonormal eigenfunctions of $K$ for $\lambda^+_1(K),\dots,\lambda^+_k(K)$, and set $W:=S^{-1}V_k$, which is a $k$-dimensional subspace.
Every unit function $y\in W$ satisfies, with the unit function $\hat s:=Sy/\Vert Sy\Vert\in V_k$,
\begin{equation*}
	\langle Ay,y\rangle=\langle KSy,\,Sy\rangle=\Vert Sy\Vert^2\,\langle K\hat s,\hat s\rangle
	\ \ge\ c^2\,\lambda^+_k(K)\ >\ 0,
\end{equation*}
since $\langle K\hat s,\hat s\rangle\ge\lambda^+_k(K)$ for unit vectors of $V_k$ and $\Vert Sy\Vert^2\ge c^2$.
Multiplying by $\Vert Sy\Vert^2$ preserves the lower bound because $\lambda^+_k(K)$ is positive.

On the other hand, let $w_1,w_2,\dots$ be orthonormal eigenfunctions of $A$ for its positive eigenvalues $\lambda^+_1(A)\ge\lambda^+_2(A)\ge\cdots$.
Let $n_+$ be the number of positive eigenvalues, and put $q:=(k-1)\wedge n_+$. The projection of $W$ onto $\operatorname{span}(w_1,\dots,w_q)$ has rank at most $q<k=\dim W$. Thus, there is a unit vector $y\in W$ orthogonal to $w_1,\dots,w_q$.
Write $y=\sum_{q<j\le n_+}\langle y,w_j\rangle w_j+y_-$ with $y_-$ orthogonal to every $w_j$. By the spectral theorem the closed span of $\{w_j\}$ reduces $A$, so the cross terms vanish and $\langle Ay_-,y_-\rangle\le0$, whence
\begin{equation*}
	\langle Ay,y\rangle=\sum_{q<j\le n_+}\lambda^+_j(A)\,\vert\langle y,w_j\rangle\vert^2+\langle Ay_-,y_-\rangle
	\ \le\
	\begin{cases}
		0, & n_+ <k \\%
		\lambda^+_k(A), & n_+ \ge k.
	\end{cases}
\end{equation*}
Since $\langle Ay,y\rangle\ge c^2\lambda^+_k(K)>0$, the case $n_+ < k$ is impossible: $A$ has at least $k$ positive eigenvalues, and $\lambda^+_k(A)\ge c^2\lambda^+_k(K)$.

Exchanging the roles of the two operators ($K=S^{-1}AS^{-1}$) with the self-adjoint operator $S^{-1}$ (bounded below by $C^{-1}$), shows in the same way that if $A$ has at least $k$ positive eigenvalues, then so does $K$, with $\lambda^+_k(K)\ge C^{-2}\lambda^+_k(A)$.
Together, $A$ and $K$ have the same number of positive eigenvalues and $c^2\lambda^+_k(K)\le\lambda^+_k(A)\le C^2\lambda^+_k(K)$ for every admissible $k$. Applying this to $-K$ and $-A=S(-K)S$ yields the assertions for the negative eigenvalues. This proves (i) and (ii).

For (iii), let $N_{B}(x) := \#\{ k : \lambda^+_k(B) \ge x\} + \#\{k : \lambda^-_k(B) \ge x\}$ count the eigenvalues of a compact self-adjoint operator $B$ that have magnitude at least $x$, for $x > 0$.
It is finite because the eigenvalues accumulate only at $0$; by the definition of the magnitude ordering, $\vert \lambda_k(B) \vert \ge x$ if and only if $N_B(x) \ge k+1$.
If $\lambda_j^{\pm}(K) \ge x$, then $\lambda_j^{\pm}(A) \ge c^2 x$ by (ii).
Summing the two sign counts gives $N_A(c^2x) \ge N_K(x)$ for all $x>0$.
For $k$ with $\lambda_k(K) \neq 0$, taking $x=\vert\lambda_k(K)\vert$ yields $N_A(c^2\vert\lambda_k(K)\vert)\ge k+1$, hence $\vert\lambda_k(A)\vert\ge c^2\vert\lambda_k(K)\vert$.
For $\lambda_k(K)=0$ this bound is trivial. Exchanging the roles of $A$ and $K$ as above gives $N_K(C^{-2}x)\ge N_A(x)$ for all $x>0$, hence $\vert\lambda_k(A)\vert\le C^2\vert\lambda_k(K)\vert$.
In particular $\lambda_k(K)=0$ forces $\lambda_k(A)=0$. This proves (iii).
\end{proof}

\begin{lem}[Proof of Lemma~\ref{lem:negspec}]
Let $K=R-V$ as in \eqref{eq:Bdef}, with $R=T F(T) \,\varphi\varphi^*$ rank-one positive with $TF(T)>0$, $\varphi:=\mathbf 1/\sqrt T$, $\Vert\varphi\Vert=1$ and $V\succeq0$ compact.
Let $\mu_0\ge\mu_1\ge\cdots\ge0$ be the eigenvalues of $V$, and let $\nu_1(K) \ge \nu_2(K) \ge \cdots$ denote the magnitudes of the negative eigenvalues of $K$, counted with multiplicity (with $\nu_j(K) := 0$ beyond their number).
$\lambda_0(K)\ge\lambda_1(K)\ge \cdots \ge \lambda_{n_+}(K) \ge 0$
\begin{enumerate}
  \item[\textup{(i)}] $\lambda_1(K)\le0$; i.e.\ $K$ has at most one positive
    eigenvalue.
  \item[\textup{(ii)}] The negative eigenvalues of $K$ interlace those of $V$:
    \begin{equation}
      \mu_{k+1}\;\le\; \nu_{k+1}(K)\;\le\;\mu_k,
      \qquad k\ge0,
    \end{equation}
	whether $K$ has one or no positive eigenvalue.
 	In particular $\nu_j (K)\to 0$ and inherits the decay rate of $\mu_j$.
\end{enumerate}
\end{lem}

\begin{proof}
Recall from \eqref{eq:Bdef} that $K=R-V$ on $L^2([0,T])$, where $R=TF(T)\,\varphi\varphi^*$ with $\varphi=\mathbf 1/\sqrt T$, $\Vert\varphi\Vert=1$,
$TF(T)>0$, and $V\succeq0$ is compact with eigenvalues $\mu_0\ge\mu_1\ge\cdots\ge0$.
For a compact self-adjoint operator $B$ we write $\lambda^{\downarrow}_0(B)\ge\lambda^{\downarrow}_1(B)\ge\cdots$ for its eigenvalues in decreasing order and use the Courant--Fischer principle in $0$-indexed form,
\begin{equation*}
  \lambda^{\downarrow}_k(B)
  =\min_{\operatorname{codim}S=k}\ \max_{x\in S,\,\Vert x\Vert=1}\langle Bx,x\rangle
  =\max_{\dim S=k+1}\ \min_{x\in S,\,\Vert x\Vert=1}\langle Bx,x\rangle .
\end{equation*}

\emph{(i)} For $x\in\varphi^\perp$ we have $Rx=0$, hence $\langle Kx,x\rangle=-\langle Vx,x\rangle\le0$.
Taking $S=\varphi^\perp$ ($\operatorname{codim}S=1$) in the min--max form with $k=1$,
\begin{equation*}
  \lambda^{\downarrow}_1(K)\le\max_{x\in\varphi^\perp,\,\Vert x\Vert=1}\langle Kx,x\rangle\le0 ,
\end{equation*}
so $K$ has at most one positive eigenvalue.

\emph{(ii)} Put $\widetilde K:=-K=V-R$, a rank-one downward perturbation of $V\succeq0$.
Since $R\succeq0$ we have $\widetilde K\preceq V$, so $\lambda^{\downarrow}_k(\widetilde K)\le\mu_k$ for all $k\ge0$.
For the lower bound let $U$ be the span of eigenvectors of $V$ for $\mu_0,\dots,\mu_{k+1}$, so $\dim U=k+2$ and $W:=U\cap\varphi^\perp$ has $\dim W\ge k+1$.
For $x\in W$ we have $Rx=0$ and $\langle Vx,x\rangle\ge\mu_{k+1}\Vert x\Vert^2$, hence $\langle\widetilde Kx,x\rangle\ge\mu_{k+1}\Vert x\Vert^2$.
The max--min form with $\dim S=k+1$ gives $\lambda^{\downarrow}_k(\widetilde K)\ge\mu_{k+1}$.
Thus
\begin{equation*}
  \mu_{k+1}\le\lambda^{\downarrow}_k(\widetilde K)\le\mu_k,\qquad k\ge0 .
\end{equation*}
By (i), $\widetilde K$ has at most one negative eigenvalue, lying below its nonnegative spectrum.
Hence, whether $K$ has one or no positive eigenvalue, the values $\lambda^{\downarrow}_0(\widetilde K)\ge\lambda^{\downarrow}_1(\widetilde K)\ge\cdots$ list exactly the magnitudes of the negative eigenvalues of $K$ in decreasing order.
That is, $\nu_j(K)=\lambda^{\downarrow}_{j-1}(\widetilde K)$ for $j\ge1$.
Substituting into the display above yields \eqref{eq:interlace}.
\end{proof}

\begin{lem} \label{lem-chidensity}
	Let $N$ be a standard normal variable in $\R$ and let $Z$ be a random variable in $\R_+$ with a density function independent of $N$. Suppose that $\lambda_1 \ge \lambda_2 > 0$. It holds that
	\begin{align*}
		f_{\lambda_2 N^2 + Z} (x) \le \sqrt{\frac{\lambda_1}{\lambda_2}} f_{\lambda_1 N^2 + Z} (x)
	\end{align*}
	for all $x > 0$.
\end{lem}

\begin{proof}
	Recall that a weighted $\chi^2$ variable of order 1 with weight $\sigma^2$ has the density function
	\begin{align*}
		f_{\sigma^2 \chi^2(1)} (x) = \frac{1}{ \sqrt{2 \pi \sigma^2 x } }  e^{-\frac{x}{2\sigma^2}} \quad \text{for } x > 0.
	\end{align*}
	It is obvious that if $\lambda_1 \ge \lambda_2 > 0$,
	\begin{align*}
		f_{\lambda_2 N^2}(x) = \sqrt{\frac{\lambda_1}{\lambda_2}} e^{-\frac{x}{2} \left(\frac{1}{\lambda_2} - \frac{1}{\lambda_1} \right)} f_{\lambda_1 N^2}(x) \le \sqrt{\frac{\lambda_1}{\lambda_2}} f_{\lambda_1 N^2}(x)
	\end{align*}
	for all $x \in (0, \infty)$. Thus, using the convolution of density functions, we have
	\begin{align*}
		f_{\lambda_2 N^2 + Z} (x) = \int_{-\infty}^{\infty} f_{\lambda_2 N^2} (x-y) f_{Z} (y) \diff y \le \sqrt{\frac{\lambda_1}{\lambda_2}} \int_{-\infty}^{\infty} f_{\lambda_1 N^2} (x-y) f_{Z} (y) \diff y = \sqrt{\frac{\lambda_1}{\lambda_2}} f_{\lambda_1 N^2 + Z}(x).
	\end{align*}
\end{proof}

\begin{lem} \label{lem-chidensityrev}
	Let $N$ be a standard normal variable in $\R$. Suppose that $\lambda_1$, $\lambda_2 > 0$. It holds that
	\begin{align*}
		f_{\lambda_1 N^2} (x) \le \sqrt{\frac{\lambda_2}{\lambda_1}} e^{\frac{x}{2\lambda_2}} f_{\lambda_2 N^2} (x)
	\end{align*}
	for all $x > 0$.
\end{lem}

\begin{proof}
	If $\lambda_1$, $\lambda_2 > 0$,
	\begin{align*}
		f_{\lambda_1 N^2}(x) = \sqrt{\frac{\lambda_2}{\lambda_1}} e^{-\frac{x}{2} \left(\frac{1}{\lambda_1} - \frac{1}{\lambda_2} \right)} f_{\lambda_2 N^2}(x) \le \sqrt{\frac{\lambda_2}{\lambda_1}} e^{\frac{x}{2\lambda_2} } f_{\lambda_2 N^2}(x).
	\end{align*}
\end{proof}

\begin{lem}[Ulyanov factor]\label{lem:ulyanov}
	Assume $0<c_F\le-F'\le C_F$ and $F(T)\ge 0$, and let $\theta_k=\lambda_k(A)$, $k\ge1$, be the positive eigenvalues of $A$ other than $\lambda_0(A)$ (Corollary~\ref{cor:posregime}). Recall the definition in \eqref{eq:rho}
	\begin{equation*}
		\rho:= \frac{C_\psi}{c_{\psi}} \left( \frac{C_{\omega}}{c_{\omega}}\right)^4=\frac{C_\omega^2\,C_F}{c_\omega^2\,c_F}\;\ge1 .
	\end{equation*}
	Then
	\begin{equation}\label{eq:thetaratios}
		\frac{\theta_1}{\theta_j}\le\rho\,(j+\frac12)^2\quad (j\ge1),
		\qquad
		\frac{\theta_k}{\theta_1}\le\frac{9\rho}{4}\,k^{-2}\quad (k\ge2).
	\end{equation}
	Consequently, for every integer $m\ge2$ with $(m+1)^2>\frac94\rho$,
	\begin{equation}\label{eq:Pimbound}
		\Pi_m\;\le\;\Bigl(\rho^{\,m-1}\prod_{j=2}^{m}(j+\frac12)^2\Bigr)^{1/2}
		\exp\!\left(\frac{9\rho}{8}\cdot\frac{\sum_{k>m}k^{-2}}{1-\frac{9\rho}{4(m+1)^2}}\right)<\infty .
	\end{equation}
	In particular, if $\rho<4$ then $m=2$ is admissible and
	\begin{equation}\label{eq:Pi2bound}
		\Pi=\Pi_2\le\frac{5\sqrt\rho}{2}\,
		\exp\!\left(\frac{9\rho}{8}\cdot\frac{\pi^2/6-5/4}{1-\rho/4}\right) .
	\end{equation}
\end{lem}

\begin{proof}
	By Corollary~\ref{cor:posregime}, for $j\ge1$
	\begin{equation*}
		\frac{c_\omega^2c_FT^2}{((j+\frac12)\pi)^2}\le\theta_j\le\frac{C_\omega^2C_FT^2}{(j\pi)^2} .
	\end{equation*}
	Hence
	\begin{equation*}
	\frac{\theta_1}{\theta_j} \le\frac{C_\omega^2C_F}{\pi^2}\cdot\frac{((j+1/2)\pi)^2}{c_\omega^2c_FT^2}\cdot\frac{T^2}{1}=\rho(j+\frac12)^2
	\end{equation*}
	and
	\begin{equation*}
		\frac{\theta_k }{\theta_1} \le \frac{C_\omega^2C_FT^2/(k\pi)^2}{c_\omega^2c_FT^2/((3/2)\pi)^2}=\frac94\rho k^{-2},
	\end{equation*}
	which is \eqref{eq:thetaratios}.

	For $k>m$ we get $\theta_k/\theta_1\le\frac94\rho(m+1)^{-2} < 1$. Also we have $-\log(1-x)\le x/(1-q)$ for $0\le x\le q < 1$. Letting $q = \frac94\rho(m+1)^{-2}$, we obtain
	\begin{equation*}
		\prod_{k>m}\Bigl(1-\frac{\theta_k}{\theta_1}\Bigr)^{-1/2}
		\le\exp\Bigl(\frac{1}{2(1-q)}\sum_{k>m}\frac{9\rho}{4}k^{-2}\Bigr).
	\end{equation*}
	\eqref{eq:Pimbound} follows from \eqref{eq:Pim} and the first part of \eqref{eq:thetaratios}.
	For $\rho<4$ and $m=2$: $q=\rho/4<1$, $\theta_1/\theta_2\le\frac{25}4\rho$, and $\sum_{k\ge3}k^{-2}=\pi^2/6-1-\frac14=\pi^2/6-\frac54$, giving \eqref{eq:Pi2bound}.
\end{proof}

\begin{lem}[Concentration of generalized $\chi^2$ variables] \label{lem:LM}
Let $(\lambda_j)_{j\ge0}$ be a real sequence with $\Vert\lambda\Vert_2^2:=\sum_j\lambda_j^2<\infty$
and let $(N_j)_{j\ge0}$ be independent standard normal variables.
Then $W:=\sum_j\lambda_j(N_j^2-1)$ converges in $L^2$ and for every $t>0$,
\begin{equation}\label{eq:LM}
	\prob\bigl(W\ge t\bigr)\vee\prob\bigl(W\le -t\bigr)
	\le\exp\!\left(-\frac{t^2}{4\Vert\lambda\Vert_2^2+4\Vert\lambda\Vert_\infty t}\right),
\end{equation}
where $\Vert\lambda\Vert_\infty:=\sup_j|\lambda_j|$.
In particular $\prob(|W|\ge t)\le2\exp\bigl(-t^2/(4\Vert\lambda\Vert_2^2+4\Vert\lambda\Vert_\infty t)\bigr)$.
\end{lem}
\begin{proof}
	See Bernstein-type inequalities for sub-gamma variables in Section 2.4 in \cite{boucheron2013}.
\end{proof}

\begin{lem}[Tail bound for the sieve MLE in inhomogeneous case] \label{lem:sieveMLE}
Let $b_i\in\Sigma(s,R)$, $d:= \lceil s \rceil$ and let $n$ be the number of paths of class $i$.
Let $\hat b_i$ be the sieve MLE with the degree at most $\lceil s\rceil -1$ on $\lceil (n/\log n )^{1/(2s+1)}\rceil$ partitions.
Set
\begin{equation}\label{eq:Kdef}
  K := \Bigl(\frac{C_\omega}{c_\omega}\Bigr)^{2}
  \left[\,(1+\sqrt3)\, \sqrt{C_{\mathrm{est},i}} \;\frac{R\,T^s}{(\lceil s\rceil-1)!}
  \;+\;4\sqrt{3+6 \lceil s\rceil \log 5}\,\right].
\end{equation}
Let
\begin{equation}\label{eq:Cs}
  C_s := 300\,\bigl(\lceil s\rceil+7\bigr).
\end{equation}
Then for both $j\in\{0,1\}$,
\begin{equation}\label{eq:inhomoresult}
  \prob\left( \bigl\Vert\hat b_i-b_i\bigr\Vert_{X,j}
    \ge K\,\Bigl(\tfrac{\log n}{n}\Bigr)^{\frac{s}{2s+1}}
    \;\Big|\; N_i=n \right)
  \le \frac{C_s}{n}.
\end{equation}
\end{lem}

\begin{proof}
	Let $d:=\lceil s\rceil$ and let $S_m\subset L^2([0,T])$ be the space of piecewise polynomials of degree at most $d-1$ on the partition of $[0,T]$ into $J:=\lceil(N/\log N)^{1/(2s+1)}\rceil$ intervals, so $m:=\dim S_m=dJ$.
	Write $\Vert\theta\Vert_N^2:=\frac1N\sum_{k=1}^N\int_0^T\theta(t)^2(X^{(k)}_t)^2\diff t$ with inner product $\langle\cdot,\cdot\rangle_N$,
	$\ell_N(\theta):=\frac1N\sum_k\bigl[\int_0^T\theta X^{(k)}\diff X^{(k)}-\frac12\int_0^T\theta^2(X^{(k)})^2\diff t\bigr]$ for the empirical log-likelihood,
	and $\nu_N(\theta):=\frac1N\sum_k\int_0^T\theta X^{(k)}\diff W^{(k)}$ for the empirical score.
	We have
	\begin{align*}
		\ell_N(\theta)=\langle b,\theta\rangle_N+\nu_N(\theta)-\frac12\Vert\theta\Vert_N^2.
	\end{align*}
	Let $(\varphi_{k,l})_{l\le d}$ be a $\langle\cdot,\cdot\rangle_{X,i}$-orthonormal basis of the polynomials on the $k$-th interval,
	and $B_k:=(\langle\varphi_{k,l},\varphi_{k,l'}\rangle_N)_{l,l'\le d}$ the corresponding block of the Gram matrix of $\langle\cdot,\cdot\rangle_N$, with $\E B_k=I_d$.
	Define
	\begin{align*}
	\Omega_n := \Bigl\{\ \tfrac12\Vert\theta\Vert_{X,i}^2\le\Vert\theta\Vert_n^2
    \le\tfrac32\Vert\theta\Vert_{X,i}^2 \quad\text{for all }\theta\in S_m\ \Bigr\}.
	\end{align*}
	Let $B_k$, $k = 1, \cdots, J$ be the block matrix of size $d$ in the Gram matrix. We have
	\begin{align*}
		\Omega_n = \Bigl\{ \max_{k \le J} \Vert B_k - I_d \Vert_{\mathrm{op}} \le \frac{1}{2} \Bigr\}.
	\end{align*}
	Using the concentration inequality in Lemma~\ref{lem:LM} with $t = 7/32$, we obtain that for a unit $v \in \R^d$
	\begin{align*}
		\prob \left(| v^T (B_k - I_d) v | \ge \frac{7}{32}\right) \le 2 e^{-\frac{n}{102}}.
	\end{align*}
	A net result in $\R^d$ gives that
	\begin{align*}
		\prob (\Omega_n^c) \le 2 J 9^d e^{-\frac{n}{102}} \le \frac{1}{n}
	\end{align*}
	for $n \ge N_s : = 300\,\bigl(\lceil s\rceil+7\bigr) \ge 2400$.
	Let $\nu_n(\theta) := \tfrac{1}{n}\sum_k \int_0^T \theta(t) X_t^{(k)} \diff W_t^{(k)}$ be the empirical score functional.
	For fixed $\theta \in S_m$ with $\Vert\theta \Vert_{X,i}=1$, $n \nu_n(\theta)$ is a continuous martingale with quadratic variation $n \Vert\theta \Vert_n^2$.
	Using the Bernstein-type inequality for continuous martingales, we have
	\begin{align*}
		\prob \left( \nu_n (\theta) \ge x, \; \Vert\theta \Vert_n^2 \le \frac{3}{2} \right) \le e^{-\frac{n x^2}{3}}.
	\end{align*}
	Set $x^2 = 3(m \log 5 + \log n)/n$ and applying a net result of the $\Vert\cdot \Vert_{X,i}$-unit sphere in $S_m$, we obtain
	\begin{align*}
		\prob \left( \Bigl(\Omega_n \cap \left\{\sup_{\theta \in S_m, \Vert\theta \Vert_{X,i} \le 1} \nu_n(\theta) \le 2x \right\}\Bigr)^c\right) \le \frac{1}{n} + 5^m e^{-\frac{n x^2}{3}} \le \frac{2}{n}.
	\end{align*}
	Taylor expansion and the Hölder condition gives $\theta^* \in S_m$ with
	\begin{align*}
		\Vert b-\theta^*\Vert_n^2\le \bigl(R\,h^s/(d-1)!\bigr)^2\Vert\mathbf1\Vert_n^2\le\tfrac{3}{2} C_{\mathrm{est},i} \bigl(R\,h^s/(d-1)!\bigr)^2.
	\end{align*}
	on $\Omega_n$. Also, $\ell_n$ is strictly concave on $S_m$, so $\hat b$ is the unique maximizer and $\langle\hat b-b,\theta\rangle_n=\nu_n(\theta)$ for all $\theta\in S_m$. Using the equivalence between norms and definitions of $m$ and $J$, we have that
	\begin{align*}
		\left\{\Vert\hat{b}_i - b_i \Vert_{X,j} \ge K v_n \right\} \subset \Bigl(\Omega_n \cap \left\{\sup_{\theta \in S_m, \Vert\theta \Vert_{X,i} \le 1} \nu_n(\theta) \le 2x \right\}\Bigr)^c.
	\end{align*}
	The result concludes by taking $C_s = N_s$ for $n \le N_s$.
\end{proof}

\section*{Acknowlegdements}
This research has been partially funded by Deutsche Forschungsgemeinschaft (DFG) SFB1294 "Data Assimilation", Project A01, Project-ID. 318763901, and DFG CRC/TRR388 "Rough Analysis, Stochastic Dynamics and Related Fields", Project B01, Project-ID. 516748464.
AI-based tools were used to assist with checking the proofs and improving the readability of the manuscript.

\bibliographystyle{spr-ims-number}
\bibliography{ref}

\end{document}